\pdfoutput=1
\RequirePackage{scrlfile}
\PassOptionsToPackage{compress}{cleveref}
\documentclass[final,onefignum,onetabnum]{siamart220329}

\usepackage[T1]{fontenc}
\usepackage[utf8]{inputenc}

\usepackage{amssymb}
\usepackage{mathtools}
\usepackage{mathrsfs}
\usepackage{stmaryrd}
\usepackage{bm}
\usepackage{colortbl}
\usepackage{diagbox}

\newtheorem{Proposition}[theorem]{Proposition}
\newtheorem{Theorem}[theorem]{Theorem}
\newtheorem{Lemma}[theorem]{Lemma}
\newtheorem{Corollary}[theorem]{Corollary}
\newtheorem{Remark}[theorem]{Remark}
\newtheorem{Definition}[theorem]{Definition}
\crefname{Corollary}{Corollary}{Corollaries}

\makeatletter
\def\dashint{\pd@mint{-}}
\newcommand*{\pd@mint}[1]{%
  \pd@mint@l{#1}{}%
}
\newcommand*{\pd@mint@l}[2]{%
  \@ifnextchar\limits{%
    \pd@mint@l{#1}%
  }{%
    \@ifnextchar\nolimits{%
      \pd@mint@l{#1}%
    }{%
      \@ifnextchar\displaylimits{%
        \pd@mint@l{#1}%
      }{%
        \pd@mint@s{#2}{#1}%
      }%
    }%
  }%
}
\newcommand*{\pd@mint@s}[2]{%
  \@ifnextchar_{%
    \pd@mint@sub{#1}{#2}%
  }{%
    \@ifnextchar^{%
      \pd@mint@sup{#1}{#2}%
    }{%
      \pd@mint@{#1}{#2}{}{}%
    }%
  }%
}
\def\pd@mint@sub#1#2_#3{%
  \@ifnextchar^{%
    \pd@mint@sub@sup{#1}{#2}{#3}%
  }{%
    \pd@mint@{#1}{#2}{#3}{}%
  }%
}
\def\pd@mint@sup#1#2^#3{%
  \@ifnextchar_{%
    \pd@mint@sup@sub{#1}{#2}{#3}%
  }{%
    \pd@mint@{#1}{#2}{}{#3}%
  }%
}
\def\pd@mint@sub@sup#1#2#3^#4{%
  \pd@mint@{#1}{#2}{#3}{#4}%
}
\def\pd@mint@sup@sub#1#2#3_#4{%
  \pd@mint@{#1}{#2}{#4}{#3}%
}
\newcommand*{\pd@mint@}[4]{%
  \mathop{}%
  \mkern-\thinmuskip
  \mathchoice{%
    \pd@mint@@{#1}{#2}{#3}{#4}%
    \displaystyle\textstyle\scriptstyle
  }{%
    \pd@mint@@{#1}{#2}{#3}{#4}%
    \textstyle\scriptstyle\scriptstyle
  }{%
    \pd@mint@@{#1}{#2}{#3}{#4}%
    \scriptstyle\scriptscriptstyle\scriptscriptstyle
  }{%
    \pd@mint@@{#1}{#2}{#3}{#4}%
    \scriptscriptstyle\scriptscriptstyle\scriptscriptstyle
  }%
  \mkern-\thinmuskip
  \int#1%
  \ifx\\#3\\\else_{#3}\fi
  \ifx\\#4\\\else^{#4}\fi  
}
\newcommand*{\pd@mint@@}[7]{%
  \begingroup
  \sbox0{$#5\int\m@th$}%
  \sbox2{$#5\int_{}\m@th$}%
  \dimen2=\wd0 %
  \let\pd@mint@limits=#1\relax
  \ifx\pd@mint@limits\relax
  \sbox4{$#5\int_{\kern1sp}^{\kern1sp}\m@th$}%
  \ifdim\wd4>\wd2 %
  \let\pd@mint@limits=\nolimits
  \else
  \let\pd@mint@limits=\limits
  \fi
  \fi
  \ifx\pd@mint@limits\displaylimits
  \ifx#5\displaystyle
  \let\pd@mint@limits=\limits
  \fi
  \fi
  \ifx\pd@mint@limits\limits
  \sbox0{$#7#3\m@th$}%
  \sbox2{$#7#4\m@th$}%
  \ifdim\wd0>\dimen2 %
  \dimen2=\wd0 %
  \fi
  \ifdim\wd2>\dimen2 %
  \dimen2=\wd2 %
  \fi
  \fi
  \rlap{%
    $#5%
    \vcenter{%
      \hbox to\dimen2{%
        \hss
        $#6{#2}\m@th$%
        \hss
      }%
    }%
    $%
  }%
  \endgroup
}
\makeatother

\newcommand*\TblDgRateOne{%
  \begin{table}[H]
    \caption{Experimental order of convergence EOC$_\ell$, $\ell \in \{1,\ldots,6\}$,
             for the primal IIDG formulation, i.e., \eqref{eq:discrete_primal_DG} with $\tilde{\nabla}_{h_{\ell}}=\nabla_{h_{\ell}}$,
             with $\beta=1.01               $ and, thus, $\rho=1.0$.}
    \label{tbl:DG_rate_1.0}
    \vspace{1ex}
    \setlength\tabcolsep{8pt}
    \centering
    \begin{tabular}{|c||c|c||c||c|c|}
      \hline
      \cellcolor{lightgray}\diagbox[height=1.1\line,width=0.1275\dimexpr\linewidth]{\vspace{-0.5mm}\hspace*{-2mm}$h_\ell$}{\\[-5mm] $p$\hspace*{-2mm}}
      & \cellcolor{lightgray}1.5 & \cellcolor{lightgray}1.7 & \cellcolor{lightgray}2.0 & \cellcolor{lightgray}3.0 & \cellcolor{lightgray}4.5
      \\ \hline\hline
      \cellcolor{lightgray}$0.0834$ & $0.933$ & $0.931$ & $0.919$ & $0.955$ & $0.967$ \\ \hline
      \cellcolor{lightgray}$0.0417$ & $0.952$ & $0.955$ & $0.949$ & $0.976$ & $0.984$ \\ \hline
      \cellcolor{lightgray}$0.0208$ & $0.961$ & $0.961$ & $0.959$ & $0.979$ & $0.989$ \\ \hline
      \cellcolor{lightgray}$0.0104$ & $0.963$ & $0.963$ & $0.963$ & $0.981$ & $0.993$ \\ \hline
      \cellcolor{lightgray}$0.0052$ & $0.966$ & $0.965$ & $0.965$ & $0.981$ & $0.994$ \\ \hline
      \cellcolor{lightgray}$0.0026$ & $0.967$ & $0.967$ & $0.967$ & $0.982$ & $0.994$ \\ \hline
      \hline
      \cellcolor{lightgray}\small\textrm{Expected}
      & 1.0 & 1.0 & 1.0 & 1.0 & 1.0
      \\ \hline
    \end{tabular}
  \end{table}
}

\newcommand*\TblDgRateHalf{%
  \begin{table}[H]
    \caption{Experimental order of convergence EOC$_\ell$, $\ell \in \{1,\ldots,6\}$,
             for the primal IIDG formulation, i.e., \eqref{eq:discrete_primal_DG} with $\tilde{\nabla}_{h_{\ell}}=\nabla_{h_{\ell}}$,
             with $\beta=1.01 - \frac{1}{p} $ and, thus, $\rho=0.5$.}
    \label{tbl:DG_rate_0.5}
    \vspace{1ex}
    \setlength\tabcolsep{8pt}
    \centering
    \begin{tabular}{|c||c|c||c||c|c|}
      \hline
      \cellcolor{lightgray}\diagbox[height=1.1\line,width=0.1275\dimexpr\linewidth]{\vspace{-0.5mm}\hspace*{-2mm}$h_\ell$}{\\[-5mm] $p$\hspace*{-2mm}}
      & \cellcolor{lightgray}1.5 & \cellcolor{lightgray}1.7 & \cellcolor{lightgray}2.0 & \cellcolor{lightgray}3.0 & \cellcolor{lightgray}4.5
      \\ \hline\hline
      \cellcolor{lightgray}$0.0834$ & $0.672$ & $0.659$ & $0.662$ & $0.758$ & $0.863$ \\ \hline
      \cellcolor{lightgray}$0.0417$ & $0.631$ & $0.613$ & $0.616$ & $0.684$ & $0.818$ \\ \hline
      \cellcolor{lightgray}$0.0208$ & $0.585$ & $0.569$ & $0.573$ & $0.612$ & $0.762$ \\ \hline
      \cellcolor{lightgray}$0.0104$ & $0.553$ & $0.542$ & $0.544$ & $0.566$ & $0.667$ \\ \hline
      \cellcolor{lightgray}$0.0052$ & $0.533$ & $0.526$ & $0.528$ & $0.541$ & $0.600$ \\ \hline
      \cellcolor{lightgray}$0.0026$ & $0.522$ & $0.518$ & $0.519$ & $0.528$ & $0.562$ \\ \hline
      \hline
      \cellcolor{lightgray}\small\textrm{Expected}
      & 0.5 & 0.5 & 0.5 & 0.5 & 0.5
      \\ \hline
    \end{tabular}
  \end{table}
}

\newcommand*\TblDgRateFifth{%
  \begin{table}[H]
    \caption{Experimental order of convergence EOC$_\ell$, $\ell \in \{1,\ldots,6\}$,
             for the primal IIDG formulation, i.e., \eqref{eq:discrete_primal_DG} with $\tilde{\nabla}_{h_{\ell}}=\nabla_{h_{\ell}}$,
             with $\beta=1.01 - \frac{8}{5p}$ and, thus, $\rho=0.2$.}
    \label{tbl:DG_rate_0.2}
    \vspace{1ex}
    \setlength\tabcolsep{8pt}
    \centering
    \begin{tabular}{|c||c|c||c||c|c|}
      \hline
      \cellcolor{lightgray}\diagbox[height=1.1\line,width=0.1275\dimexpr\linewidth]{\vspace{-0.5mm}\hspace*{-2mm}$h_\ell$}{\\[-5mm] $p$\hspace*{-2mm}}
      & \cellcolor{lightgray}1.5 & \cellcolor{lightgray}1.7 & \cellcolor{lightgray}2.0 & \cellcolor{lightgray}3.0 & \cellcolor{lightgray}4.5
      \\ \hline\hline
      \cellcolor{lightgray}$0.0834$ & $0.475$ & $0.560$ & $0.355$ & $0.330$ & $0.461$ \\ \hline
      \cellcolor{lightgray}$0.0417$ & $0.332$ & $0.374$ & $0.267$ & $0.255$ & $0.315$ \\ \hline
      \cellcolor{lightgray}$0.0208$ & $0.258$ & $0.273$ & $0.231$ & $0.228$ & $0.253$ \\ \hline
      \cellcolor{lightgray}$0.0104$ & $0.227$ & $0.231$ & $0.217$ & $0.219$ & $0.232$ \\ \hline
      \cellcolor{lightgray}$0.0052$ & $0.215$ & $0.216$ & $0.212$ & $0.216$ & $0.226$ \\ \hline
      \cellcolor{lightgray}$0.0026$ & $0.211$ & $0.211$ & $0.211$ & $0.215$ & $0.224$ \\ \hline
      \hline
      \cellcolor{lightgray}\small\textrm{Expected}
      & 0.2 & 0.2 & 0.2 & 0.2 & 0.2
      \\ \hline
    \end{tabular}
  \end{table}
}

\newcommand*\TblLdgRateOne{%
  \begin{table}[H]
    \caption{Experimental order of convergence EOC$_\ell$, $\ell \in \{1,\ldots,6\}$,
             for the primal LDG formulation, i.e., \eqref{eq:discrete_primal_DG} with $\tilde{\nabla}_{h_{\ell}}=\pmb{\mathsf{\mathcal{G}}}_{h_{\ell}}^1$,
             with $\beta=1.01               $ and, thus, $\rho=1.0$.}
    \label{tbl:LDG_rate_1.0}
    \vspace{1ex}
    \setlength\tabcolsep{8pt}
    \centering
    \begin{tabular}{|c||c|c||c||c|c|}
      \hline
      \cellcolor{lightgray}\diagbox[height=1.1\line,width=0.1275\dimexpr\linewidth]{\vspace{-0.5mm}\hspace*{-2mm}$h_\ell$}{\\[-5mm] $p$\hspace*{-2mm}}
      & \cellcolor{lightgray}1.5 & \cellcolor{lightgray}1.7 & \cellcolor{lightgray}2.0 & \cellcolor{lightgray}3.0 & \cellcolor{lightgray}4.5
      \\ \hline\hline
      \cellcolor{lightgray}$0.0834$ & $0.932$ & $0.928$ & $0.910$ & $0.939$ & $0.970$ \\ \hline
      \cellcolor{lightgray}$0.0417$ & $0.951$ & $0.953$ & $0.944$ & $0.963$ & $0.972$ \\ \hline
      \cellcolor{lightgray}$0.0208$ & $0.961$ & $0.960$ & $0.956$ & $0.971$ & $0.978$ \\ \hline
      \cellcolor{lightgray}$0.0104$ & $0.963$ & $0.963$ & $0.961$ & $0.976$ & $0.985$ \\ \hline
      \cellcolor{lightgray}$0.0052$ & $0.965$ & $0.965$ & $0.964$ & $0.979$ & $0.989$ \\ \hline
      \cellcolor{lightgray}$0.0026$ & $0.967$ & $0.966$ & $0.967$ & $0.980$ & $0.991$ \\ \hline
      \hline
      \cellcolor{lightgray}\small\textrm{Expected}
      & 1.0 & 1.0 & 1.0 & 1.0 & 1.0
      \\ \hline
    \end{tabular}
  \end{table}
}

\newcommand*\TblLdgRateHalf{%
  \begin{table}[H]
    \caption{Experimental order of convergence EOC$_\ell$, $\ell \in \{1,\ldots,6\}$,
             for the primal LDG formulation, i.e., \eqref{eq:discrete_primal_DG} with $\tilde{\nabla}_{h_{\ell}}=\pmb{\mathsf{\mathcal{G}}}_{h_{\ell}}^1$,
             with $\beta=1.01 - \frac{1}{p} $ and, thus, $\rho=0.5$.}
    \label{tbl:LDG_rate_0.5}
    \vspace{1ex}
    \setlength\tabcolsep{8pt}
    \centering
    \begin{tabular}{|c||c|c||c||c|c|}
      \hline
      \cellcolor{lightgray}\diagbox[height=1.1\line,width=0.1275\dimexpr\linewidth]{\vspace{-0.5mm}\hspace*{-2mm}$h_\ell$}{\\[-5mm] $p$\hspace*{-2mm}}
      & \cellcolor{lightgray}1.5 & \cellcolor{lightgray}1.7 & \cellcolor{lightgray}2.0 & \cellcolor{lightgray}3.0 & \cellcolor{lightgray}4.5
      \\ \hline\hline
      \cellcolor{lightgray}$0.0834$ & $0.671$ & $0.659$ & $0.658$ & $0.744$ & $0.857$ \\ \hline
      \cellcolor{lightgray}$0.0417$ & $0.630$ & $0.613$ & $0.614$ & $0.675$ & $0.824$ \\ \hline
      \cellcolor{lightgray}$0.0208$ & $0.584$ & $0.569$ & $0.572$ & $0.610$ & $0.747$ \\ \hline
      \cellcolor{lightgray}$0.0104$ & $0.553$ & $0.542$ & $0.544$ & $0.566$ & $0.661$ \\ \hline
      \cellcolor{lightgray}$0.0052$ & $0.533$ & $0.526$ & $0.528$ & $0.541$ & $0.599$ \\ \hline
      \cellcolor{lightgray}$0.0026$ & $0.522$ & $0.518$ & $0.519$ & $0.528$ & $0.562$ \\ \hline
      \hline
      \cellcolor{lightgray}\small\textrm{Expected}
      & 0.5 & 0.5 & 0.5 & 0.5 & 0.5
      \\ \hline
    \end{tabular}
  \end{table}
}

\newcommand*\TblLdgRateFifth{%
  \begin{table}[H]
    \caption{Experimental order of convergence EOC$_\ell$, $\ell \in \{1,\ldots,6\}$,
             for the primal LDG formulation, i.e., \eqref{eq:discrete_primal_DG} with $\tilde{\nabla}_{h_{\ell}}=\pmb{\mathsf{\mathcal{G}}}_{h_{\ell}}^1$,
             with $\beta=1.01 - \frac{8}{5p}$ and, thus, $\rho=0.2$.}
    \label{tbl:LDG_rate_0.2}
    \vspace{1ex}
    \setlength\tabcolsep{8pt}
    \centering
    \begin{tabular}{|c||c|c||c||c|c|}
      \hline
      \cellcolor{lightgray}\diagbox[height=1.1\line,width=0.1275\dimexpr\linewidth]{\vspace{-0.5mm}\hspace*{-2mm}$h_\ell$}{\\[-5mm] $p$\hspace*{-2mm}}
      & \cellcolor{lightgray}1.5 & \cellcolor{lightgray}1.7 & \cellcolor{lightgray}2.0 & \cellcolor{lightgray}3.0 & \cellcolor{lightgray}4.5
      \\ \hline\hline
      \cellcolor{lightgray}$0.0834$ & $0.475$ & $0.557$ & $0.351$ & $0.323$ & $0.440$ \\ \hline
      \cellcolor{lightgray}$0.0417$ & $0.332$ & $0.373$ & $0.266$ & $0.254$ & $0.308$ \\ \hline
      \cellcolor{lightgray}$0.0208$ & $0.258$ & $0.273$ & $0.230$ & $0.228$ & $0.252$ \\ \hline
      \cellcolor{lightgray}$0.0104$ & $0.227$ & $0.231$ & $0.217$ & $0.219$ & $0.232$ \\ \hline
      \cellcolor{lightgray}$0.0052$ & $0.215$ & $0.216$ & $0.212$ & $0.216$ & $0.226$ \\ \hline
      \cellcolor{lightgray}$0.0026$ & $0.211$ & $0.211$ & $0.211$ & $0.215$ & $0.224$ \\ \hline
      \hline
      \cellcolor{lightgray}\small\textrm{Expected}
      & 0.2 & 0.2 & 0.2 & 0.2 & 0.2
      \\ \hline
    \end{tabular}
  \end{table}
}

\newcommand*\TblCrRateOne{%
  \begin{table}[H]
    \caption{Experimental order of convergence EOC$_\ell$, $\ell \in \{1,\ldots,6\}$,
             for the primal Crouzeix--Raviart formulation (without jump stabilisation), i.e., $\tilde{\nabla}_{h_{\ell}}=\nabla_{h_{\ell}}$ and $\alpha=0$,
             with $\beta=1.01               $ and, thus, $\rho=1.0$.}
    \label{tbl:CR_rate_1.0}
    \vspace{1ex}
    \setlength\tabcolsep{8pt}
    \centering
    \begin{tabular}{|c||c|c||c||c|c|}
      \hline
      \cellcolor{lightgray}\diagbox[height=1.1\line,width=0.1275\dimexpr\linewidth]{\vspace{-0.5mm}\hspace*{-2mm}$h_\ell$}{\\[-5mm] $p$\hspace*{-2mm}}
      & \cellcolor{lightgray}1.5 & \cellcolor{lightgray}1.7 & \cellcolor{lightgray}2.0 & \cellcolor{lightgray}3.0 & \cellcolor{lightgray}4.5
      \\ \hline\hline
      \cellcolor{lightgray}$0.0834$ & $0.770$ & $0.817$ & $0.878$ & $0.899$ & $1.061$ \\ \hline
      \cellcolor{lightgray}$0.0417$ & $0.783$ & $0.834$ & $0.903$ & $0.903$ & $0.929$ \\ \hline
      \cellcolor{lightgray}$0.0208$ & $0.835$ & $0.883$ & $0.944$ & $0.963$ & $0.995$ \\ \hline
      \cellcolor{lightgray}$0.0104$ & $0.873$ & $0.913$ & $0.961$ & $0.981$ & $1.006$ \\ \hline
      \cellcolor{lightgray}$0.0052$ & $0.900$ & $0.931$ & $0.966$ & $0.986$ & $1.004$ \\ \hline
      \cellcolor{lightgray}$0.0026$ & $0.918$ & $0.942$ & $0.968$ & $0.986$ & $1.000$ \\ \hline
      \hline
      \cellcolor{lightgray}\small\textrm{Expected}
      & 1.0 & 1.0 & 1.0 & 1.0 & 1.0
      \\ \hline
    \end{tabular}
  \end{table}
}

\newcommand*\TblCrRateHalf{%
  \begin{table}[H]
    \caption{Experimental order of convergence EOC$_\ell$, $\ell \in \{1,\ldots,6\}$,
             for the primal Crouzeix--Raviart formulation (without jump stabilisation), i.e., $\tilde{\nabla}_{h_{\ell}}=\nabla_{h_{\ell}}$ and $\alpha=0$,
             with $\beta=1.01 - \frac{1}{p} $ and, thus, $\rho=0.5$.}
    \label{tbl:CR_rate_0.5}
    \vspace{1ex}
    \setlength\tabcolsep{8pt}
    \centering
    \begin{tabular}{|c||c|c||c||c|c|}
      \hline
      \cellcolor{lightgray}\diagbox[height=1.1\line,width=0.1275\dimexpr\linewidth]{\vspace{-0.5mm}\hspace*{-2mm}$h_\ell$}{\\[-5mm] $p$\hspace*{-2mm}}
      & \cellcolor{lightgray}1.5 & \cellcolor{lightgray}1.7 & \cellcolor{lightgray}2.0 & \cellcolor{lightgray}3.0 & \cellcolor{lightgray}4.5
      \\ \hline\hline
      \cellcolor{lightgray}$0.0834$ & $0.517$ & $0.546$ & $0.599$ & $0.643$ & $0.910$ \\ \hline
      \cellcolor{lightgray}$0.0417$ & $0.550$ & $0.574$ & $0.607$ & $0.643$ & $0.675$ \\ \hline
      \cellcolor{lightgray}$0.0208$ & $0.554$ & $0.565$ & $0.579$ & $0.643$ & $0.697$ \\ \hline
      \cellcolor{lightgray}$0.0104$ & $0.540$ & $0.545$ & $0.550$ & $0.614$ & $0.679$ \\ \hline
      \cellcolor{lightgray}$0.0052$ & $0.525$ & $0.528$ & $0.531$ & $0.583$ & $0.651$ \\ \hline
      \cellcolor{lightgray}$0.0026$ & $0.515$ & $0.518$ & $0.521$ & $0.560$ & $0.625$ \\ \hline
      \hline
      \cellcolor{lightgray}\small\textrm{Expected}
      & 0.5 & 0.5 & 0.5 & 0.5 & 0.5
      \\ \hline
    \end{tabular}
  \end{table}
}

\newcommand*\TblCrRateFifth{%
  \begin{table}[H]
    \caption{Experimental order of convergence EOC$_\ell$, $\ell \in \{1,\ldots,6\}$,
             for the primal Crouzeix--Raviart formulation (without jump stabilisation), i.e., $\tilde{\nabla}_{h_{\ell}}=\nabla_{h_{\ell}}$ and $\alpha=0$,
             with $\beta=1.01 - \frac{8}{5p}$ and, thus, $\rho=0.2$.}
    \label{tbl:CR_rate_0.2}
    \vspace{1ex}
    \setlength\tabcolsep{8pt}
    \centering
    \begin{tabular}{|c||c|c||c||c|c|}
      \hline
      \cellcolor{lightgray}\diagbox[height=1.1\line,width=0.1275\dimexpr\linewidth]{\vspace{-0.5mm}\hspace*{-2mm}$h_\ell$}{\\[-5mm] $p$\hspace*{-2mm}}
      & \cellcolor{lightgray}1.5 & \cellcolor{lightgray}1.7 & \cellcolor{lightgray}2.0 & \cellcolor{lightgray}3.0 & \cellcolor{lightgray}4.5
      \\ \hline\hline
      \cellcolor{lightgray}$0.0834$ & $0.692$ & $0.660$ & $0.358$ & $0.347$ & $0.469$ \\ \hline
      \cellcolor{lightgray}$0.0417$ & $0.459$ & $0.443$ & $0.279$ & $0.293$ & $0.344$ \\ \hline
      \cellcolor{lightgray}$0.0208$ & $0.365$ & $0.341$ & $0.240$ & $0.267$ & $0.326$ \\ \hline
      \cellcolor{lightgray}$0.0104$ & $0.286$ & $0.267$ & $0.221$ & $0.243$ & $0.298$ \\ \hline
      \cellcolor{lightgray}$0.0052$ & $0.240$ & $0.230$ & $0.214$ & $0.229$ & $0.274$ \\ \hline
      \cellcolor{lightgray}$0.0026$ & $0.219$ & $0.216$ & $0.211$ & $0.222$ & $0.256$ \\ \hline
      \hline
      \cellcolor{lightgray}\small\textrm{Expected}
      & 0.2 & 0.2 & 0.2 & 0.2 & 0.2
      \\ \hline
    \end{tabular}
  \end{table}
}

\newcommand*\TblMixedDgRateOne{%
  \begin{table}[H]
    \caption{Experimental order of convergence EOC$_\ell$, $\ell \in \{1,\ldots,6\}$,
             for the mixed LDG formulation,
             with $\beta=1.01               $ and, thus, $\rho=1.0$.}
    \label{tbl:mixed_DG_rate_1.0}
    \vspace{1ex}
    \setlength\tabcolsep{8pt}
    \centering
    \begin{tabular}{|c||c|c||c||c|c|}
      \hline
      \cellcolor{lightgray}\diagbox[height=1.1\line,width=0.1275\dimexpr\linewidth]{\vspace{-0.5mm}\hspace*{-2mm}$h_\ell$}{\\[-5mm] $p$\hspace*{-2mm}}
      & \cellcolor{lightgray}1.5 & \cellcolor{lightgray}1.7 & \cellcolor{lightgray}2.0 & \cellcolor{lightgray}3.0 & \cellcolor{lightgray}4.5
      \\ \hline\hline
      \cellcolor{lightgray}$0.0834$ & $0.935$ & $0.932$ & $0.916$ & $0.932$ & $0.935$ \\ \hline
      \cellcolor{lightgray}$0.0417$ & $0.949$ & $0.951$ & $0.945$ & $0.965$ & $0.976$ \\ \hline
      \cellcolor{lightgray}$0.0208$ & $0.958$ & $0.958$ & $0.956$ & $0.973$ & $0.986$ \\ \hline
      \cellcolor{lightgray}$0.0104$ & $0.959$ & $0.961$ & $0.961$ & $0.976$ & $0.991$ \\ \hline
      \cellcolor{lightgray}$0.0052$ & $0.952$ & $0.964$ & $0.964$ & $0.978$ & $0.992$ \\ \hline
      \cellcolor{lightgray}$0.0026$ & $0.920$ & $0.965$ & $0.966$ & $0.979$ & $0.991$ \\ \hline
      \hline
      \cellcolor{lightgray}\small\textrm{Expected}
      & 1.0 & 1.0 & 1.0 & 1.0 & 1.0
      \\ \hline
    \end{tabular}
  \end{table}
}

\newcommand*\TblMixedDgRateHalf{%
  \begin{table}[H]
    \caption{Experimental order of convergence EOC$_\ell$, $\ell \in \{1,\ldots,6\}$,
             for the mixed LDG formulation,
             with $\beta=1.01 - \frac{1}{p} $ and, thus, $\rho=0.5$.}
    \label{tbl:mixed_DG_rate_0.5}
    \vspace{1ex}
    \setlength\tabcolsep{8pt}
    \centering
    \begin{tabular}{|c||c|c||c||c|c|}
      \hline
      \cellcolor{lightgray}\diagbox[height=1.1\line,width=0.1275\dimexpr\linewidth]{\vspace{-0.5mm}\hspace*{-2mm}$h_\ell$}{\\[-5mm] $p$\hspace*{-2mm}}
      & \cellcolor{lightgray}1.5 & \cellcolor{lightgray}1.7 & \cellcolor{lightgray}2.0 & \cellcolor{lightgray}3.0 & \cellcolor{lightgray}4.5
      \\ \hline\hline
      \cellcolor{lightgray}$0.0834$ & $0.691$ & $0.680$ & $0.673$ & $0.744$ & $0.838$ \\ \hline
      \cellcolor{lightgray}$0.0417$ & $0.640$ & $0.622$ & $0.620$ & $0.677$ & $0.826$ \\ \hline
      \cellcolor{lightgray}$0.0208$ & $0.591$ & $0.575$ & $0.575$ & $0.612$ & $0.745$ \\ \hline
      \cellcolor{lightgray}$0.0104$ & $0.556$ & $0.545$ & $0.546$ & $0.568$ & $0.660$ \\ \hline
      \cellcolor{lightgray}$0.0052$ & $0.535$ & $0.528$ & $0.529$ & $0.543$ & $0.600$ \\ \hline
      \cellcolor{lightgray}$0.0026$ & $0.522$ & $0.519$ & $0.520$ & $0.529$ & $0.564$ \\ \hline
      \hline
      \cellcolor{lightgray}\small\textrm{Expected}
      & 0.5 & 0.5 & 0.5 & 0.5 & 0.5
      \\ \hline
    \end{tabular}
  \end{table}
}

\newcommand*\TblMixedDgRateFifth{%
  \begin{table}[H]
    \caption{Experimental order of convergence EOC$_\ell$, $\ell \in \{1,\ldots,6\}$,
             for the mixed LDG formulation,
             with $\beta=1.01 - \frac{8}{5p}$ and, thus, $\rho=0.2$.}
    \label{tbl:mixed_DG_rate_0.2}
    \vspace{1ex}
    \setlength\tabcolsep{8pt}
    \centering
    \begin{tabular}{|c||c|c||c||c|c|}
      \hline
      \cellcolor{lightgray}\diagbox[height=1.1\line,width=0.1275\dimexpr\linewidth]{\vspace{-0.5mm}\hspace*{-2mm}$h_\ell$}{\\[-5mm] $p$\hspace*{-2mm}}
      & \cellcolor{lightgray}1.5 & \cellcolor{lightgray}1.7 & \cellcolor{lightgray}2.0 & \cellcolor{lightgray}3.0 & \cellcolor{lightgray}4.5
      \\ \hline\hline
      \cellcolor{lightgray}$0.0834$ & $0.583$ & $0.655$ & $0.406$ & $0.361$ & $0.489$ \\ \hline
      \cellcolor{lightgray}$0.0417$ & $0.403$ & $0.450$ & $0.291$ & $0.271$ & $0.340$ \\ \hline
      \cellcolor{lightgray}$0.0208$ & $0.290$ & $0.312$ & $0.240$ & $0.234$ & $0.266$ \\ \hline
      \cellcolor{lightgray}$0.0104$ & $0.240$ & $0.247$ & $0.220$ & $0.221$ & $0.237$ \\ \hline
      \cellcolor{lightgray}$0.0052$ & $0.220$ & $0.222$ & $0.213$ & $0.217$ & $0.228$ \\ \hline
      \cellcolor{lightgray}$0.0026$ & $0.212$ & $0.213$ & $0.211$ & $0.216$ & $0.224$ \\ \hline
      \hline
      \cellcolor{lightgray}\small\textrm{Expected}
      & 0.2 & 0.2 & 0.2 & 0.2 & 0.2
      \\ \hline
    \end{tabular}
  \end{table}
}
 
\def\mathp{$p\mkern 1.5mu$}

\title{%
  Quasi-optimal Discontinuous Galerkin discretisations of the \mathp-Dirichlet problem\thanks{%
    Submitted to the editors November~27, 2023.
    \funding{%
      AK acknowledges support from the Deutsche Forschungsgemeinschaft  (DFG, German Research
      Foundation) --- within the Walter--Benjamin-Program (project number:
      525389262) and the hospitality of the University of Pisa.
      JB's contribution to this work has been supported by Charles University
      Research program No.\ UNCE/SCI/023.
    }
  }
}

\author{%
    Jan Blechta\thanks{%
    Faculty of Mathematics and Physics, Charles University, 186\;75 Prague, Czech Republic
    (\email{blechta@karlin.mff.cuni.cz}).}
    \and
    Pablo Alexei Gazca-Orozco\thanks{%
    Department of Applied Mathematics, University of Freiburg, 79104, Freiburg, Germany
    (\email{alexei.gazca@mathematik.uni-freiburg.de}, \email{rose@mathematik.uni-freiburg.de}).}
    \and
    Alex Kaltenbach\thanks{%
    Institute of Mathematics, Technical University of Berlin, 10623, Berlin, Germany
    (\email{kaltenbach@math.tu-berlin.de}).}
    \and
    Michael R\r{u}\v{z}i\v{c}ka\footnotemark[3]
}
\headers{Quasi-optimal DG for the $\lowercase{p}$-Dirichlet problem}
        {J.~Blechta, P.~A.~Gazca-Orozco, A.~Kaltenbach, and M.~R\r{u}\v{z}i\v{c}ka}

\begin{document}
\maketitle

\begin{abstract}
    The classical arguments employed when obtaining error estimates of Finite
    Element (FE) discretisations of elliptic problems lead to more restrictive
    assumptions on the regularity of the exact solution when applied to
    non-conforming methods. The so-called minimal regularity estimates
    available in the literature relax some of these assumptions, but are not
    truly of \emph{minimal regularity}, since a~data oscillation term appears
    in the error estimate. Employing an approach based on a~smoothing operator,
    we derive for the first time error estimates for Discontinuous Galerkin
    (DG) type discretisations of non-linear problems with
    $(p,\delta)$-structure that only assume the natural $W^{1,p}$-regularity of
    the exact solution, and which do not contain any oscillation terms.
\end{abstract}



\tableofcontents

\section{Introduction}
\begingroup\setlength\emergencystretch{\hsize}\hbadness=10000
In this paper, we examine Local Discontinuous Galerkin (LDG) and
Incomplete Interior~penalty Discontinuous Galerkin (IIDG)
discretisations of non-linear problems of \textit{$p$-Dirichlet
type},~i.e.,
\begin{align}
    \begin{aligned}
        \label{eq:PDE}
        -\mathop{\mathrm{div}}\nolimits\pmb{\mathsf{\mathcal{S}}}(\nabla \boldsymbol{u}) &= \boldsymbol{f}  &&\quad\text{ in } \Omega\,,\\
        \boldsymbol{u} &= \bm{0}  &&\quad\text{ on }\partial\Omega\,.
    \end{aligned}
\end{align}
\endgroup
Here $\Omega \subseteq \mathbb{R}^d$, $d\geq 2$, is a bounded polyhedral domain having a  Lipschitz continuous boundary $\partial\Omega$, $\bm{f}\colon \Omega \to \mathbb{R}^n$ is a given vector field, and we seek a vector field $\boldsymbol{u}\colon \overline{\Omega} \to \mathbb{R}^n$ solving the system \cref{eq:PDE}.
	The non-linear operator $\pmb{\mathsf{\mathcal{S}}} \colon \mathbb{R}^{n\times d}\to
	\mathbb{R}^{n\times d}$ is assumed to have \textit{$(p,\delta)$-structure}, cf.\ \cref{assum:extra_stress}; the prototypical example falling in to this class
	is
	\begin{align*}
        \SwapAboveDisplaySkip
		\pmb{\mathsf{\mathcal{S}}}(\nabla \boldsymbol{u}) = (\delta +|\nabla \boldsymbol{u}|)^{p-2}\nabla\boldsymbol{u}\,, 
	\end{align*}
	where $p\in (1,\infty)$ and  $\delta \ge 0$.
	
	The central objective of this work is to establish
	a \textit{quasi-optimal} (a priori) error estimate, i.e., 
	a best-approximation result of the form\enlargethispage{1mm}
	\begin{align}\label{eq:error_estimate}
          \begin{aligned}
            &\|\pmb{\mathsf{\mathcal{F}}}(\tilde{\nabla}_h \boldsymbol{u}_h) - \pmb{\mathsf{\mathcal{F}}}(\nabla
            \boldsymbol{u})\|^2_{2,\Omega} +
            m_{\varphi_{\boldsymbol{\beta}_h(\boldsymbol{u}_h)},h}(\boldsymbol{u}_h) \\&\lesssim
            \inf_{\boldsymbol{v}_h \in V_h} \left(\|\pmb{\mathsf{\mathcal{F}}}(\tilde{\nabla}_h \boldsymbol{v}_h) -
              \pmb{\mathsf{\mathcal{F}}}(\nabla \boldsymbol{u})\|^2_{2,\Omega} +
              m_{\varphi_{\boldsymbol{\beta}_h(\boldsymbol{u}_h)},h}(\boldsymbol{v}_h)\right)\,,
          \end{aligned}
	\end{align}
	where $\boldsymbol{u}_h\in V_h$ is the discrete solution, $\tilde{\nabla}_h\colon W^{1,p}(\mathcal{T}_h)^n\to L^p(\Omega)^{n\times d}$ is a suitable discrete~gradient, and
	$\pmb{\mathsf{\mathcal{F}}}$ and $m_{\varphi_{\boldsymbol{\beta}_h(\boldsymbol{u}_h)},h}$ are appropriate measures of the error related to the $(p,\delta)$-structure~of~$\pmb{\mathsf{\mathcal{S}}}$; here, $V_h$ is typically a space of broken polynomials, i.e., $V_h \coloneqq \mathbb{P}^k(\mathcal{T}_h)^n$, on a~\mbox{triangulation}~$\mathcal{T}_h$~of~$\Omega$, which is meant to approximate the full space $W^{1,p}(\Omega)^n$. 
	Crucially, when deriving the estimate \cref{eq:error_estimate} we only assume the natural regularity of the continuous problem; namely, $\boldsymbol{u}\in W^{1,p}_0(\Omega)^n$ and $\boldsymbol{f}\in W^{-1,p'}(\Omega)^n$. Then, as a~corollary we obtain convergence rates, depending on additional regularity conditions.
	
	The main issue with classical approaches based on Strang's
	Lemma (cf.\ \cite{Strang.1972}; see also \cite[p.~106]{Braess.2007} and \cite[Sec.\ 2.3]{EG.2021}) or similar tools (cf.\ \cite{dkrt-ldg}) when deriving error estimates for
	non-conforming discretisations, is that they rely either on
	integrating-by-parts (and, e.g., requiring that $\mathop{\mathrm{div}}\nolimits\pmb{\mathsf{\mathcal{S}}}(\nabla
	\boldsymbol{u})$ is an integrable function) or on assuming that traces of
	$\pmb{\mathsf{\mathcal{S}}}(\nabla \boldsymbol{u})$ are well-defined on the mesh edges to handle
	consistency in the jump terms. The first work that got around
	the assumption of additional unnatural regularity is that of
	Gudi (cf.\ \cite{Gud.2010}), where the author proved, using~tools~from~a posteriori error analysis, the following estimate for the
	Dirichlet~problem,~assuming~just~${\boldsymbol{u}\in W^{1,2}_0(\Omega)^n}$:
	\begin{align}
		\|\boldsymbol{u} - \boldsymbol{u}_h\|_{1,h} \lesssim \inf_{\boldsymbol{v}_h \in V_h} \|\boldsymbol{u} - \boldsymbol{v}_h\|_{1,h} 
		+ \mathrm{osc}_h(\boldsymbol{f})\,,
	\end{align}
	where $\|\cdot\|_{1,h}$ is a broken Sobolev norm and $\mathrm{osc}_h(\boldsymbol{f})$ is a measure for the oscillation of $\boldsymbol{f}\in L^2(\Omega)^n$. Such estimates are in the literature often qualified as being of \emph{minimal regularity}, and the idea has been extended to various other contexts; see, e.g., \cite{BN.2011,GN.2011,BGGS.2012,GGN.2013,LMZ.2014,BCGG.2014,LMNN.2018,BRW.2022} and, in particular,  \cite{AK.2023} for systems of \mathp-Dirichlet type. While this is certainly an improvement over the classical approach, it should be noted that such estimates are not quite minimal in their regularity assumptions, since an additional assumption on the data is needed, namely  ${-\Delta \boldsymbol{u}=\bm{f} \in L^2(\Omega)^n}$. In contrast, we strive for minimal regularity estimates which entail the \emph{equivalence of the error and the~distance~to~}$V_h$.
	
	A whole theory dealing with the characterisation of truly quasi-optimal discretisations of symmetric and elliptic linear problems in $W^{1,2}_0(\Omega)^n$ was developed in \cite{VZ.2018.I,VZ.2019.II,VZ.2018.III} (see also \cite{VZ.2019,KZ.2020,KVZ.2021} for similar results for the Stokes system). In those works, the authors established that quasi-optimality is achieved exactly when the discretisation is \emph{fully algebraically consistent} (i.e., if the exact solution belongs to the discrete space, then it coincides with the discrete solution),~and~\emph{fully~stable}~(i.e., the map $(\bm{f}\mapsto \boldsymbol{u}_h)\colon W^{-1,2}(\Omega)^n\to V_h$ is well-defined and bounded). In particular, it is necessary that the scheme is \emph{entire}, meaning that it is well-defined for general forcing terms $\bm{f}\in W^{-1,2}(\Omega)^n$.
	This is achieved by working with modified schemes of the form: \hspace*{-0.15em}Find $\boldsymbol{u}_h\hspace*{-0.15em} \in\hspace*{-0.15em}  V_h$ such that for every ${\boldsymbol{v}_h\hspace*{-0.15em}\in\hspace*{-0.15em} V_h}$, it holds that
	\begin{align}\label{eq:modified_linear}
		(\tilde{\nabla}_h\boldsymbol{u}_h, \nabla \pmb{\mathsf{\mathcal{E}}}_h  \boldsymbol{v}_h)_\Omega = \langle \boldsymbol{f}, \pmb{\mathsf{\mathcal{E}}}_h  \boldsymbol{v}_h \rangle_{W^{1,p}_0(\Omega)}\,.
	\end{align}
	where \hspace*{-0.1mm}$\pmb{\mathsf{\mathcal{E}}}_h \colon \hspace*{-0.15em} V_h\hspace*{-0.15em} \to\hspace*{-0.15em}  W^{1,2}_0(\Omega)$ \hspace*{-0.1mm}is \hspace*{-0.1mm}a \hspace*{-0.1mm}bounded \hspace*{-0.1mm}operator \hspace*{-0.1mm}(a \hspace*{-0.1mm}so-called \hspace*{-0.1mm}\emph{smoothing \hspace*{-0.1mm}operator}) \hspace*{-0.1mm}leaving~\hspace*{-0.1mm}${V_h \hspace*{-0.15em} \cap \hspace*{-0.15em}  W^{1,2}_0(\Omega)}$ invariant; see also \cite{GP.2018}, where $\pmb{\mathsf{\mathcal{E}}}_h\colon V_h \to  W^{1,2}_0(\Omega)$ is applied to the unknown ${\boldsymbol{u}_h\in  V_h}$ as well. In order to obtain a quasi-optimal scheme, the smoothing operator  is constructed in such a way that it preserves certain moments. More precisely, in the DG setting, for every $ \boldsymbol{z}_h\in V_h$, one has that
	\begin{subequations}\label{eq:moments_preservation}
		\begin{alignat}{2}
			(\boldsymbol{q}_F,\pmb{\mathsf{\mathcal{E}}}_h  \boldsymbol{z}_h)_F &= (\boldsymbol{q}_F, \{\!\!\{ \boldsymbol{z}_h \}\!\!\})_F
			&&\quad \text{ for all } F\in \Gamma_h^i\,, \; \boldsymbol{q}_F\in \mathbb{P}^{k-1}(F)\,,  \label{eq:moments_preservation_face} \\
			(\boldsymbol{q}_K,\pmb{\mathsf{\mathcal{E}}}_h  \boldsymbol{z}_h)_K &= (\boldsymbol{q}_K, \boldsymbol{z}_h)_K
			&&\quad \text{ for all } K\in \mathcal{T}_h\,, \; \boldsymbol{q}_K \in \mathbb{P}^{k-2}(K)\,,  \label{eq:moments_preservation_int}
		\end{alignat}
	\end{subequations}
where $\{\!\!\{\cdot\}\!\!\}$ stands for facet average, $\Gamma_h^i$ is the set of internal facets, and
we set $\mathbb{P}^{-1}(K) \coloneqq \emptyset$; see \cref{sec:prelims} for more details on the DG notation.  %For Crouzeix--Raviart, the averages in \cref{eq:moments_preservation_face} are not needed, and in the lowest-order case the second condition \cref{eq:moments_preservation_int} is also not needed.

	An important consequence of the preservation properties \cref{eq:moments_preservation} (at least in the linear case) is that in practice one does not need to implement the smoothing operator on the left-hand side~of~\cref{eq:modified_linear}. To see this, take an arbitrary element $\pmb{\mathsf{T}}_h \in \Sigma_h \coloneqq  \mathbb{P}^{k-1}(\mathcal{T}_h)^{n\times d}$ (which will represent the flux). Then, using the preservation properties \cref{eq:moments_preservation}, for every $\boldsymbol{z}_h\in V_h$, integration-by-parts yields that
	\begin{align}\label{eq:equivalence_without_E}
		\begin{aligned}
			( \pmb{\mathsf{T}}_h,\nabla \pmb{\mathsf{\mathcal{E}}}_h  \boldsymbol{z}_h)_\Omega
			&=
			-( \mathop{\mathrm{div}}\nolimits_h \pmb{\mathsf{T}}_h, \pmb{\mathsf{\mathcal{E}}}_h  \boldsymbol{z}_h)_\Omega
			+ (\llbracket{\pmb{\mathsf{T}}_h\boldsymbol{n}}\rrbracket ,\pmb{\mathsf{\mathcal{E}}}_h \boldsymbol{z}_h)_{\Gamma_h^i} 
			+  (\{\!\!\{ \pmb{\mathsf{T}}_h \}\!\!\},\llbracket{\pmb{\mathsf{\mathcal{E}}}_h  \boldsymbol{z}_h \boldsymbol{n}}\rrbracket)_{\Gamma_h} \\
			&= 
			-( \mathop{\mathrm{div}}\nolimits_h \pmb{\mathsf{T}}_h,  \boldsymbol{z}_h)_\Omega
			+ (\llbracket{\pmb{\mathsf{T}}_h \boldsymbol{n}}\rrbracket, \boldsymbol{z}_h)_{\Gamma_h^i}
			\\&=	( \pmb{\mathsf{T}}_h,\pmb{\mathsf{\mathcal{G}}}_h \boldsymbol{z}_h)_\Omega \,,
		\end{aligned}
	\end{align}
	where $\pmb{\mathsf{\mathcal{G}}}_h\colon \hspace*{-0.1em} W^{1,p}(\mathcal{T}_h)^n\hspace*{-0.1em}\to\hspace*{-0.1em}
        L^p(\Omega)^{n\times d}$  is the usual DG gradient (cf.\
        \cref{sec:DG_gradient}). In other words, if the fluxes
        are broken polynomials of one degree less than  polynomial
        degree of the solution~$\boldsymbol{u}_h\in V_h$,~the smoothing operator
        needs to be implemented only in the right-hand side (akin~to~modifying only the forcing term as $\pmb{\mathsf{\mathcal{E}}}_h^* \bm{f}\in V_h^*$, where $\pmb{\mathsf{\mathcal{E}}}_h^* \colon\hspace*{-0.1em}(W^{1,p}(\mathcal{T}_h)^n)^*\hspace*{-0.1em} \to\hspace*{-0.1em} V_h^*$~is~the~\mbox{adjoint}~to~${\pmb{\mathsf{\mathcal{E}}}_h \colon \hspace*{-0.1em}V_h\hspace*{-0.1em}\to\hspace*{-0.1em} W^{1,p}(\mathcal{T}_h)^n}$).
	
	In this work, we first prove a best-approximation result of the
	type \cref{eq:error_estimate} for a non-linear analogue of the
	problem \cref{eq:modified_linear} (cf.\ \cref{thm:error_primal});
    since we do not have a~Hilbert
	structure~at~our~disposal (in contrast to the works
	\cite{VZ.2018.I,VZ.2019.II,VZ.2018.III}), our approach is
	based on working directly with an error equation. This is to
	our knowledge the first best-approximation result for
	non-conforming discretisations of non-linear systems of
	\mathp-Dirichlet type that is genuinely minimal in its regularity
	assumptions. One disadvantage of the scheme in the non-linear
	case is that for polynomial degree $k\geq 2$, the argument
	\cref{eq:equivalence_without_E} cannot be employed (since
	non-linear functions of element-wise polynomials are in general
	not element-wise polynomials) and it is, therefore, necessary to implement the smoothing operator also in the left-hand side. As an alternative to this, we propose also a mixed discretisation in which the flux variable belongs, by construction, to $\Sigma_h = \mathbb{P}^{k-1}(\mathcal{T}_h)^{n\times d}$ and,~as a result,~\cref{eq:equivalence_without_E}~applies. For this mixed discretisation, we prove a minimal regularity error estimate as well (cf.\ \cref{thm:error_estimate_mixed}).
	 
	\emph{This paper is organized as follows:} In \cref{sec:prelims}, we introduce the employed notation, define
	relevant function spaces, 
	basic assumptions on $\pmb{\mathsf{\mathcal{S}}}$, and the used discrete operators. In \cref{sec:smoothing_operator}, we recall the construction of the smoothing operator and prove an interpolation error estimate in terms of $N$-functions.
	In \cref{sec:primal}, we introduce the continuous  and the discrete primal problem, establish their well-posedness and the validity of a quasi-optimal (a priori) error estimate~of~the~form~\cref{eq:error_estimate}. Then, from the latter, we deduce the convergence of the discrete primal problem under \textit{minimal regularity} assumptions and derive fractional error decay rates given fractional regularity assumptions expressed in Nikolski\u{\i} spaces.
	In addition, aided by the quasi-optimal  error estimate, we carry out an ansatz class competition that shows that the approximation capabilities of LDG and IIDG approximations and continuous Lagrange approximations  of the problem \cref{eq:PDE}  are comparable. In \cref{sec:mixed}, we introduce the continuous   and the discrete mixed problem, establish their well-posedness and the validity of a quasi-optimal (a priori) error estimate~in~a~similar~form~as~\cref{eq:error_estimate}.
	In \cref{sec:experiments}, we carry out numerical experiments to complement the theoretical findings.
	
	\section{Preliminaries}\label{sec:prelims}
	
	We employ $c, C>0$ to denote generic constants, that may
	change from line to line and may depend only on
	the polynomial degree $k$,
	the chunkiness $\omega_0$,  the
	characteristics of $\pmb{\mathsf{\mathcal{S}}}$, and the dimensions $n$, $d$.
	
	Moreover, we~write $f \lesssim g$ if
	there exist a constant $c>0$ such that
	$f \le c\, g$, and ${f\sim g}$ if and only if there exists constants $c,C>0$ such
	that $c\, f \le g\le C\, f$.
	
	Throughout the entire paper, let $\Omega\subseteq \mathbb{R}^d$, $d \in  \mathbb{N}$, is a bounded,
	polyhedral Lipschitz domain~and $M\subseteq  \mathbb{R}^d$, $d \in  \mathbb{N}$,  a (Lebesgue) measurable set. Then, 
	for  every $k\in \ensuremath{\mathbb{N}}$ and $p\in [1,\infty]$, we employ the customary
	Lebesgue spaces $(L^p(M), \smash{\|\cdot\|_{p,M}}) $ and Sobolev
	spaces $(W^{k,p}(M), \smash{\|\cdot\|_{k,p,M}})$. The space $W^{1,p}_0(M)$
	is defined as the closure of the vector space of smooth and compactly supported functions $C^\infty_c(M)$ in  $W^{1,p}(M)$. We equip $\smash{W^{1,p}_0(M)}$ 
	with the norm~$\smash{\|\nabla\cdot\|_{p,M}}$. 
	
	We  always denote
	vector-valued functions by boldface letters~and~tensor-valued
	functions by capital boldface letters. The Euclidean scalar product
	between  two vectors $\boldsymbol{a}=(a_1,\dots,a_n)^\top$, ${\boldsymbol{b}=(b_1,\dots,b_n)^\top}\in \mathbb{R}^n$, $n\in \mathbb{N}$, is defined by $\boldsymbol{a} \cdot\boldsymbol{b}\coloneqq \sum_{i=1}^{n}{a_ib_i}$, while the
	Frobenius scalar product between two tensors $\pmb{\mathsf{A}}=(A_{ij}) _{1\leq i\leq n; 1\leq j\leq \ell},\pmb{\mathsf{B}}=(B_{ij}) _{1\leq i\leq n; 1\leq j\leq \ell}\in \mathbb{R}^{n\times \ell}$, $n,\ell\in \mathbb{N}$, is defined by
	$\pmb{\mathsf{A}}: \pmb{\mathsf{B}}\coloneqq \sum_{i=1}^{n}\sum_{j=1}^{\ell}{A_{ij}B_{ij}}$. Then, the Euclidean norm of a vector   $\boldsymbol{a}\in \mathbb{R}^n$, $n\in \mathbb{N}$, is defined by $\vert \boldsymbol{a} \vert \coloneqq \sqrt{\boldsymbol{a}\cdot\boldsymbol{a}}$, while the Frobenius norm of a tensor $\pmb{\mathsf{A}}\in \mathbb{R}^{n\times \ell}$, $n,\ell\in \mathbb{N}$,~is~defined~by~$\vert \pmb{\mathsf{A}} \vert \coloneqq \sqrt{\pmb{\mathsf{A}}:\pmb{\mathsf{A}}}$.

The  mean  value  of a~locally integrable function $f$  is
denoted by
$\langle{f}\rangle_M\allowbreak\coloneqq\smash{\dashint_M f
        \,\textup{d}x}\allowbreak\coloneqq \smash{\frac 1 {|M|}\int_M f
        \,\textup{d}x}$. Moreover,  we employ the notation
$(f,g)_M\coloneqq \int_M f g\,\textup{d}x$, whenever the
right-hand side is~\mbox{well-defined}.

	Drawing from the theory of Orlicz spaces $L^\psi (M)$ (cf.~\cite{ren-rao}) and generalized
	Orlicz~spaces~$L^{\psi(\cdot)} (M)$ (cf.~\cite{HH19}), we employ \mbox{N-functions}
	$\psi \colon \ensuremath{\mathbb{R}}^{\geq 0} \to \ensuremath{\mathbb{R}}^{\geq 0}$ and generalized \mbox{N-functions}
	$\psi \colon M \times \ensuremath{\mathbb{R}}^{\ge 0} \to \ensuremath{\mathbb{R}}^{\ge 0}$, i.e.,
	$\psi$ is a Carath\'eodory function such that $\psi(x,\cdot)$ is an
	N-function for a.e.\ $x \in M$,~\mbox{respectively}. The
	modular~is~defined~via
	$\rho_{\psi,M}(f)\coloneqq \int_M
	\psi(\left| f \right|)\,\textup{d}x $ if $\psi$ is an N-function, and via
	$ \rho_{\psi,M}(f)\coloneqq \int_M
	\psi(x,\left| f(x) \right|)\,\textup{d}x $, if $\psi$ is a generalized
	N-function. An N-function
	$\psi$ satisfies the $\Delta_2$-condition (in short,
	$\psi \in  \Delta_2$), if there exists
	$K> 2$ such that for every
	$t \ge 
	0$,~it~holds~that~${\psi(2\,t) \leq K\,
		\psi(t)}$. We denote the smallest such constant by
	$\Delta_2(\psi)>0$. We define the (convex) conjugate (generalized) N-function $\psi^*\colon M\times \mathbb{R}^{\ge 0}\to \mathbb{R}^{\ge 0}$ via $\psi^*(x,t)\coloneqq \sup_{s\ge 0}{ ts-\psi(x,s)}$ for all $t\ge 0$ and a.e.\ ${x\in M}$. If $\psi,
	\psi^* \in \Delta_2$,  then we have that
	\begin{align}
		\label{eq:psi'}
		\psi^* \circ \psi'\sim \psi\,,
	\end{align}
	with constants depending only on $\Delta_2(\psi),\Delta_2(
	\psi ^*)>0$. 
	We will also need the $\varepsilon$-Young inequality: for every
	$\varepsilon\hspace*{-0.1em}> \hspace*{-0.1em} 0$, there exits a constant $c_\varepsilon\hspace*{-0.1em}>\hspace*{-0.1em}0 $, depending only on $\Delta_2(\psi),\Delta_2( \psi ^*)\hspace*{-0.1em}<\hspace*{-0.1em}\infty$, such that for every $s,t\geq 0$, it holds that
	\begin{align}
		\label{ineq:young}
		\begin{aligned}
			t\,s&\leq \varepsilon \, \psi(t)+ c_\varepsilon \,\psi^*(s)\,.
		\end{aligned}
	\end{align}

	\subsection{Basic properties of the non-linear operator}
	
	Throughout the paper, we always assume that the non-linear operator 
	$\pmb{\mathsf{\mathcal{S}}}$
	has $(p,\delta)$-structure. 
	A detailed discussion and  proofs can be found, e.g., in
	\cite{die-ett,dr-nafsa}. 
	
	For $p \in (1,\infty)$~and~$\delta\ge 0$, we define a special N-function
	$\varphi=\varphi_{p,\delta}\colon\ensuremath{\mathbb{R}}^{\ge 0}\to \ensuremath{\mathbb{R}}^{\ge 0}$ via
	\begin{align} 
        \SwapAboveDisplaySkip
		\label{eq:def_phi} 
		\varphi(t)\coloneqq  \int _0^t \varphi'(s)\, \mathrm ds,\quad\text{where}\quad
		\varphi'(t) \coloneqq  (\delta +t)^{p-2} t\,,\quad\textup{ for all }t\ge 0\,.
	\end{align}
	The properties of $\varphi$ are discussed in detail in \cite{die-ett,dr-nafsa,kr-pnse-ldg-1}.
	An important tool in our analysis play {\em shifted N-functions}
	$\{\psi_a\}_{\smash{a \ge 0}}$,~cf.~\cite{DK08,dr-nafsa}. For a given N-function $\psi\colon\mathbb{R}^{\ge 0}\to \mathbb{R}^{\ge
		0}$ we define the family  of shifted N-functions ${\psi_a\colon\mathbb{R}^{\ge
			0}\to \mathbb{R}^{\ge 0}}$,~${a \ge 0}$,~via
	\begin{align}
		\label{eq:phi_shifted}
		\psi_a(t)\coloneqq  \int _0^t \psi_a'(s)\, \mathrm ds\,,\quad\text{where }\quad
		\psi'_a(t)\coloneqq \psi'(a+t)\frac {t}{a+t}\,,\quad\textup{ for all }t\ge 0\,.
	\end{align}

	\begin{Definition}[$(p,\delta)$-structure]\label{assum:extra_stress} 
		Let $\pmb{\mathsf{\mathcal{S}}}\in\allowbreak\hspace*{-0.05em} \allowbreak C^0(\mathbb{R}^{n \times
			d},\allowbreak\mathbb{R}^{n \times d} ) $ satisfy ${\pmb{\mathsf{\mathcal{S}}}(\mathbf{0})\hspace*{-0.05em}=\hspace*{-0.05em}\mathbf{0}}$. Then, we say
		that  $\pmb{\mathsf{\mathcal{S}}}$ has
		\textup{$(p,\delta)$-structure} if for some $p \in (1, \infty)$,
		$ \delta\in [0,\infty)$, and the N-function
		$\varphi=\varphi_{p,\delta}$ (cf.~\cref{eq:def_phi}), there
		exist constants $C_0, C_1 >0$ such that
		\begin{align}
			\label{eq:ass_S}
			\begin{aligned}
                \big({\pmb{\mathsf{\mathcal{S}}}}(\pmb{\mathsf{Q}}) - {\pmb{\mathsf{\mathcal{S}}}}(\pmb{\mathsf{P}})\big) \mathbin{:} (\pmb{\mathsf{Q}}-\pmb{\mathsf{P}}) &\ge C_0 \,\varphi_{\vert \pmb{\mathsf{Q}}\vert}(\left| \pmb{\mathsf{Q}} -
					\pmb{\mathsf{P}} \right|) \,,%\label{1.4b}
				\\
				\left| \pmb{\mathsf{\mathcal{S}}}(\pmb{\mathsf{Q}}) - \pmb{\mathsf{\mathcal{S}}}(\pmb{\mathsf{P}}) \right| &\le C_1 \,
				\varphi'_{\vert \pmb{\mathsf{Q}}\vert}(\left| \pmb{\mathsf{Q}} -
					\pmb{\mathsf{P}} \right|)%\label{1.5b}
			\end{aligned}
		\end{align}
		are satisfied for all $\pmb{\mathsf{Q}},\pmb{\mathsf{P}} \in \ensuremath{\mathbb{R}}^{n\times d}$
        with $\pmb{\mathsf{Q}}\neq \mathbf{0}$. The constants $C_0,C_1>0$ and $p\in
        (1,\infty)$ are called the {characteristics} of $\pmb{\mathsf{\mathcal{S}}}$.
	\end{Definition}

	\begin{Remark}\label{rem:phi}
		{\rm (i) 
			Assume that $\pmb{\mathsf{\mathcal{S}}}$ has $(p,\delta)$-structure for some
			$\delta \in [0,\delta_0]$. Then, if not otherwise stated, the
			constants in the estimates depend only on the characteristics~of~$\pmb{\mathsf{\mathcal{S}}}$~and on $\delta_0\ge 0$, but are independent of $\delta\ge 0$.
			
			(ii)     Let $\varphi$ be defined in \cref{eq:def_phi} and 
			$\{\varphi_a\}_{a\ge 0}$ be the corresponding family of the shifted \mbox{N-functions}. Then, the operators 
			$\pmb{\mathsf{\mathcal{S}}}_a\colon\mathbb{R}^{n\times d}\to \smash{\mathbb{R}_{\textup{sym}}^{n\times
					d}}$, $a \ge 0$, for every $a \ge 0$
				and~$\pmb{\mathsf{Q}} \in \mathbb{R}^{n\times d}$, defined via 
			\begin{align}
				\label{eq:flux}
				\pmb{\mathsf{\mathcal{S}}}_a(\pmb{\mathsf{Q}}) \coloneqq 
				\frac{\varphi_a'(\vert \pmb{\mathsf{Q}}\vert)}{\vert \pmb{\mathsf{Q}}\vert}\,
				\pmb{\mathsf{Q}} =(\delta +a +\vert \pmb{\mathsf{Q}}\vert)^{p-2}\pmb{\mathsf{Q}}\,, 
			\end{align}
			have $(p, \delta +a)$-structure.  In this case, the characteristics of
			$\pmb{\mathsf{\mathcal{S}}}_a$ depend only on ${p\in (1,\infty)}$ and are independent of
			$\delta \geq 0$ and $a\ge 0$.
			
			(iii)   Note that $\varphi_a(t) \sim (\delta+a+t)^{p-2} t^2$ and 
			${(\varphi_a)^*(t) \sim ((\delta+a)^{p-1} + t)^{p'-2} t^2}$ uniformly with respect to $t,a\ge 0$.  The
			families ${\{{\varphi_a}\}}_{a \ge 0}$ and ${\{{(\varphi_a)^*}\}}_{a \ge 0}$
			satisfy the $\Delta_2$-condition uniformly with respect to ${a \ge
				0}$, with $\Delta_2(\varphi_a) \lesssim 2^{\max {\{{2,p}\}}}$ and
			$\Delta_2((\varphi_a)^*) \lesssim 2^{\max {\{{2,p'}\}}}$,
			respectively. Moreover, for all $0\le a\le b$, we have that for every $t \ge0$, it holds that $(\varphi_a)^*(t)
			\ge (\varphi_b)^*(t) $, $\varphi_a(t)  \le \varphi_b(t) $~if~$p\ge 2$ and $(\varphi_a)^*(t)
			\le (\varphi_b)^*(t) $, $\varphi_a(t)  \ge \varphi_b(t) $ if $p\le 2$.
		}
	\end{Remark}
	
	Closely related to the non-linear operator  $\pmb{\mathsf{\mathcal{S}}}$ with
	$(p,\delta)$-structure are the non-linear operators
	$\pmb{\mathsf{\mathcal{F}}},\pmb{\mathsf{\mathcal{F}}}^*\colon\ensuremath{\mathbb{R}}^{n\times d}\to \ensuremath{\mathbb{R}}^{n\times d}$, 
	for every $\pmb{\mathsf{Q}}\in \mathbb{R}^{n\times d}$, defined via
	\begin{align}
		\begin{aligned}
			\pmb{\mathsf{\mathcal{F}}}(\pmb{\mathsf{Q}})&\coloneqq (\delta+\vert \pmb{\mathsf{Q}}\vert)^{\smash{\frac{p-2}{2}}}\pmb{\mathsf{Q}}\,,\\
			\pmb{\mathsf{\mathcal{F}}}^*(\pmb{\mathsf{Q}})&\coloneqq (\delta^{p-1}+\vert \pmb{\mathsf{Q}}\vert)^{\smash{\frac{p'-2}{2}}}\pmb{\mathsf{Q}}\,.
		\end{aligned}\label{eq:def_F}
	\end{align}
	The connections between
	$\pmb{\mathsf{\mathcal{S}}},\pmb{\mathsf{\mathcal{F}}},\pmb{\mathsf{\mathcal{F}}}^* \hspace{-0.05em}\colon\hspace{-0.05em}\ensuremath{\mathbb{R}}^{n \times d}
	\hspace{-0.05em}\to\hspace{-0.05em} \ensuremath{\mathbb{R}}^{n\times d}$ and
	$\varphi_a,(\varphi_a)^*\hspace{-0.05em}\colon\hspace{-0.05em}\ensuremath{\mathbb{R}}^{\ge
		0}\hspace{-0.05em}\to\hspace{-0.05em} \ensuremath{\mathbb{R}}^{\ge
		0}$,~${a\hspace{-0.05em}\ge\hspace{-0.05em} 0}$, are best explained
	by the following result (cf.~\cite{die-ett,dr-nafsa,dkrt-ldg}). 
	\begin{Proposition}
		\label{lem:hammer}
		Let $\pmb{\mathsf{\mathcal{S}}}$ have $(p,\delta)$-structure, let $\varphi$ be defined in \cref{eq:def_phi}, and let $\pmb{\mathsf{\mathcal{F}}},\pmb{\mathsf{\mathcal{F}}}^*$ be defined in \cref{eq:def_F}. Then, uniformly with respect to 
		$\pmb{\mathsf{Q}}, \pmb{\mathsf{P}} \in \ensuremath{\mathbb{R}}^{n \times d}$, we have that
		\begin{align}\label{eq:hammera}
			\begin{aligned}
				\big(\pmb{\mathsf{\mathcal{S}}}(\pmb{\mathsf{Q}}) - \pmb{\mathsf{\mathcal{S}}}(\pmb{\mathsf{P}})\big)
				:(\pmb{\mathsf{Q}}-\pmb{\mathsf{P}} ) &\sim  \left|  \pmb{\mathsf{\mathcal{F}}}(\pmb{\mathsf{Q}}) - \pmb{\mathsf{\mathcal{F}}}(\pmb{\mathsf{P}}) \right|^2
				\\
				&\sim \varphi_{\vert \pmb{\mathsf{Q}}\vert}(\left| \pmb{\mathsf{Q}}
					- \pmb{\mathsf{P}} \right|)
				\\
				&\sim(\varphi_{\vert \pmb{\mathsf{Q}}\vert})^*(\left| \pmb{\mathsf{\mathcal{S}}}(\pmb{\mathsf{Q}} ) - \pmb{\mathsf{\mathcal{S}}}(\pmb{\mathsf{P}} ) \right|)
				\\&\sim (\varphi^*) _{|\pmb{\mathsf{\mathcal{S}}} ( \pmb{\mathsf{Q}})|} (|\pmb{\mathsf{\mathcal{S}}} (\pmb{\mathsf{Q}}) - \pmb{\mathsf{\mathcal{S}}} ( \pmb{\mathsf{P}}) |)
				\,,
			\end{aligned}
		\end{align}
		%\vspace{-6.5mm}
		\begin{align}
			\label{eq:hammere}
		\mspace{-10mu}	\left| \pmb{\mathsf{\mathcal{S}}}(\pmb{\mathsf{Q}}) - \pmb{\mathsf{\mathcal{S}}}(\pmb{\mathsf{P}}) \right| &\sim   \smash{\varphi'_{\vert \pmb{\mathsf{Q}}\vert}(\vert\pmb{\mathsf{Q}}-\pmb{\mathsf{P}}\vert )}\,,
		\end{align}
		\begin{align}
			\mspace{-70mu}\smash{\left|  \pmb{\mathsf{\mathcal{F}}}^*(\pmb{\mathsf{Q}}) - \pmb{\mathsf{\mathcal{F}}}^*(\pmb{\mathsf{P}}) \right|^2}
			\label{eq:hammerf}
			&\sim
			\smash{\smash{(\varphi^*)}_{\smash{\vert \pmb{\mathsf{Q}}\vert}}(\left| \pmb{\mathsf{Q}}
					- \pmb{\mathsf{P}} \right|)} \,,
			\\[2mm]
			\label{eq:F-F*3}
			\mspace{-70mu}\smash{\left| \pmb{\mathsf{\mathcal{F}}}^*(\pmb{\mathsf{\mathcal{S}}}(\pmb{\mathsf{P}}))-\pmb{\mathsf{\mathcal{F}}}^*(\pmb{\mathsf{\mathcal{S}}}(\pmb{\mathsf{Q}})) \right|^2}
			&\sim  \smash{\left| \pmb{\mathsf{\mathcal{F}}}(\pmb{\mathsf{P}})-\pmb{\mathsf{\mathcal{F}}}(\pmb{\mathsf{Q}}) \right|^2}\,.
		\end{align}
		The constants in \crefrange{eq:hammera}{eq:F-F*3}
		depend only on the characteristics of ${\pmb{\mathsf{\mathcal{S}}}}$.
	\end{Proposition} 
	\begin{Remark}\label{rem:sa}
		{\rm
			For the operators $\pmb{\mathsf{\mathcal{S}}}_a\colon\mathbb{R}^{n\times d}\to\smash{\mathbb{R}_{\textup{sym}}^{n\times
					d}}$, $a \ge  0$, defined~in~\cref{eq:flux}, the assertions of \cref{lem:hammer} hold with $\varphi\colon\mathbb{R}^{\ge 0}\to \mathbb{R}^{\ge 0}$ replaced
			by $\varphi_a\colon\mathbb{R}^{\ge 0}\to \mathbb{R}^{\ge 0}$, $a\ge 0$.}
	\end{Remark}

	The following results can be found in~\cite{DK08,dr-nafsa}.
	
	\begin{Lemma}[Change of shift]\label{lem:shift-change}
		Let $\varphi$ be defined in \cref{eq:def_phi} and let $\pmb{\mathsf{\mathcal{F}}}$ be defined in \cref{eq:def_F}. Then,
		for each $\varepsilon>0$, there exists $c_\varepsilon\geq 1$ (depending only
		on~$\varepsilon>0$ and $p$) such that for every $\pmb{\mathsf{Q}},\pmb{\mathsf{P}}\in\smash{\ensuremath{\mathbb{R}}^{n \times d}}$ and $t\geq 0$, it holds that
		\begin{align*}
			\smash{\varphi_{\left| \pmb{\mathsf{P}} \right|}(t)}&\leq \smash{c_\varepsilon\, \varphi_{\vert \pmb{\mathsf{Q}}\vert}(t)
				+\varepsilon\, \left| \pmb{\mathsf{\mathcal{F}}}(\pmb{\mathsf{P}}) - \pmb{\mathsf{\mathcal{F}}}(\pmb{\mathsf{Q}}) \right|^2\,,}
			\\
			\smash{(\varphi_{\left| \pmb{\mathsf{P}} \right|})^*(t)}&\leq \smash{c_\varepsilon\, (\varphi_{\vert \pmb{\mathsf{Q}}\vert})^*(t)
				+\varepsilon\, \left| \pmb{\mathsf{\mathcal{F}}}(\pmb{\mathsf{P}}) - \pmb{\mathsf{\mathcal{F}}}(\pmb{\mathsf{Q}}) \right|^2}\,.
		\end{align*}
	\end{Lemma}

	\subsection{Mesh regularity}
	
	Throughout the entire paper, $\{\mathcal{T}_h\}_{h>0}$ always denotes
	a~family of conforming
	triangulations of $\overline{\Omega}\subseteq \mathbb{R}^d$,
	$d\ge 2$, cf.\ \cite{BS08},  consisting of
	$d$-dimensional closed simplices~$K$.
	The parameter $h>0$, refers to the maximal mesh-size of $\mathcal{T}_h$, i.e., if we define $h_K\coloneqq \textup{diam}(K)$ for all $K\in \mathcal{T}_h$, then ${h\coloneqq \max_{K\in \mathcal{T}_h}{h_K}}$.
	For simplicity, we  assume  that  $h \le 1$.
	For a simplex $K \in \mathcal{T}_h$,
	we denote by $\rho_K>0$, the supremum of diameters~of~inscribed~balls. We assume that there exists a constant $\omega_0>0$, independent
	of $h>0$, such that ${h_K}{\rho_K^{-1}}\le
	\omega_0$ for every $K \in \mathcal{T}_h$. The smallest such constant is called the chunkiness of $\{\mathcal{T}_h\}_{h>0}$. 
	
	We define the sets of $(d-1)$-dimensional faces $\Gamma_h$, interior faces $\Gamma_h^{i}$, and boundary faces $\Gamma_h^{\partial}$ of the partition $\mathcal{T}_h$ via
	\begin{align*}
		\Gamma_h^{i}&\coloneqq  \{K\cap K'\mid K,K'\in \mathcal{T}_h\,,\text{dim}_{\mathscr{H}}(K\cap K')=d-1\}\,,\\[-0.5mm]
		\Gamma_h^{\partial}&\coloneqq\{K\cap\partial\Omega\mid K\in \mathcal{T}_h\,,\text{dim}_{\mathscr{H}}(K\cap
		\partial\Omega)=d-1\}\,,\\[-0.5mm]
		\Gamma_h&\coloneqq \Gamma_h^{i}\cup \Gamma_h^{\partial}\,,
	\end{align*}
	where for every subset  $S\subseteq \mathbb{R}^d$, we denote by $\text{dim}_{\mathscr{H}}(S)\coloneqq\inf\{d'\geq 0\mid \mathscr{H}^{d'}(S)=0\}$, the~Hausdorff dimension ($\mathscr{H}^{d'}(S)$ represents here the $d'$-dimensional Hausdorff measure of $S$). The (local) mesh-size function $h_{\mathcal{T}}\colon \overline{\Omega}\to \mathbb{R}$  is defined via ${h_{\mathcal{T}}|_K\coloneqq h_K} $ for all $K\in \mathcal{T}_h$.
	The (local) face-size function $h_{\Gamma}\colon \Gamma_h\to
        \mathbb{R}$  is defined via $h_{\Gamma}|_F\coloneqq h_F
        \coloneqq \text{diam}(F)$ for all $F\in \Gamma_h$.
	
	\subsubsection{Broken function spaces and projectors}
	
	For every $k \in \mathbb{N}_0$~and~$K\in \mathcal{T}_h$,
	we denote~by~$\mathbb{P}^k(K)$, the space of
	polynomials of degree at most $k$ on $K$.  Then, for given
	$k\in \mathbb{N}_0$, we define the space of \textit{broken
		polynomials of global degree at most $k$} via
	\begin{align*}
        \SwapAboveDisplaySkip
		\mathbb{P}^k(\mathcal T_h)&\coloneqq\big\{v_h\in
		L^\infty(\Omega)\mid v_h|_K\in \mathbb{P}^k(K)\text{ for all }K\in \mathcal{T}_h\big\}\,,
	\end{align*}
	and the space of \textit{broken
		polynomials} via
	$
	\mathbb{P}(\mathcal T_h) \coloneqq \bigcup_{k \in
		\ensuremath{\mathbb{N}}} \mathbb{P}^k(\mathcal T_h)\,.
	$ 
	In addition, 
	we define the \textit{space of element-wise continuous functions} via
	\begin{align*}
		C^0(\mathcal T_h)&\coloneqq\big\{w_h\in L^\infty(\Omega)\mid w_h|_K\in C^0(K)\text{ for all }K\in \mathcal{T}_h\big\}\,,
	\end{align*}
	and for
	given~$p\in (1,\infty)$, we define the \textit{broken Sobolev space} via
	\begin{align*}
		W^{1,p}(\mathcal T_h)&\coloneqq\big\{w_h\in L^p(\Omega)\mid w_h|_K\in W^{1,p}(K)\text{ for all }K\in \mathcal{T}_h\big\}\,.
	\end{align*}

	For  each $\boldsymbol{w}_h\in W^{1,p}(\mathcal{T}_h)^n$, we denote by $\nabla_h \boldsymbol{w}_h\in L^p(\Omega)^{n\times d}$ 
	the \textit{local gradient}: for every ${K\in\mathcal{T}_h}$ it is defined via
	$(\nabla_h \boldsymbol{w}_h)|_K\coloneqq \nabla(\boldsymbol{w}_h|_K)$ for all ${K\in\mathcal{T}_h}$.
	For each $K\in \mathcal{T}_h$, ${\boldsymbol{w}_h\in W^{1,p}(\mathcal{T}_h)^n}$ admits a trace ${\textrm{tr}^K(\boldsymbol{w}_h)\in L^p(\partial K)^n}$. For each face
	$F\in \Gamma_h$ of a given element $K\in \mathcal{T}_h$, we
        define this interior trace via 
	$\smash{\textup{tr}^K_F(\boldsymbol{w}_h)\in L^p(F)^n}$. Then, given some multiplication operator ${\odot\colon \mathbb{R}^n\times \mathbb{R}^d\to \mathbb{R}^\ell}$,~${\ell\in \mathbb{N}}$, for
	every $\boldsymbol{w}_h\in W^{1,p}(\mathcal{T}_h)^n$ and interior faces $F\in \Gamma_h^{i}$ shared by
	adjacent elements $K^-_F, K^+_F\in \mathcal{T}_h$,
	we denote by
	\begin{align*}
		\{\!\!\{\boldsymbol{w}_h\}\!\!\}_F&\coloneqq\tfrac{1}{2}\big(\textup{tr}_F^{K^+}(\boldsymbol{w}_h)+
		\textup{tr}_F^{K^-}(\boldsymbol{w}_h)\big)\in
		L^p(F)^n\,,\\
		\llbracket \boldsymbol{w}_h\odot \boldsymbol{n}\rrbracket_F
		&\coloneqq\textup{tr}_F^{K^+}(\boldsymbol{w}_h)\odot\boldsymbol{n}^+_F+
		\textup{tr}_F^{K^-}(\boldsymbol{w}_h)\odot\boldsymbol{n}_F^- 
		\in L^p(F)^\ell\,,
	\end{align*}
	the \textit{average} and \textit{jump}, respectively, of $\boldsymbol{w}_h$ on $F$.
	\!Moreover,  for every $\boldsymbol{w}_h\!\in\! W^{1,p}(\mathcal{T}_h)^n$ and boundary faces $F\in \Gamma_h^{\partial}$, we define \textit{boundary averages} and 
	\textit{boundary jumps}, respectively, via
	\begin{align*}
        \SwapAboveDisplaySkip
		\{\!\!\{\boldsymbol{w}_h\}\!\!\}_F&\coloneqq\textup{tr}^\Omega_F(\boldsymbol{w}_h) \in L^p(F)^n\,, \\
		\llbracket \boldsymbol{w}_h\odot\boldsymbol{n}\rrbracket_F&\coloneqq
		\textup{tr}^\Omega_F(\boldsymbol{w}_h)\odot\boldsymbol{n} \in L^p(F)^\ell\,.
	\end{align*}
	If there is no
	danger~of~confusion, we will omit the index $F\in \Gamma_h$; in particular,  if we interpret jumps and averages as global functions defined on the whole of $\Gamma_h$.
	In addition, for every $\boldsymbol{w}_h\in W^{1,p}(\mathcal{T}_h)^n$, we  introduce the DG norm via
	\begin{align*}
		\|\boldsymbol{w}_h\|_{h,p}\coloneqq\big(\|\nabla_h\boldsymbol{w}_h\|_{p,\Omega}^p+\big\|h^{-1/p'}_\Gamma\llbracket{\boldsymbol{w}_h\otimes\boldsymbol{n}}\rrbracket\big\|_{p,\Gamma_h}^p\big)^{1/p}\,,
	\end{align*}
	which turns $W^{1,p}(\mathcal{T}_h)^n$ into a Banach space. We denote by $\Pi_h^k \colon L^1(\Omega)\to \mathbb{P}^k(\mathcal{T}_h)$, the (local)~$L^2$-projection onto $\mathbb{P}^k(\mathcal{T}_h)$, defined for every $v \in
	L^1(\Omega)$ via  $(\Pi_h^k v,v_h)_{\Omega}=(v,v_h)_{\Omega}$ for all $v_h \in \mathbb{P}^k(\mathcal{T}_h)$.
	Analogously, we define the (local)
	$L^2$-projections into $\mathbb{P}^k(\mathcal{T}_h)^n$, i.e., $\Pi_h^k\colon L^1(\Omega)^n \to \mathbb{P}^k(\mathcal{T}_h)^n$, and into $\mathbb{P}^k(\mathcal{T}_h)^{n\times d}$, i.e., $\Pi_h^k\colon L^1(\Omega)^{n\times d} \to \mathbb{P}^k(\mathcal{T}_h)^{n\times d}$.
	
	\subsubsection{DG gradient and jump operator} \label{sec:DG_gradient}
	
	For every $k\in \mathbb{N}$, we define the \textit{(global)
		jump lifting operator}
	$\boldsymbol{\mathcal{R}}_h^k\colon W^{1,p}(\mathcal{T}_h)^n \to
	\mathbb{P}^{k-1} (\mathcal{T}_h)^{n\times d}$  (using Riesz representation)  for every
	$\boldsymbol{w}_h\in W^{1,p}(\mathcal{T}_h)^n $ 
	via
	\begin{align*}
		(\boldsymbol{\mathcal{R}}_h\boldsymbol{w}_h,\pmb{\mathsf{T}}_h)_{\Omega}\coloneqq\langle
		\llbracket\boldsymbol{w}_h\otimes{\mathbf{n}}\rrbracket,\{\!\!\{\pmb{\mathsf{T}}_h\}\!\!\}\rangle_{\Gamma_h}\quad
		\text { for all } \pmb{\mathsf{T}}_h\in \mathbb{P}^{k-1}(\mathcal{T}_h)^{n\times d}\,.
	\end{align*}
	Then, for every $k\in \mathbb{N}_0$, the \textit{DG gradient}
	$  \pmb{\mathsf{\mathcal{G}}}_h\colon W^{1,p}(\mathcal{T}_h)^n \to
	L^p(\Omega)^{n\times d}$, for every ${\boldsymbol{w}_h\in W^{1,p}(\mathcal{T}_h)^n}$, is defined via
	\begin{align*}
		\pmb{\mathsf{\mathcal{G}}}_h \boldsymbol{w}_h\coloneqq \nabla_h \boldsymbol{w}_h -\boldsymbol{\mathcal{R}}_h \boldsymbol{w}_h\quad\text{ in }L^p(\Omega)^{n\times d}\,.
	\end{align*}
	Owing to \cite[{(A.26)--(A.28)}]{dkrt-ldg}, there exists a constant $c>0$ such that $\boldsymbol{w}_h\in W^{1,p}(\mathcal{T}_h)^n$, it holds that
	\begin{align}\label{eq:eqiv0}
		c^{-1}\,\|\boldsymbol{w}_h\|_{h,p}\leq \big(\|\pmb{\mathsf{\mathcal{G}}}_h\boldsymbol{w}_h\|_{p,\Omega}^p+\big\|h^{-1/p'}_\Gamma\llbracket{\boldsymbol{w}_h\otimes\boldsymbol{n}}\rrbracket\big\|_{p,\Gamma_h}^p\big)^{1/p}\leq c\,\|\boldsymbol{w}_h\|_{h,p}\,.
	\end{align}
	
	For a generalized N-function $\psi\colon \Omega\times
	\ensuremath{\mathbb{R}}^{\ge 0}\to \ensuremath{\mathbb{R}}^{\ge 0}$, the pseudo-modular\footnote{The definition of a pseudo-modular can be found in \cite{Mu}. We extend the notion of DG Sobolev spaces to DG Sobolev--Orlicz spaces $W^{1,\psi }(\mathcal{T}_h)\coloneqq \{w_h\in L^1(\Omega)\mid w_h\in W^{1,1}(K)\allowbreak\text{with }\psi(\cdot,\nabla (w_h|_K))\in L^1(K)\text{ for all }K\in \mathcal{T}_K\}$.} 
	$m_{\psi,h}\colon\allowbreak W^{1,\psi}(\mathcal T_h)^n\to \mathbb{R}^{\ge 0}$,  for every $\boldsymbol{w}_h\in W^{1,\psi}(\mathcal T_h)^n$ %\textcolor{blue}{(has this space been defined? Could we just write $W^{1,p}?$)},
	is defined via\enlargethispage{4mm}
	\begin{align*} 
         % m_{\psi,h}(\bw_h)
         % &\coloneqq h\,\rho_{\psi,\smash{\Gamma_h}}(h^{-1}\jump{\bw_h\otimes
       %     \bfn})\,.
      %    \\
         m_{\psi,h}(\boldsymbol{w}_h)&\coloneqq \smash{\sum_{F \in \Gamma_h}} h_F\int
                            _F\psi(h_F^{-1}\llbracket{\boldsymbol{w}_h\otimes
                            {\mathbf{n}}}\rrbracket)\, \textup{d}s \,. 
	\end{align*}
	For $\psi = \varphi_{p,0}$, it holds that
	$m_{\psi,h}(\boldsymbol{w}_h)=\|h^{-1/p'}_\Gamma\llbracket{\boldsymbol{w}_h\otimes
		{\mathbf{n}}}\rrbracket\|_{p,\Gamma_h}^p $~for~all~${\boldsymbol{w}_h\in W^{1,\psi}(\mathcal T_h)^n}$.

	\section{Construction of moment-preserving operators}\label{sec:smoothing_operator}
	
	We employ here the smoothing operator
	$\pmb{\mathsf{\mathcal{E}}}_h    \coloneqq  \pmb{\mathsf{\mathcal{E}}}_{h,k} \colon
	C^0(\mathcal{T}_h)^n \to {\mathbb{P}}^{k+d}(\mathcal{T}_h)^n\cap
	W^{1,1}_0(\Omega)^n$, $k \in \ensuremath{\mathbb{N}}$, %($k \in \setN$ is the polynomial degree of $V_h$ \textcolor{blue}{$V_h$ hasn't been defined yet; should we move this?}) 
	introduced in \cite{VZ.2018.III}.
	The operator,  for every $\boldsymbol{w}_h\in C^0(\mathcal{T}_h)^n$, is constructed as:
    \begin{alignat*}{2}
		\pmb{\mathsf{\mathcal{E}}}_h \boldsymbol{w}_h &\coloneqq  \pmb{\mathsf{\mathcal{A}}}_h \boldsymbol{w}_h  + \pmb{\mathsf{\mathcal{B}}}_h(\boldsymbol{w}_h  -\pmb{\mathsf{\mathcal{A}}}_h \boldsymbol{w}_h )
        \qquad&\text{in }{\mathbb{P}}^{k+d}(\mathcal{T}_h) \cap W^{1,p}_0(\Omega)^n\,, \\
		\shortintertext{where}
		\pmb{\mathsf{\mathcal{B}}}_h \boldsymbol{w}_h &\coloneqq  \pmb{\mathsf{\mathcal{B}}}^{(1)}_h\boldsymbol{w}_h + \pmb{\mathsf{\mathcal{B}}}^{(2)}_h(\boldsymbol{w}_h  - \pmb{\mathsf{\mathcal{B}}}^{(1)}_h\boldsymbol{w}_h)
        \qquad&\text{in }{\mathbb{P}}^{k+d}(\mathcal{T}_h) \cap W^{1,p}_0(\Omega)^n\,,
    \end{alignat*}
    \vskip-\belowdisplayskip\vskip\belowdisplayshortskip\noindent
    and
    \begin{align*}
        \SwapAboveDisplaySkip
		\pmb{\mathsf{\mathcal{A}}}_h&\coloneqq \pmb{\mathsf{\mathcal{A}}}_{h,k} \colon
		C^0(\mathcal{T}_h)^n\to {\mathbb{P}}^k(\mathcal{T}_h)^n \cap
		W^{1,1}_0(\Omega)^n\,,\\
		\pmb{\mathsf{\mathcal{B}}}_h^{(1)}&\coloneqq\pmb{\mathsf{\mathcal{B}}}_{h,k} ^{(1)} \colon
		C^0(\mathcal{T}_h)^n \to {\mathbb{P}}^{k+d-1}(\mathcal{T}_h)^n \cap
		W^{1,1}_0(\Omega)^n\,,\\
		\pmb{\mathsf{\mathcal{B}}}_h^{(2)}&\coloneqq\pmb{\mathsf{\mathcal{B}}}_{h,k}^{(2)} \colon
		C^0(\mathcal{T}_h)^n\to {\mathbb{P}}^{k+d}(\mathcal{T}_h)^n \cap
		W^{1,1}_0(\Omega)^n
	\end{align*}
	%	${\tens{\mathcal{A}}_h\hspace*{-0.1em}\coloneqq\hspace*{-0.1em}\tens{\mathcal{A}}_{h,k}    \colon 
		%C^0(\mathcal{T}_h)^n\hspace*{-0.1em} \to \hspace*{-0.1em}\bbP^k(\tria)^n \cap
		%	W^{1,1}_0(\Omega)^n}$,  $\fB_h^{(1)}\hspace*{-0.1em}\coloneqq \hspace*{-0.1em}\tens{\mathcal{B}}_{h,k} ^{(1)}  \colon
	%	C^0(\mathcal{T}_h)^n \hspace*{-0.1em}\to\hspace*{-0.1em} \bbP^{k+d-1}(\tria)^n \cap
	%	W^{1,1}_0(\Omega)^n$, and $\fB_h^{(2)}\hspace*{-0.1em}\coloneqq\hspace*{-0.1em}\fB_{h,k}^{(2)} \colon
	%	C^0(\mathcal{T}_h)^n\hspace*{-0.1em} \to\hspace*{-0.1em} \bbP^{k+d}(\tria)^n \cap
	%	W^{1,1}_0(\Omega)^n$   
	are defined via:
	\begin{itemize}
		\item \textbf{$\pmb{\mathsf{\mathcal{A}}}_h$ (simplified nodal averaging with 1 dof):} Denote by $\mathcal{L}^{\mathrm{int}}_k(\mathcal{T}_h)$ the set of interior Lagrange node on $\mathcal{T}_h$ of degree $k$ and for every $y\in  \mathcal{L}^{\mathrm{int}}_k(\mathcal{T}_h)$ by $\varphi_y^k\in \mathbb{P}^k(\mathcal{T}_h)$,~the~\mbox{corresponding} basis function. In~addition, for every 
		$y \hspace*{-0.1em} \in\hspace*{-0.1em} \mathcal{L}^{\mathrm{int}}_k(\mathcal{T}_h)$  fix an arbitrary  $K_y \hspace*{-0.1em} \in\hspace*{-0.1em}  \mathcal{T}_h$~such~that~$y \hspace*{-0.1em} \in\hspace*{-0.1em}  K_y $. Then, the operator $\pmb{\mathsf{\mathcal{A}}}_h \hspace*{-0.15em} \colon \hspace*{-0.175em}
		C^0(\mathcal{T}_h)^n\hspace*{-0.175em} \to \hspace*{-0.175em}{\mathbb{P}}^k(\mathcal{T}_h)\hspace*{-0.15em} \cap\hspace*{-0.15em}
		W^{1,1}_0(\Omega)^n$, for every $\boldsymbol{w}_h\hspace*{-0.175em}\in\hspace*{-0.175em} C^0(\mathcal{T}_h)^n$, is~\mbox{defined}~via
		\begin{align*}
			\pmb{\mathsf{\mathcal{A}}}_h  \boldsymbol{w}_h \coloneqq  \sum_{y \in \mathcal{L}^{\mathrm{int}}_k(\mathcal{T}_h)} (\boldsymbol{w}_h|_{K_y })(y ) \varphi_y ^k\quad\text{  in }{\mathbb{P}}^k(\mathcal{T}_h)^n \cap W^{1,1}_0(\Omega)^n\,.
		\end{align*}
		\item \textbf{$\pmb{\mathsf{\mathcal{B}}}_h^{(1)}$ (facet bubble smoother):} Denote for
		$F\in \Gamma_h$ by $\mathcal{L}_{k-1}(F)$ the set of Lagrange nodes on $F$~of~degree $k-1$ and by $\varphi_F \coloneqq \varphi_{y_1}^1\cdot \ldots\cdot\varphi_{y_d}^1 \in \mathbb{P}^d(F)\cap W^{1,1}_0(F)$, where $y_1,\dots,y_d\in \mathcal{L}_1(F)$ are such that
		$F=\textup{conv}\{y_1,\dots, y_d\}$, the corresponding facet bubble function.
		Then, the operator $\pmb{\mathsf{\mathcal{B}}}_h^{(1)} \colon 
		C^0(\mathcal{T}_h)^n \to {\mathbb{P}}^{k+d-1}(\mathcal{T}_h)^n \cap
		W^{1,1}_0(\Omega)^n$, for every $\boldsymbol{w}_h\in  C^0(\mathcal{T}_h)^n$,~is~\mbox{defined}~via
		\begin{align*}
			\pmb{\mathsf{\mathcal{B}}}^{(1)}_h\boldsymbol{w}_h \coloneqq  
            \smashoperator[l]{\sum_{F \in \Gamma_h^{i}}}
            \smashoperator[r]{\sum_{y \in \mathcal{L}_{k-1}(F)}}
			(\boldsymbol{\mathcal{Q}}_F \{\!\!\{ \boldsymbol{w}_h \}\!\!\})(y  )\, \varphi^{k-1}_y \, \varphi_F\quad\text{  in }{\mathbb{P}}^{k+d-1}(\mathcal{T}_h) ^n\cap W^{1,1}_0(\Omega)^n\,,
		\end{align*}
		where $\boldsymbol{\mathcal{Q}}_F\colon C^0(\mathcal{T}_h)^n \to \mathbb{P}^{k-1}(F)^n$ is the weighted
		$L^2$-projection, for every $\boldsymbol{w}_h\in C^0(\mathcal{T}_h)^n$ and $\boldsymbol{z}_h\in  \mathbb{P}^{k-1}(F)^n$, defined via
		\begin{align*}
			(\varphi_F\boldsymbol{\mathcal{Q}}_F \boldsymbol{w}_h, \boldsymbol{z}_h)_F =( \boldsymbol{v}_h, \boldsymbol{z}_h)_F \,.
		\end{align*}
		\item 
		\textbf{$\pmb{\mathsf{\mathcal{B}}}_h^{(2)}$ (interior bubble smoother):}  Denote for $K\hspace*{-0.1em}\in\hspace*{-0.1em} \mathcal{T}_h$ by $\mathcal{L}_k(K)$ the set of Lagrange~nodes on $K$ of degree $k$ 
		and by $\varphi_K\coloneqq \varphi_{y_1}^1\cdot \ldots\cdot\varphi_{y_{d+1}}^1 \hspace*{-0.15em}\in \hspace*{-0.1em}\mathbb{P}^{d+1}(K)\cap W^{1,1}_0(K)$, where $y_1,\dots,y_{d+1}\in \mathcal{L}_1(K)$ are such that
		$K=\textup{conv}\{y_1,\ldots, y_{d+1}\}$, the corresponding bubble function.
		\!Then,~the operator $\pmb{\mathsf{\mathcal{B}}}_h^{(2)} \colon 
		C^0(\mathcal{T}_h)^n \to {\mathbb{P}}^{k+d}(\mathcal{T}_h)^n \cap
		W^{1,1}_0(\Omega)^n$, for every $\boldsymbol{w}_h\in C^0(\mathcal{T}_h)^n$,~is~defined~via
		\begin{align*}
			\pmb{\mathsf{\mathcal{B}}}^{(2)}_h \boldsymbol{w}_h \coloneqq  \sum_{K \in \mathcal{T}_h} (\boldsymbol{\mathcal{Q}}_K \boldsymbol{w}_h)\, \varphi_K\,,
		\end{align*}
		where ${\boldsymbol{\mathcal{Q}}_K\colon C^0(\mathcal{T}_h)^n\to \mathbb{P}^{k-1}(K)^n}$ is the weighted $L^2$-projection,
		for every $\boldsymbol{w}_h\in C^0(\mathcal{T}_h)^n$ and $\boldsymbol{z}_h\in  \mathbb{P}^{k-1}(K)^n$, defined via
		\begin{align*}
			(\varphi_K\boldsymbol{\mathcal{Q}}_K \boldsymbol{w}_h, \boldsymbol{z}_h)_K =( \boldsymbol{w}_h, \boldsymbol{z}_h)_K\,.
		\end{align*}
		Note that $\pmb{\mathsf{\mathcal{B}}}^{(2)}_h$ is not needed when $k=1$.
	\end{itemize}

	The idea in all these cases is that the bubble smoother
	helps with the preservation of moments, but as the bubble
	smoothers are not stable, a nodal averaging operator
	is added~to~recover~stability. Namely, \cite[(3.17),
	(3.18)]{VZ.2018.III} and
	$\left| \llbracket{\boldsymbol{v}_h}\rrbracket \right|=\left| \llbracket{\boldsymbol{v}_h \otimes \boldsymbol{n}}\rrbracket \right|$, for
	every $\boldsymbol{v}_h\in {\mathbb{P}}^k(\mathcal{T}_h)^n$, imply that
	\begin{align}\label{eq:E_stability_L2}
		\|\nabla_h(\boldsymbol{v}_h - \pmb{\mathsf{\mathcal{E}}}_h  \boldsymbol{v}_h )\|_{2,K}
		\lesssim  \sum_{F\in \Gamma_h(K)}{\|h_\Gamma^{-1/2} \llbracket{\boldsymbol{v}_h \otimes \boldsymbol{n}}\rrbracket_F\|_{2,F}}\,,
	\end{align}
	where $\Gamma_h(K)\coloneqq \{F\in \Gamma_h\mid K\cap
	F\neq\emptyset\}$ and the constants depend only on $k$
	and $\omega_0$. We now proceed to generalise the estimate \cref{eq:E_stability_L2} to the Orlicz setting.
	
	\begin{Proposition}[Interpolation estimate]\label{prop:n-function_E}
		Let $\psi\colon \mathbb{R}^{\ge 0}\to\mathbb{R}^{\ge
			0} $ be an N-function with $\psi\in \Delta_2$,
		$\psi^*\in \Delta_2$, and let $k\in \mathbb{N}$. Then,  for every $\boldsymbol{v}_h\in \mathbb{P}^k(K)^n$ and $K\in \mathcal{T}_h$, it holds that
		\begin{align*}
			\frac{1}{\vert K\vert }\int_K{\psi(h_K\vert \tilde{\nabla}_h (\boldsymbol{v}_h-\pmb{\mathsf{\mathcal{E}}}_h  \boldsymbol{v}_h)\vert)\,\textup{d}x}\lesssim  \sum_{F\in \Gamma_h(K)}{\frac{1}{\vert F\vert }\int_F{\psi(\vert\llbracket{\boldsymbol{v}_h \otimes \boldsymbol{n}}\rrbracket_F\vert)\, \textup{d}s}}\,,
		\end{align*}
		where either $\tilde{\nabla}_h= {\nabla}_h$ or
		$\tilde{\nabla}_h =\pmb{\mathsf{\mathcal{G}}}_h$ and the constants depend only on $k$, $\Delta_2(\psi),\Delta_2(\psi^*)$ and $\omega_0$.
	\end{Proposition}

	\begin{proof}
		We need to distinguish the cases $\tilde{\nabla}_h=\nabla_h$ and $\tilde{\nabla}_h=\pmb{\mathsf{\mathcal{G}}}_h$:
		
		\textit{Case $\tilde{\nabla}_h=\nabla_h$.}
		Owing to \cref{eq:E_stability_L2} together with \cite[Lem.\ 12.1]{EG.2021}, we have that
		\begin{align}
			h_K\| \nabla (\boldsymbol{v}_h-\pmb{\mathsf{\mathcal{E}}}_h  \boldsymbol{v}_h)\|_{\infty,K}\lesssim \sum_{F\in \Gamma_h(K)}{\frac{1}{\vert  F\vert }\int_F{\vert \llbracket{\boldsymbol{v}_h \otimes \boldsymbol{n}}\rrbracket_F\vert\,\textup{d}s}}\,,\label{prop:n-function_E.1}
		\end{align}
		where the constants depends only on $k$ and $\omega_0$. 
		Using  \cref{prop:n-function_E.1}, the $\Delta_2$-condition and
		convexity of $\psi$, in particular, Jensen's inequality, and that $\sup_{h>0}\sup_{K\in \mathcal{T}_h}{\textup{card}(\Gamma_h(K))}\allowbreak\leq c$, where $c>0$ depends only on $\omega_0$,
		we find that
		\begin{align}\label{prop:n-function_E.2}
%			\begin{aligned}
				\frac{1}{\vert K\vert }\int_K{\psi(h_K\vert \nabla (\boldsymbol{v}_h-\pmb{\mathsf{\mathcal{E}}}_h  \boldsymbol{v}_h)\vert)\,\textup{d}x}
				&\lesssim 
				\psi\bigg(\frac{1}{\textup{card}(\Gamma_h(K))}\sum_{F\in\Gamma_h(K)}{\frac{1}{\vert F\vert }\int_F{\vert \llbracket{\boldsymbol{v}_h \otimes \boldsymbol{n}}\rrbracket_F\vert\,\textup{d}s}}\bigg)\notag
				\\&\lesssim \frac{1}{\textup{card}(\Gamma_h(K))}\sum_{F\in\Gamma_h(K)}{\frac{1}{\vert F\vert }\int_F{\psi(\vert \llbracket{\boldsymbol{v}_h \otimes \boldsymbol{n}}\rrbracket_F\vert)\,\textup{d}s}}
				\\&\lesssim \sum_{F\in\Gamma_h(K)}{\frac{1}{\vert F\vert }\int_F{\psi(\vert \llbracket{\boldsymbol{v}_h \otimes \boldsymbol{n}}\rrbracket_F\vert)\,\textup{d}s}}\,,\notag
%			\end{aligned}
		\end{align}
		where the constants depend only on $k$, $\Delta_2(\psi),\Delta_2(\psi^*)$ and $\omega_0$.
		
		\textit{Case $\tilde{\nabla}_h=\pmb{\mathsf{\mathcal{G}}}_h$.} 
		Appealing to \cite[(A.23)]{dkrt-ldg},  for every $F\in \Gamma_h$ and $K\in \mathcal{T}_h$ with $F\subseteq \partial K$, it holds that
		\begin{align}
			\frac{1}{\vert K\vert }\int_K{\psi(h_K\vert \boldsymbol{\mathcal{R}}_h \boldsymbol{v}_h\vert)\,\textup{d}x}\lesssim \frac{1}{\vert F\vert }\int_F{\psi(\vert\llbracket{\boldsymbol{v}_h \otimes \boldsymbol{n}}\rrbracket_F\vert)\,\textup{d}s}\,,\label{prop:n-function_E.3}
		\end{align}
		where the constants depend only on $k$, $\Delta_2(\psi),\Delta_2(\psi^*)$ and $\omega_0$.
		Thus, using \cref{prop:n-function_E.2}, that ${\boldsymbol{\mathcal{R}}_h \pmb{\mathsf{\mathcal{E}}}_h  \boldsymbol{v}_h=\boldsymbol{0}}$, and \cref{prop:n-function_E.3}, we obtain
		\begin{align*}
			\frac{1}{\vert K\vert }\int_K{\psi(h_K\vert \tilde{\nabla}_h (\boldsymbol{v}_h-\pmb{\mathsf{\mathcal{E}}}_h  \boldsymbol{v}_h)\vert)\,\textup{d}x}&\lesssim \frac{1}{\vert K\vert }\int_K{\psi(h_K\vert \nabla (\boldsymbol{v}_h-\pmb{\mathsf{\mathcal{E}}}_h  \boldsymbol{v}_h)\vert)\,\textup{d}x}\\&\quad +\frac{1}{\vert K\vert }\int_K{\psi(h_K\vert \boldsymbol{\mathcal{R}}_h \boldsymbol{v}_h\vert)\,\textup{d}x}\\&\lesssim \sum_{F\in\Gamma_h(K)}{\frac{1}{\vert F\vert }\int_F{\psi(\vert \llbracket{\boldsymbol{v}_h \otimes \boldsymbol{n}}\rrbracket_F\vert)\,\textup{d}s}}\,,
		\end{align*}
		where the constants depend only on $k$, $\Delta_2(\psi),\Delta_2(\psi^*)$ and $\omega_0$.
	\end{proof}

	\section{Primal formulation}\label{sec:primal}
	
	In this section, we examine the primal formulation of the non-linear problem \cref{eq:PDE} and its discretization. 	In doing so, throughout the entire section, we assume that $\pmb{\mathsf{\mathcal{S}}}$ has $(p,\delta)$-structure in the sense of \cref{assum:extra_stress}.
	
	\subsection{Continuous primal formulation} We abbreviate
	\begin{align*}
		V\coloneqq W^{1,p}_0(\Omega)^n\,.
	\end{align*}	
	Then, for given right-hand side $\boldsymbol{f}\in V^*$, the \textit{primal formulation} of \cref{eq:PDE} seeks for a function (or vector field, respectively) $\boldsymbol{u}\in V$ such that for every $\boldsymbol{z}\in V$, it holds that
	\begin{align}\label{eq:primal}
		(\pmb{\mathsf{\mathcal{S}}}(\nabla\boldsymbol{u} ), \nabla \boldsymbol{z})_{\Omega}=\langle \boldsymbol{f},  \boldsymbol{z} \rangle_V\,.
	\end{align}
	Resorting to the celebrated theory of monotone operators, cf.\ \cite{zei-IIB,ru-fa}, it is readily seen that the primal formulation admits a unique solution.
	
	\subsection{Discrete primal formulation}
	For given $k\in \mathbb{N}$, we abbreviate
	\begin{align*}
		V_h\coloneqq V_h^k\coloneqq \mathbb{P}^k(\mathcal T_h)^n\,.
	\end{align*}
	Then, for given right-hand side $\boldsymbol{f}\in V^*$, the \textit{discrete primal formulation} of \cref{eq:PDE} seeks for a discrete function (or vector field, respectively) $\boldsymbol{u}_h\in V_h$ such that for every $\boldsymbol{z}_h\in V_h$, it holds that
	\begin{align}\label{eq:discrete_primal_DG}
		(\pmb{\mathsf{\mathcal{S}}}(\tilde{\nabla}_h \boldsymbol{u}_h),  \nabla \pmb{\mathsf{\mathcal{E}}}_h  \boldsymbol{z}_h)_\Omega
		+ \alpha\, \langle \pmb{\mathsf{\mathcal{S}}}_{\boldsymbol{\beta}_h(\boldsymbol{u}_h)} ( h_\Gamma^{-1} \llbracket{\boldsymbol{u}_h\otimes \boldsymbol{n}}\rrbracket ), \llbracket{\boldsymbol{z}_h \otimes \boldsymbol{n}}\rrbracket\rangle_{\Gamma_h}
		=
		\langle \boldsymbol{f}, \pmb{\mathsf{\mathcal{E}}}_h  \boldsymbol{z}_h \rangle_V
		\,,
	\end{align}
	where $\alpha\hspace*{-0.15em}>\hspace*{-0.15em}0$ is the \textit{stabilisation parameter},
	$\boldsymbol{\beta}_h\colon \hspace*{-0.15em}V_h\hspace*{-0.15em}\to\hspace*{-0.15em} \mathbb{R}^{\ge 0}$ is the
	\textit{max-shift functional},~for~\mbox{every}~${\boldsymbol{v}_h\hspace*{-0.15em}\in\hspace*{-0.15em}
	V_h}$,  defined via 
	\begin{align}\label{def_shift}
		\boldsymbol{\beta}_h(\boldsymbol{v}_h) \coloneqq  \begin{cases}
			0 & \text{if }p\leq 2\,,\\
			\|\tilde{\nabla}_h \boldsymbol{v}_h\|_{\infty,\Omega} & \text{if }p>2\,,
		\end{cases}
	\end{align}
	with $\tilde{\nabla}_h\in \{ \pmb{\mathsf{\mathcal{G}}}_h,\nabla_h\}$, and         $\pmb{\mathsf{\mathcal{S}}}_{\boldsymbol{\beta}_h(\boldsymbol{u}_h)}$ is defined  in \cref{eq:flux}. More precisely, 
	setting $\tilde{\nabla}_h = \pmb{\mathsf{\mathcal{G}}}_h$, results in an LDG-type
	method and setting $\tilde{\nabla}_h=\nabla_h$, results is an
	IIDG-type method.
	
	\begin{Remark}\label{rem:grad_tilde}
		Note that for each choice $\tilde{\nabla}_h\in
		\{\pmb{\mathsf{\mathcal{G}}}_h,\nabla_h\}$ and for every $\boldsymbol{w}_h\in C^0(\mathcal{T}_h)^n$,~we~have~that $\tilde{\nabla}_h\pmb{\mathsf{\mathcal{E}}}_h\boldsymbol{w}_h\!=\!\nabla \pmb{\mathsf{\mathcal{E}}}_h\boldsymbol{w}_h$.
		\!This enables us to replace the classical gradient in the primal formulation~\cref{eq:discrete_primal_DG}, by  $\tilde{\nabla}_h\in \{\pmb{\mathsf{\mathcal{G}}}_h,\nabla_h\}$, which is of crucial importance in the proof of a best-approximation type result in \cref{thm:error_primal}; see also \cref{lem:max-shift}.
	\end{Remark}
	
	\begin{Remark}
		It can be seen that if we add to the left-hand side of
                the LDG discretisation~the~term
                $(\pmb{\mathsf{\mathcal{S}}}(\boldsymbol{\mathcal{R}}_h \boldsymbol{u}_h),
                \nabla_h(\pmb{\mathsf{\mathcal{E}}}_h  \boldsymbol{z}_h - \boldsymbol{z}_h))_{\Omega}$, then we
                arrive at a SIP-like method (in that it reduces to the
                usual Symmetric Interior Penalty method in the linear
                case). Thus, all results proved for~the~LDG~method and the
                IIDG method can be carried over to this method.
	\end{Remark}

	\begin{Remark}\label{rem:motivation_shift}
		The max-shift functional is constructed in such a way that
		is has the following two properties. First, for the globally
		non-constant shift $\vert \tilde{\nabla}_h\boldsymbol{u}_h\vert \in L^p(\Omega)^{n\times
			d}$ and the globally constant shift $\boldsymbol{\beta}_h(\boldsymbol{u}_h)\in \ensuremath{\mathbb{R}}$, it holds that
		$\varphi_{\vert \tilde{\nabla}_h\boldsymbol{u}_h\vert} (t)\le
		\varphi_{\boldsymbol{\beta}_h(\boldsymbol{u}_h)}(t)$ for all $t\ge 0$. To the latter N-function, the approximation
		properties of the smoothing operator $\pmb{\mathsf{\mathcal{E}}}_h \colon V_h\to V$
		apply, cf.\ \cref{eq:max-shift.2} (below). Second, for the
		conjugate N-functions, it holds that $          (\varphi_{\boldsymbol{\beta}_h(\boldsymbol{u}_h)})^*(t)\le (\varphi_{\vert \tilde{\nabla}_h\boldsymbol{u}_h\vert})^* (t)
		$ for all $t\ge 0$. For the latter N-function, the
		equivalences \cref{eq:hammera} apply, cf.\
		\cref{eq:max-shift.3} (below). 
	\end{Remark}
	
	More precisely, the max-shift functional is precisely constructed in such a way that the following lemma applies.
	
	\begin{Lemma}\label{lem:max-shift}
		For every $\kappa>0$, there exists constant
		$c_\kappa>0$, depending  on the characteristics~of~$\pmb{\mathsf{\mathcal{S}}}$, such that for every $\boldsymbol{v}_h,\boldsymbol{z}_h\in V_h$ and $\boldsymbol{v}\in V$, it holds that
		\begin{align*}
			&\vert (\pmb{\mathsf{\mathcal{S}}}(\tilde{\nabla}_h \boldsymbol{v}_h) - \pmb{\mathsf{\mathcal{S}}}(\nabla \boldsymbol{v}), \tilde{\nabla}_h(\pmb{\mathsf{\mathcal{E}}}_h  \boldsymbol{z}_h - \boldsymbol{z}_h))_{\Omega}\vert \\
            &\lesssim \kappa \,\|\pmb{\mathsf{\mathcal{F}}}(\tilde{\nabla}_h \boldsymbol{v}_h) - \pmb{\mathsf{\mathcal{F}}}(\nabla \boldsymbol{v})\|^2_{2,\Omega
			}+
			c_\kappa\, m_{\varphi_{\boldsymbol{\beta}_h(\boldsymbol{v}_h)},h}(\boldsymbol{z}_h)\,,
		\end{align*}
		where the constants depend only on $k$, $\omega_0$, and the characteristics of $\pmb{\mathsf{\mathcal{S}}}$.
	\end{Lemma}
	
	\begin{proof}
		Using \cref{eq:hammere}  and the $\varepsilon$-Young  inequality
		\cref{ineq:young} with $\psi=\varphi_{\boldsymbol{\beta}_h(\boldsymbol{v}_h)}$, we find that
		\begin{align}
            \SwapAboveDisplaySkip
            \label{eq:max-shift.1}
			\begin{aligned}
				&\vert (\pmb{\mathsf{\mathcal{S}}}(\tilde{\nabla}_h \boldsymbol{v}_h) - \pmb{\mathsf{\mathcal{S}}}(\nabla \boldsymbol{v}), \tilde{\nabla}_h(\pmb{\mathsf{\mathcal{E}}}_h  \boldsymbol{z}_h - \boldsymbol{z}_h))_{\Omega}\vert\\ &\lesssim c_\kappa\,\rho_{\varphi_{\boldsymbol{\beta}_h(\boldsymbol{v}_h)},\Omega}(\tilde{\nabla}_h(\pmb{\mathsf{\mathcal{E}}}_h  \boldsymbol{z}_h - \boldsymbol{z}_h))
				+\kappa\, \rho_{(\varphi_{\boldsymbol{\beta}_h(\boldsymbol{v}_h)})^*,\Omega}(\varphi'_{\smash{|\tilde{\nabla}_h \boldsymbol{v}_h|}}(|\tilde{\nabla}_h  \boldsymbol{v}_h - \nabla \boldsymbol{v}|))\,.
			\end{aligned}
		\end{align}
		Due to $\boldsymbol{\beta}_h(\boldsymbol{u}_h)\in \mathbb{R}$,
		\cref{prop:n-function_E} is applicable and yields, also
                using $h_F\sim h_K$, that
		\begin{align}
            %\SwapAboveDisplaySkip
            \label{eq:max-shift.2}
			\rho_{\varphi_{\boldsymbol{\beta}_h(\boldsymbol{v}_h)},\Omega}(\tilde{\nabla}_h(\pmb{\mathsf{\mathcal{E}}}_h  \boldsymbol{z}_h - \boldsymbol{z}_h)) \lesssim m_{\varphi_{\boldsymbol{\beta}_h(\boldsymbol{u}_h)},h}(\boldsymbol{z}_h)\,.
		\end{align}
		By definition of the max-shift functional and
		\cref{rem:phi} (iii), we obtain
		$ (\varphi_{\boldsymbol{\beta}_h(\boldsymbol{v}_h)})^*\leq
		(\varphi_{\smash{\vert \tilde{\nabla}_h
				\boldsymbol{v}_h\vert}})^*$. Using this, the equivalence \cref{eq:psi'} for 
		$\psi=\varphi_{\smash{\vert \tilde{\nabla}_h
				\boldsymbol{v}_h\vert}}$,  and  \cref{eq:hammera}, we obtain
		\begin{align}\label{eq:max-shift.3}
			\begin{aligned}
				\rho_{(\varphi_{\boldsymbol{\beta}_h(\boldsymbol{v}_h)})^*,\Omega}(\varphi'_{\smash{|\tilde{\nabla}_h \boldsymbol{v}_h|}}(|\tilde{\nabla}_h  \boldsymbol{v}_h - \nabla \boldsymbol{v}|))&\leq \rho_{(\varphi_{|\tilde{\nabla}_h \boldsymbol{v}_h|})^*,\Omega}(\varphi'_{\smash{|\tilde{\nabla}_h \boldsymbol{v}_h|}}(|\tilde{\nabla}_h  \boldsymbol{v}_h - \nabla \boldsymbol{v}|))\\&\lesssim \|\pmb{\mathsf{\mathcal{F}}}(\tilde{\nabla}_h\boldsymbol{v}_h) - \pmb{\mathsf{\mathcal{F}}}(\nabla \boldsymbol{v})\|^2_{2,\Omega}\,.
			\end{aligned}
		\end{align}
		Eventually, combining \cref{eq:max-shift.2,eq:max-shift.3} in \cref{eq:max-shift.1}, we conclude the claimed estimate.
	\end{proof}

	In the linear case $\pmb{\mathsf{\mathcal{S}}}(\tilde{\nabla}_h \boldsymbol{u}) =
	\tilde{\nabla}_h \boldsymbol{u}$, non-degeneracy and existence of
	discrete solutions is a direct consequence of the construction
	of $\pmb{\mathsf{\mathcal{E}}}_h $, because then the preservation property
	\cref{eq:equivalence_without_E} means that we can forget the
	operator $\pmb{\mathsf{\mathcal{E}}}_h $ on the left-hand side and so the usual
	argument for proving coercivity works (note that, by
	construction, $\tilde{\nabla}_h \boldsymbol{u}_h \in {\mathbb{P}}^{k-1}(\mathcal{T}_h)^{n
		\times d}$). In the non-linear case this argument only works
	for the lowest order DG ansatz 
	for which the gradients are element-wise~constant.~Hence, in general, we would need to check that discrete solutions exist.

	\begin{Proposition} For $\alpha>0$ sufficiently large, 
		the primal formulation \cref{eq:discrete_primal_DG} admits a solution.
	\end{Proposition}
	
	\begin{proof}
          We equip $V_h$ with the norm
          $\|\cdot\|_{V_h }\coloneqq \|\cdot\|_{h,p}$, and consider
          the operator $\pmb{\mathsf{\mathcal{T}}}_h\colon V_h\to (V_h)^*$,
          for every $\boldsymbol{v}_h\in V_h $ and $\boldsymbol{z}_h\in V_h $, defined via
          \begin{align*}
            \langle \pmb{\mathsf{\mathcal{T}}}_h\boldsymbol{v}_h,\boldsymbol{z}_h\rangle_{V_h}\coloneqq 
            (\pmb{\mathsf{\mathcal{S}}}(\tilde{\nabla}_h \boldsymbol{v}_h),\nabla \pmb{\mathsf{\mathcal{E}}}_h  \boldsymbol{z}_h)_{\Omega}
            + \alpha\, \langle \pmb{\mathsf{\mathcal{S}}}_{\boldsymbol{\beta}_h(\boldsymbol{v}_h)}( h_\Gamma^{-1} \llbracket{\boldsymbol{v}_h \otimes\boldsymbol{n}}\rrbracket ) , \llbracket{\boldsymbol{z}_h \otimes\boldsymbol{n}}\rrbracket\rangle_{\Gamma_h}\,.
          \end{align*}
          Since $V_h$ is finite dimensional, consists of broken
          polynomials, \cite[Lem.~3.18]{bdr-7-5}, the properties of
          $\pmb{\mathsf{\mathcal{S}}}$ and $\pmb{\mathsf{\mathcal{S}}}_a$, $a\ge 0$, imply that the operator
          $\pmb{\mathsf{\mathcal{T}}}_h\colon V_h\to (V_h)^*$, for every fixed
          $h>0$, is well-defined, continuous, and \mbox{monotone}.
          To prove the
          boundedness of $\pmb{\mathsf{\mathcal{T}}}_h\colon V_h\to (V_h)^*$,
          we use the properties of $\pmb{\mathsf{\mathcal{S}}}, \pmb{\mathsf{\mathcal{S}}}_{\boldsymbol{\beta}_h(\boldsymbol{v}_h)}$
          (cf.\ \cref{lem:hammer,rem:sa}),
          Young's inequality, \cref{eq:psi'},
          \cref{prop:n-function_E} with $\psi=\varphi$, the stability of
          $\boldsymbol{\mathcal{R}}_h$ (cf.~\cite[(A.25)]{dkrt-ldg}), a~shift change
          (cf.\ \cref{lem:shift-change}),
          $\sum_{F \in \Gamma_h}\!h_F\mathscr{H}^{d-1}(F)\!\sim\! \vert \Omega\vert$, 
          \cite[Lem- ma 12.1]{EG.2021},  Jensen's inequality, and  $\varphi(t)+t^p\sim
          t^p+ \delta^p$ for all $t\ge 0$ (cf.~\cite{bdr-7-5}), to
          find that 
          \begin{align*}
            \left| \langle \pmb{\mathsf{\mathcal{T}}}_h\boldsymbol{v}_h,\boldsymbol{z}_h\rangle_{V_h} \right|
            &\lesssim \rho_{\varphi,\Omega}(\tilde{\nabla}_h \boldsymbol{v}_h) +
              \rho_{\varphi,\Omega}(\tilde{\nabla}_h(
              \pmb{\mathsf{\mathcal{E}}}_h \boldsymbol{z}_h-\boldsymbol{z}_h)) +\rho_{\varphi,\Omega}(\tilde{\nabla}_h
              \boldsymbol{z}_h)
            \\
            &\quad + \alpha \, m_{\varphi_{\boldsymbol{\beta}_h(\boldsymbol{v}_h)},h}(\boldsymbol{v}_h)
              +\alpha \, m_{\varphi_{\boldsymbol{\beta}_h(\boldsymbol{v}_h)},h}(\boldsymbol{z}_h)
            \\
            &\lesssim \rho_{\varphi,\Omega}({\nabla}_h \boldsymbol{v}_h) +\rho_{\varphi,\Omega}({\nabla}_h
              \boldsymbol{z}_h) +m_{\varphi,h}(\boldsymbol{v}_h)
              + m_{\varphi,h}(\boldsymbol{z}_h) +\varphi (\boldsymbol{\beta}_h(\boldsymbol{v}_h))
            \\
            &\lesssim \|\boldsymbol{v}_h\|^p_{p,h} +\|
              \boldsymbol{z}_h\|^p_{p,h} + \delta^p\,.
          \end{align*}
         Thus, $\pmb{\mathsf{\mathcal{T}}}_h\colon V_h\to (V_h)^*$  is pseudo-monotone.  If we
		can prove its coercivity, then, from pseudo-monotone
		operator theory (cf.\ \cite[Thm.\ 27.A]{zei-IIB}), it follows that
		$\pmb{\mathsf{\mathcal{T}}}_h\colon V_h\to (V_h)^*$~is~\mbox{surjective}.  For every $\boldsymbol{v}_h\in V_h $, using first
		that
		$\nabla \pmb{\mathsf{\mathcal{E}}}_h \boldsymbol{v}_h=\tilde{\nabla}_h \pmb{\mathsf{\mathcal{E}}}_h \boldsymbol{v}_h$,
		then that
		$\pmb{\mathsf{\mathcal{S}}}(\pmb{\mathsf{Q}}):\pmb{\mathsf{Q}}\sim \varphi(\vert
		\pmb{\mathsf{Q}}\vert )$ and
		${\pmb{\mathsf{\mathcal{S}}}_{\boldsymbol{\beta}_h(\boldsymbol{v}_h)} (\pmb{\mathsf{Q}}):\pmb{\mathsf{Q}}\sim
			\varphi_{\boldsymbol{\beta}_h(\boldsymbol{v}_h)}(\vert \pmb{\mathsf{Q}}\vert
			)}$ for~all~${\pmb{\mathsf{Q}}\in \mathbb{R}^{n\times d}}$
		(cf.\ \cref{lem:hammer,rem:sa}), $\varphi_{\boldsymbol{\beta}_h(\boldsymbol{v}_h)}\ge \varphi$
		(cf.\ \cref{rem:phi} (iii)), the
		$\varepsilon$-Young inequality \cref{ineq:young} with
		$\psi=\varphi$, and
		\cref{prop:n-function_E} with $\psi=\varphi$, we find that
		\begin{align*}
			\langle \pmb{\mathsf{\mathcal{T}}}_h\boldsymbol{v}_h,\boldsymbol{v}_h\rangle_{V_h} &=(\pmb{\mathsf{\mathcal{S}}}(\tilde{\nabla}_h\boldsymbol{v}_h),\tilde{\nabla}_h \boldsymbol{v}_h)_{\Omega}+ \alpha\, \langle \pmb{\mathsf{\mathcal{S}}}_{\boldsymbol{\beta}_h(\boldsymbol{v}_h)}( h_\Gamma^{-1} \llbracket{ \boldsymbol{v}_h \otimes\boldsymbol{n}}\rrbracket ) , \llbracket{ \boldsymbol{v}_h \otimes\boldsymbol{n}}\rrbracket\rangle_{\Gamma_h}
			\\&\quad + (\pmb{\mathsf{\mathcal{S}}}(\tilde{\nabla}_h \boldsymbol{v}_h),\tilde{\nabla}_h (\pmb{\mathsf{\mathcal{E}}}_h \boldsymbol{v}_h- \boldsymbol{v}_h))_{\Omega}
			\\&	\ge (c-\varepsilon)\, \rho_{\varphi,\Omega}(\tilde{\nabla}_h \boldsymbol{v}_h)+\alpha \,c\, m_{\varphi_{\boldsymbol{\beta}_h(\boldsymbol{v}_h)},h}(\boldsymbol{v}_h)-c_\varepsilon\, \rho_{\varphi,\Omega}(\tilde{\nabla}_h (\pmb{\mathsf{\mathcal{E}}}_h \boldsymbol{v}_h- \boldsymbol{v}_h))
			\\&	\ge (c-\varepsilon)\, \rho_{\varphi,\Omega}(\tilde{\nabla}_h \boldsymbol{v}_h)+(\alpha \,c-c_\varepsilon)\, m_{\varphi,h}(\boldsymbol{v}_h)\,.
		\end{align*}
		Consequently, choosing first $\varepsilon>0$
		sufficiently small and, subsequently, $\alpha>0$
		sufficiently large, also  using $\varphi(t)+t^p\sim
		t^p+ \delta^p$ for all $t\ge 0$ (cf.~\cite{bdr-7-5}), $h\,\mathscr{H}^{d-1}(\Gamma_h)\sim \vert \Omega\vert$, 
		for every $\boldsymbol{v}_h\in V_h $, and the norm equivalence \cref{eq:eqiv0}, we arrive at
		\begin{align*}
			\langle \pmb{\mathsf{\mathcal{T}}}_h\boldsymbol{v}_h,\boldsymbol{v}_h\rangle_{V_h}&\gtrsim \|\tilde{\nabla}_h \boldsymbol{v}_h\|_{p,\Omega}^p+\alpha \, \|h_{\Gamma}^{\smash{-1/p'}}\llbracket{\boldsymbol{v}_h\otimes \boldsymbol{n}}\rrbracket\|_{p,\Gamma_h}^p-\delta^p\,\vert \Omega\vert\,(1+\alpha)\\&\gtrsim 
			\min\{1,\alpha\}\,\| \boldsymbol{v}_h\|_{V_h }^p-\delta^p\,\vert \Omega\vert\,\max\{1,\alpha\}\,.
		\end{align*}
		Putting everything together, we proved that the operator
		$\pmb{\mathsf{\mathcal{T}}}_h\colon V_h\!\to\! (V_h)^*$ is
		well-defined, bounded, pseudo-monotone, and coercive and, thus, surjective (cf.\ \cite[Thm.\ 27.A]{zei-IIB}). 
	\end{proof}

	From \cref{eq:primal,eq:discrete_primal_DG}, it follows that
	the error equation, for every $\boldsymbol{z}_h\in V_h$, takes the form
	\begin{align}\label{eq:error_equation_primal}
		(\pmb{\mathsf{\mathcal{S}}}(\tilde{\nabla}_h \boldsymbol{u}_h) - \pmb{\mathsf{\mathcal{S}}}(\nabla \boldsymbol{u}),\nabla \pmb{\mathsf{\mathcal{E}}}_h  \boldsymbol{z}_h)_{\Omega}
		+ \alpha\, \langle \pmb{\mathsf{\mathcal{S}}}_{\boldsymbol{\beta}_h(\boldsymbol{u}_h)}(h_\Gamma^{-1} \llbracket{\boldsymbol{u}_h\otimes \boldsymbol{n}}\rrbracket), \llbracket{\boldsymbol{z}_h \otimes \boldsymbol{n}}\rrbracket\rangle_{\Gamma_h}
		= 0\,.
	\end{align}

	\begin{Theorem}\label{thm:error_primal}
	Let $\boldsymbol{u} \in V$ be a solution of \cref{eq:primal} and $\boldsymbol{u}_h \in
        V_h$ ba a solution of \cref{eq:discrete_primal_DG}.	If $\alpha >0$ sufficiently large, we have that
		\begin{align*}
			&\|\pmb{\mathsf{\mathcal{F}}}(\tilde{\nabla}_h \boldsymbol{u}_h) - \pmb{\mathsf{\mathcal{F}}}(\nabla \boldsymbol{u})\|^2_{2,\Omega}
			+ m_{\varphi_{\boldsymbol{\beta}_h(\boldsymbol{u}_h)},h}(\boldsymbol{u}_h) \\
			&\lesssim \inf_{\boldsymbol{v}_h \in V_h} \big(\|\pmb{\mathsf{\mathcal{F}}}(\tilde{\nabla}_h \boldsymbol{v}_h) - \pmb{\mathsf{\mathcal{F}}}(\nabla \boldsymbol{u})\|^2_{2,\Omega}
			+  m_{\varphi_{\boldsymbol{\beta}_h(\boldsymbol{u}_h)},h}(\boldsymbol{v}_h)\big)\,,
		\end{align*}
		where the constant depends only on $k$, $\omega_0$, and the characteristics of $\pmb{\mathsf{\mathcal{S}}}$.
	\end{Theorem}

	\begin{proof}
		Adding and substracting $\tilde{\nabla}_h \boldsymbol{z}_h\in \mathbb{P}^{k-1}(\mathcal{T}_h)^{n\times d}$ in \cref{eq:error_equation_primal} for arbitrary $\boldsymbol{z}_h\in V_h$, using that $\tilde{\nabla}_h \pmb{\mathsf{\mathcal{E}}}_h  \boldsymbol{z}_h = \nabla \pmb{\mathsf{\mathcal{E}}}_h  \boldsymbol{z}_h$, we get
		\begin{align}\label{eq:error_primal.1}
			\begin{aligned}
				0 &= (\pmb{\mathsf{\mathcal{S}}}(\tilde{\nabla}_h \boldsymbol{u}_h) - \pmb{\mathsf{\mathcal{S}}}(\nabla \boldsymbol{u}), \tilde{\nabla}_h\boldsymbol{z}_h)_{\Omega}
				\\&\quad+ \alpha\, \langle \pmb{\mathsf{\mathcal{S}}}_{\boldsymbol{\beta}_h(\boldsymbol{u}_h)}(h_\Gamma^{-1} \llbracket{\boldsymbol{u}_h\otimes \boldsymbol{n}}\rrbracket), \llbracket{\boldsymbol{z}_h \otimes \boldsymbol{n}}\rrbracket\rangle_{\Gamma_h} \\
				&\quad+ (\pmb{\mathsf{\mathcal{S}}}(\tilde{\nabla}_h \boldsymbol{u}_h) - \pmb{\mathsf{\mathcal{S}}}(\nabla \boldsymbol{u}), \tilde{\nabla}_h(\pmb{\mathsf{\mathcal{E}}}_h  \boldsymbol{z}_h - \boldsymbol{z}_h))_{\Omega}\\&
				\eqqcolon \mathfrak{I}_1 +
				\mathfrak{I}_2 +
				\mathfrak{I}_3\,.
			\end{aligned}
		\end{align}
		Next, we choose  $\boldsymbol{z}_h = \boldsymbol{u}_h - \boldsymbol{v}_h\in V_h$, where $\boldsymbol{v}_h\in V_h$ is arbitrary, and estimate $\mathfrak{I}_1$, $\mathfrak{I}_2$, $\mathfrak{I}_3$:
		
		\textit{ad $\mathfrak{I}_1$.} Using \cref{eq:hammera}, and \cref{eq:hammere}, the
		$\varepsilon$-Young  inequality \cref{ineq:young}
		with $\psi=\varphi_{\vert {\nabla} \boldsymbol{u}\vert}$, we find that
		\begin{align}
			\label{eq:error_primal.2}
			\begin{aligned}
				\mathfrak{I}_1 
				&=(\pmb{\mathsf{\mathcal{S}}}(\tilde{\nabla}_h \boldsymbol{u}_h) - \pmb{\mathsf{\mathcal{S}}}(\nabla \boldsymbol{u}), \tilde{\nabla}_h \boldsymbol{u}_h - \nabla \boldsymbol{u} + \nabla \boldsymbol{u} -\tilde{\nabla}_h \boldsymbol{v}_h)_{\Omega} \\
				&\geq c\, \|\pmb{\mathsf{\mathcal{F}}}(\tilde{\nabla}_h\boldsymbol{u}_h) - \pmb{\mathsf{\mathcal{F}}}(\nabla \boldsymbol{u})\|^2_{2,\Omega}
				-
				\vert  (\pmb{\mathsf{\mathcal{S}}}(\tilde{\nabla}_h\boldsymbol{u}_h) - \pmb{\mathsf{\mathcal{S}}}(\nabla \boldsymbol{u}),\nabla \boldsymbol{u} -\tilde{\nabla}_h \boldsymbol{v}_h)_{\Omega} 
				\vert \\
				&\geq (c-\varepsilon)\, \|\pmb{\mathsf{\mathcal{F}}}(\tilde{\nabla}_h \boldsymbol{u}_h) - \pmb{\mathsf{\mathcal{F}}}(\nabla \boldsymbol{u})\|^2_{2,\Omega}
				- c_\varepsilon \, \|\pmb{\mathsf{\mathcal{F}}}(\nabla \boldsymbol{u}) - \pmb{\mathsf{\mathcal{F}}}(\tilde{\nabla}_h \boldsymbol{v}_h)\|^2_{2,\Omega}\,.
			\end{aligned}
		\end{align}
		
		\textit{ad $\mathfrak{I}_2$.} Using that $\pmb{\mathsf{\mathcal{S}}}_{\boldsymbol{\beta}_h(\boldsymbol{u}_h)}(\pmb{\mathsf{Q}}): \pmb{\mathsf{Q}} \sim \varphi_{\boldsymbol{\beta}_h(\boldsymbol{u}_h)}(\pmb{\mathsf{Q}})$ uniformly in $\pmb{\mathsf{Q}}\in \mathbb{R}^{n\times d}$ (cf.\ \cref{lem:hammer,rem:sa})
        and   the $\varepsilon$-Young  inequality
		\cref{ineq:young} with $\psi=\varphi_{\vert \boldsymbol{\beta}_h(\boldsymbol{u}_h)\vert}$, we obtain
		\begin{align}
			\label{eq:error_primal.3}
			\begin{aligned}
				\mathfrak{I}_2 
				&=
				\alpha\, \langle \pmb{\mathsf{\mathcal{S}}}_{\boldsymbol{\beta}_h(\boldsymbol{u}_h)}(h_\Gamma^{-1} \llbracket{\boldsymbol{u}_h\otimes \boldsymbol{n}}\rrbracket),\llbracket{\boldsymbol{u}_h\otimes \boldsymbol{n}}\rrbracket - \llbracket{\boldsymbol{v}_h\otimes \boldsymbol{n}}\rrbracket\rangle_{\Gamma_h}\\
				&\geq
				\alpha\,c\, m_{\varphi_{\boldsymbol{\beta}_h(\boldsymbol{u}_h)},h} (\boldsymbol{u}_h)
				-
				\alpha\, \langle h_\Gamma (\varphi_{\boldsymbol{\beta}_h(\boldsymbol{u}_h)})'(h_\Gamma^{-1} \llbracket{\boldsymbol{u}_h\otimes \boldsymbol{n}}\rrbracket),h_\Gamma^{-1} \llbracket{\boldsymbol{v}_h\otimes \boldsymbol{n}}\rrbracket\rangle_{\Gamma_h}  \\
				&\geq
				\alpha \,(c- \varepsilon )\, m_{\varphi_{\boldsymbol{\beta}_h(\boldsymbol{u}_h)},h} (\boldsymbol{u}_h)
				- \alpha\, c_\varepsilon\, m_{\varphi_{\boldsymbol{\beta}_h(\boldsymbol{u}_h)},h} (\boldsymbol{v}_h)\,.
			\end{aligned}
		\end{align}
		
		\textit{ad $\mathfrak{I}_3$.} Using \cref{lem:max-shift}, we find that
		\begin{align}
			\label{eq:error_primal.5}
			\begin{aligned}
				\mathfrak{I}_3 &\leq
				\kappa\, \|\pmb{\mathsf{\mathcal{F}}}(\tilde{\nabla}_h\boldsymbol{u}_h) - \pmb{\mathsf{\mathcal{F}}}(\nabla \boldsymbol{u})\|^2_{2,\Omega}
				+
				c_\kappa\, m_{\varphi_{\boldsymbol{\beta}_h(\boldsymbol{u}_h)},h}(\boldsymbol{z}_h)\\
				&\lesssim
				\kappa \,\|\pmb{\mathsf{\mathcal{F}}}(\tilde{\nabla}_h \boldsymbol{u}_h) - \pmb{\mathsf{\mathcal{F}}}(\nabla \boldsymbol{u})\|^2_{2,\Omega}
				+
				c_\kappa\, (m_{\varphi_{\boldsymbol{\beta}_h(\boldsymbol{u}_h)},h}(\boldsymbol{u}_h) + m_{\varphi_{\boldsymbol{\beta}_h(\boldsymbol{u}_h)},h}(\boldsymbol{v}_h))\,.
			\end{aligned}
		\end{align}
		Combining \cref{eq:error_primal.2,eq:error_primal.3,eq:error_primal.5}, for $\varepsilon,\kappa>0$ sufficiently small and $\alpha>0$ sufficiently large, we arrive at 
		the claimed best-approximation result.
	\end{proof}
	
	As a first immediate consequence of the best-approximation result in \cref{thm:error_primal}, we obtain the convergence of the method under minimal regularity assumptions, i.e., merely~${\boldsymbol{u}\hspace*{-0.1em}\in \hspace*{-0.1em}V}$~and~${\boldsymbol{f}\hspace*{-0.1em}\in\hspace*{-0.1em} V^*}$.
	\begin{Corollary}[Convergence]\label{cor:primal_convergence} For $\alpha >0$ sufficiently large, it holds that
		\begin{align*}
			\|\pmb{\mathsf{\mathcal{F}}}(\tilde{\nabla}_h \boldsymbol{u}_h) - \pmb{\mathsf{\mathcal{F}}}(\nabla \boldsymbol{u})\|^2_{2,\Omega}
			+ m_{\varphi_{\boldsymbol{\beta}_h(\boldsymbol{u}_h)},h}(\boldsymbol{u}_h)\to 0\quad (h\to 0)\,.
		\end{align*}
	\end{Corollary}
	
	\begin{proof}
		Let $\Pi_h^{SZ}\colon V\to V_h\cap V$ denote the
Scott--Zhang quasi-interpolation operator (cf.~\cite{ZS1990}). Then,
due to $\llbracket{\Pi_h^{SZ}\boldsymbol{u}\otimes \boldsymbol{n}}\rrbracket=\pmb{\mathsf{0}}$ a.e.\ on $\Gamma_h$,
we have that
		\begin{align}\label{cor:primal_convergence.1}
m_{\varphi_{\boldsymbol{\beta}_h(\boldsymbol{u}_h)},h}(\Pi_h^{SZ}\boldsymbol{u})=0\,.
		\end{align} Therefore, using the stability and
convergence properties of $\Pi_h^{SZ}\colon V\to V_h\cap V$
(cf.~\cite{ZS1990}) and the density of smooth functions in $V$, we
conclude that $\Pi_h^{SZ}\boldsymbol{u}\to \boldsymbol{u}$ in $V$ $(h\to 0)$. This, together
with \cref{eq:hammera}, \cref{rem:phi} (iii), H\"older's inequality implies
that
		\begin{align}\label{cor:primal_convergence.2}
			\|\pmb{\mathsf{\mathcal{F}}}(\nabla \boldsymbol{u}) -
			\pmb{\mathsf{\mathcal{F}}}(\nabla\Pi_h^{SZ} \boldsymbol{u})\|^2_{2,\Omega} \lesssim \|\nabla \boldsymbol{u} - \nabla\Pi_h^{SZ} \boldsymbol{u}\|_p^{\min\{p,2\}}\to 0\quad (h\to 0)
		\end{align}
		with a constant depending possibly on $\delta$,
                $\|\nabla \boldsymbol{u}\|_p$. Putting everything together, choosing $\boldsymbol{v}_h=\Pi_h^{SZ} \boldsymbol{u}\in V_h$ in \cref{thm:error_primal}, using \cref{cor:primal_convergence.1,cor:primal_convergence.2}, we conclude that
		\begin{align*}
			&\|\pmb{\mathsf{\mathcal{F}}}(\tilde{\nabla}_h \boldsymbol{u}_h) - \pmb{\mathsf{\mathcal{F}}}(\nabla \boldsymbol{u})\|^2_{2,\Omega}
			+ m_{\varphi_{\boldsymbol{\beta}_h(\boldsymbol{u}_h)},h}(\boldsymbol{u}_h) \\
            &\lesssim \|\pmb{\mathsf{\mathcal{F}}}(\nabla \boldsymbol{u}) - \pmb{\mathsf{\mathcal{F}}}(\nabla\Pi_h^{SZ} \boldsymbol{u})\|^2_{2,\Omega}\to 0\quad(h\to 0)\,,
		\end{align*}
		which is the claimed convergence under minimal regularity assumptions.
	\end{proof}
	
	As a second immediate consequence of the best-approximation
	result in \cref{thm:error_primal}, we obtain fractional
	convergence rates of the method given fractional regularity
	assumptions on the solution of the continuous primal problem 
	\cref{eq:primal}. In order to express the fractional
	regularity of the solution of 
	\cref{eq:primal}, we make use of Nikolski\u{\i} spaces. For
	given $p \in  [1, \infty)$, $\beta\in (0,1]$, and $v\in
	L^p(\Omega)$, the  \textit{Nikolski\u{\i} semi-norm} is
	defined via 
	\begin{align*}
		[v]_{N^{\beta,p}(\Omega)}\coloneqq \sup_{h\in \mathbb{R}^d\setminus\{0\}}{\vert h\vert^{-\beta}\bigg(\int_{\Omega\cap (\Omega-h)}{\vert v(x + h)-v(x)\vert^p\,\mathrm{d}x}\bigg)^{\frac{1}{p}}}<\infty\,.
	\end{align*}
	Then, for $p \in  [1, \infty)$ and $\beta\in (0,1]$, the \textit{Nikolski\u{\i} space} is defined via
	\begin{align*}
		N^{\beta,p}(\Omega)\coloneqq \big\{ v\in L^p(\Omega)\mid [v]_{N^{\beta,p}(\Omega)}<\infty\big\}\,,
	\end{align*}
	and the \textit{Nikolski\u{\i} norm} $\|\cdot\|_{N^{\beta,p}(\Omega)}\coloneqq \| \cdot\|_{p,\Omega}+[v]_{N^{\beta,p}(\Omega)}$ turns $N^{\beta,p}(\Omega)$  into a Banach space.
	\begin{Corollary}[Fractional convergence rates]\label{cor:primal_rate}
		Assume that the family of triangulations $\{\mathcal{T}_h\}_{h}$ is quasi-uniform, and that
		$\pmb{\mathsf{\mathcal{F}}}(\nabla \boldsymbol{u})\in N^{\beta,2}(\Omega)^{n\times d}$ for some $\beta\in (0,1]$.
		Then, for $\alpha >0$ sufficiently large, it holds that
		\begin{align*}
			\|\pmb{\mathsf{\mathcal{F}}}(\tilde{\nabla}_h \boldsymbol{u}_h) - \pmb{\mathsf{\mathcal{F}}}(\nabla \boldsymbol{u})\|^2_{2,\Omega}
			+ m_{\varphi_{\boldsymbol{\beta}_h(\boldsymbol{u}_h)},h}(\boldsymbol{u}_h)\lesssim h^{2\beta}\, [\pmb{\mathsf{\mathcal{F}}}(\nabla \boldsymbol{u})]_{N^{\beta,2}(\Omega)}^2\,,
		\end{align*}
		where the constant depends only on $k$, $\omega_0$, and the characteristics of $\pmb{\mathsf{\mathcal{S}}}$.
	\end{Corollary}
	
	\begin{proof}
		Thanks to the quasi-uniformity, appealing to \cite[Thm.\ 4.2]{breit-lars-etal}, we have that
		\begin{align}\label{cor:primal_rate.1}
			\|\pmb{\mathsf{\mathcal{F}}}(\nabla \boldsymbol{u}) - \pmb{\mathsf{\mathcal{F}}}(\nabla\Pi_h^{SZ} \boldsymbol{u})\|^2_{2,\Omega}\lesssim h^{2\beta}\, [\pmb{\mathsf{\mathcal{F}}}(\nabla \boldsymbol{u})]_{N^{\beta,2}(\Omega)}^2\,.
		\end{align}
		Thus, choosing $\boldsymbol{v}_h=\Pi_h^{SZ} \boldsymbol{u}\in V_h\cap V$ in \cref{thm:error_primal}, due to $\llbracket{\Pi_h^{SZ} \boldsymbol{u}\otimes \boldsymbol{n}}\rrbracket=0$ on $\Gamma_h$, we obtain 
		\begin{align*}
			\|\pmb{\mathsf{\mathcal{F}}}(\tilde{\nabla}_h \boldsymbol{u}_h) - \pmb{\mathsf{\mathcal{F}}}(\nabla \boldsymbol{u})\|^2_{2,\Omega}
			+ m_{\varphi_{\boldsymbol{\beta}_h(\boldsymbol{u}_h)},h}(\boldsymbol{u}_h)&\lesssim \|\pmb{\mathsf{\mathcal{F}}}(\nabla \boldsymbol{u}) - \pmb{\mathsf{\mathcal{F}}}(\nabla\Pi_h^{SZ} \boldsymbol{u})\|^2_{2,\Omega}\\&\lesssim h^{2\beta}\, [\pmb{\mathsf{\mathcal{F}}}(\nabla \boldsymbol{u})]_{N^{\beta,2}(\Omega)}^2\,,
		\end{align*}
		which is the claimed fractional a priori error estimate.
	\end{proof}

        \begin{Remark}
          In view of \cite[Corollary 5.8]{dr-interpol} the assertion
          of \cref{cor:primal_rate} for $\beta=1$ is also valid if $\pmb{\mathsf{\mathcal{F}}}(\nabla \boldsymbol{u})\in
          W^{1,2}(\Omega)^{n\times d}$ without the additional
          assumption that the triangulation $\{\mathcal T_h\}_h$ is quasi-uniform. 
        \end{Remark}
	Eventually, resorting to the approximation properties of the \textit{node-averaging quasi-interpolation operator} $\boldsymbol{\mathcal{I}}_h^{av}\colon
	C^0(\mathcal{T}_h)\to V_{h,c}$,
	for every $\boldsymbol{v}_h\in
	\mathbb{P}(\mathcal{T}_h)$, defined via
	\begin{align*}
		\boldsymbol{\mathcal{I}}_h^{av}\boldsymbol{v}_h\coloneqq \sum_{y\in \mathcal{L}_k^{\textup{int}}(\mathcal{T}_h)}{\langle \boldsymbol{v}_h\rangle_y\,\varphi_y^k}\,,\quad\text{ where }\quad \langle \boldsymbol{v}_h\rangle_y\coloneqq %\begin{cases} 
		\sum_{T\in \mathcal{T}_h;y\in T}{(\boldsymbol{v}_h|_T)(y)}
		%&\quad\text{ if }y\in \Omega\,,\\
		%\boldsymbol{0}&\quad\text{ if }y\in\partial \Omega
		\,,
		%		\end{cases}
		\end{align*}
		\cref{thm:error_primal} allows to carry out an ansatz class competition, which reveals that the approximation capabilities of the LDG and IIDG approximations and continuous Lagrange approximation of the problem \cref{eq:PDE} are comparable.
	
	\begin{Corollary}[Ansatz class
          competition]\label{cor:primal_competition} Let $\boldsymbol{u}_h^c\in
          V_{h,c}\coloneqq V_h\cap V$ be the continuous Lagrange solution of \cref{eq:PDE}, i.e.,
		for every $\boldsymbol{v}_h\in V_{h,c}$, it holds that
		\begin{align}
			(\pmb{\mathsf{\mathcal{S}}}(\nabla\boldsymbol{u}_h^c ),\nabla \boldsymbol{v}_h)_{\Omega}=(\boldsymbol{f},\boldsymbol{v}_h)_{\Omega}\,.\label{eq:continuous_lagrange}
		\end{align}
		Then, for $\alpha >0$ sufficiently large, it holds that
		\begin{align*}
			\|\pmb{\mathsf{\mathcal{F}}}(\tilde{\nabla}_h \boldsymbol{u}_h) - \pmb{\mathsf{\mathcal{F}}}(\nabla \boldsymbol{u})\|^2_{2,\Omega}+ m_{\varphi_{\boldsymbol{\beta}_h(\boldsymbol{u}_h)},h}(\boldsymbol{u}_h)\sim 	\|\pmb{\mathsf{\mathcal{F}}}(\nabla \boldsymbol{u}_h^c) - \pmb{\mathsf{\mathcal{F}}}(\nabla \boldsymbol{u})\|^2_{2,\Omega}\,,
		\end{align*}
		with constants depending only on $k$, $\omega_0$, and the characteristics of $\pmb{\mathsf{\mathcal{S}}}$.
	%	i.e., the approximation capabilities of the discrete primal formulation \cref{eq:discrete_primal_DG} and the continuous Lagrange approximation \cref{eq:continuous_lagrange} of \cref{eq:PDE} are comparable.
	\end{Corollary}
	
	\begin{proof}
		\textit{ad $\lesssim$.} Using
		\cref{thm:error_primal} with $\boldsymbol{v}_h=\boldsymbol{u}_h^c\in
		V_{h,c}\subseteq V_h$ and observing that
		$m_{\varphi_{\boldsymbol{\beta}_h(\boldsymbol{u}_h)},h}(\boldsymbol{u}_h^c)=0$, we find that
		\begin{align*}
			\|\pmb{\mathsf{\mathcal{F}}}(\tilde{\nabla}_h \boldsymbol{u}_h) - \pmb{\mathsf{\mathcal{F}}}(\nabla \boldsymbol{u})\|^2_{2,\Omega}+ m_{\varphi_{\boldsymbol{\beta}_h(\boldsymbol{u}_h)},h}(\boldsymbol{u}_h)\lesssim 	\|\pmb{\mathsf{\mathcal{F}}}(\nabla \boldsymbol{u}_h^c) - \pmb{\mathsf{\mathcal{F}}}(\nabla \boldsymbol{u})\|^2_{2,\Omega}\,.
		\end{align*}

		\textit{ad $\gtrsim$.} 
%		Denote by
%		$\boldsymbol{\mathcal{I}}_h^{av}\colon
%		\mathbb{P}(\mathcal{T}_h)\to V_{h,c}$,
%		defined, for every $\bv_h\in
%		\mathbb{P}(\mathcal{T}_h)$, via
%		\begin{align*}
%			\boldsymbol{\mathcal{I}}_h^{av}\bv_h\coloneqq \sum_{y\in \mathcal{L}_k^{\textup{int}}(\mathcal{T}_h)}{\langle \bv_h\rangle_y\,\varphi_y^k}\,,\quad\text{ where }\quad \langle \bv_h\rangle_y\coloneqq %\begin{cases} 
%			\sum_{T\in \mathcal{T}_h;y\in T}{(\bv_h|_T)(y)}
%			%&\quad\text{ if }y\in \Omega\,,\\
%			%\boldsymbol{0}&\quad\text{ if }y\in\partial \Omega
%			\,,
%			%		\end{cases}
%	\end{align*}
%	the node-averaging quasi-interpolation operator. Then, u
	Using \cref{eq:hammera}, that $\varphi_{\vert \tilde{\nabla}_h \boldsymbol{u}_h\vert}\leq \varphi_{\boldsymbol{\beta}_h(\boldsymbol{u}_h)}$, and  \cite[Prop.\ A.1]{AK.2023},
	for every $K\in \mathcal{T}_h$,~we~find~that
	\begin{align}\label{eq:primal_competition.1}
		\begin{aligned}
			\|\pmb{\mathsf{\mathcal{F}}}(\nabla 	\boldsymbol{\mathcal{I}}_h^{av}\boldsymbol{u}_h) - \pmb{\mathsf{\mathcal{F}}}(\tilde{\nabla}_h\boldsymbol{u}_h)\|^2_{2,K}&\lesssim \rho_{\varphi_{\vert \tilde{\nabla}_h \boldsymbol{u}_h\vert},K}(\tilde{\nabla}_h(\boldsymbol{u}_h- 	\boldsymbol{\mathcal{I}}_h^{av}\boldsymbol{u}_h))\\&\lesssim
			\rho_{\varphi_{\boldsymbol{\beta}_h(\boldsymbol{u}_h)},K}(\tilde{\nabla}_h(\boldsymbol{u}_h- 	\boldsymbol{\mathcal{I}}_h^{av}\boldsymbol{u}_h))
			\\&\lesssim
                        m_{\varphi_{\boldsymbol{\beta}_h(\boldsymbol{u}_h)},h}(\boldsymbol{u}_h)  \,. %h\,\rho_{\varphi_{\bbeta_h(\bu_h)},\Gamma_h(K)}(h^{-1}\jump{\bu_h\otimes \bn}) \,.
		\end{aligned}
	\end{align}
	Eventually, using the best-approximation properties of  $\boldsymbol{u}_h^c\in V_{h,c}$ (cf.\ \cite[Lem.\ 5.2]{dr-interpol}) and \cref{eq:primal_competition.1}, we conclude that
	\begin{align}\label{eq:primal_competition.2}
%		\begin{aligned}
			\|\pmb{\mathsf{\mathcal{F}}}(\nabla \boldsymbol{u}_h^c) - \pmb{\mathsf{\mathcal{F}}}(\nabla
			\boldsymbol{u})\|^2_{2,\Omega}&\lesssim \|\pmb{\mathsf{\mathcal{F}}}(\nabla
			\boldsymbol{\mathcal{I}}_h^{av}\boldsymbol{u}_h) - \pmb{\mathsf{\mathcal{F}}}(\nabla
			\boldsymbol{u})\|^2_{2,\Omega} \notag\\&\lesssim \|\pmb{\mathsf{\mathcal{F}}}(\tilde{\nabla}_h
			\boldsymbol{u}_h) - \pmb{\mathsf{\mathcal{F}}}(\nabla \boldsymbol{u})\|^2_{2,\Omega}+\|\pmb{\mathsf{\mathcal{F}}}(\nabla
			\boldsymbol{\mathcal{I}}_h^{av}\boldsymbol{u}_h) -
			\pmb{\mathsf{\mathcal{F}}}(\tilde{\nabla}_h\boldsymbol{u}_h)\|^2_{2,\Omega} \\&\lesssim
			\|\pmb{\mathsf{\mathcal{F}}}(\tilde{\nabla}_h \boldsymbol{u}_h) - \pmb{\mathsf{\mathcal{F}}}(\nabla
			\boldsymbol{u})\|^2_{2,\Omega}+m_{\varphi_{\boldsymbol{\beta}_h(\boldsymbol{u}_h)},h}(\boldsymbol{u}_h)\,.\notag
%		\end{aligned}
	\end{align}
	Putting everything together, we arrive ar the claimed equivalence.
\end{proof}

\begin{Remark}[Application to Crouzeix--Raviart element]\label{rmk:CR}
	If $V_h\coloneqq \mathcal{CR}^k(\mathcal{T}_h)$, where $\mathcal{CR}^k(\mathcal{T}_h)$ is the Crouzeix--Raviart element and its generalization to higher orders (cf.\ \cite[Sec.\ 3.3 (3.17)]{VZ.2019.II}),~and if we use the same primal formulation 
	with $\tilde{\nabla}_h=\nabla_h$, then
	it is possible to prove the same results. However, omitting the stabilization terms in the primal formulation as in the linear case is still an open problem, since, in this case, one cannot simply absorb the term $c_\kappa\, m_{\varphi_{\boldsymbol{\beta}_h(\boldsymbol{u}_h)},h}(\boldsymbol{v}_h)$ in \cref{eq:error_primal.5}  in the term  
	$\alpha\, c_\varepsilon\, m_{\varphi_{\boldsymbol{\beta}_h(\boldsymbol{u}_h)},h}(\boldsymbol{v}_h)$ in \cref{eq:error_primal.3} via choosing $\alpha>0$ large enough.
\end{Remark}

\section{Mixed formulation}\label{sec:mixed}

Since $\pmb{\mathsf{\mathcal{S}}}\colon \mathbb{R}^{n\times d}\to \mathbb{R}^{n\times d}$ is bounded, continuous, and strictly monotone, by the Browder--Minty theorem (cf.\ \cite[Thm.\ 26.A]{zei-IIB}), it is also bijective and its inverse  $\pmb{\mathsf{\mathcal{D}}}\coloneqq\pmb{\mathsf{\mathcal{S}}}^{-1}\colon \mathbb{R}^{n\times d}\to \mathbb{R}^{n\times d}$ is bounded, continuous, and strictly monotone as well. This motivates to consider the following mixed formulation and  corresponding discrete mixed formulation. In doing so, throughout the entire~section, we assume that $\pmb{\mathsf{\mathcal{S}}}\colon \mathbb{R}^{n\times d}\to \mathbb{R}^{n\times d}$ has $(p,\delta)$-structure in the sense of \cref{assum:extra_stress}. To ensure the validity of the identity \cref{eq:equivalence_without_E}, in this formulation, it will be crucial~to~have~that~$\tilde{\nabla}_h(V_h)\subseteq \Sigma_h$.

\subsection{Continuous mixed formulation} We abbreviate
\begin{align*}
	\Sigma\coloneqq L^{p'}(\Omega)^{n\times d}\,.
\end{align*}
Then, for given $\boldsymbol{f}\in V^*$, the \textit{(continuous) mixed formulation} of \cref{eq:PDE} seeks for 
$(\pmb{\mathsf{S}},\boldsymbol{u})^\top \in \Sigma\times V$ such that for every $(\pmb{\mathsf{T}},\boldsymbol{z})^\top\in \Sigma\times V$, it holds that
\begin{subequations}\label{eq:mixed}
\begin{align}
  \label{eq:mixed_1}
		(\pmb{\mathsf{\mathcal{D}}}(\pmb{\mathsf{S}}),\pmb{\mathsf{T}})_{\Omega}
		-(\pmb{\mathsf{T}} ,\nabla \boldsymbol{u})_{\Omega}&=0\,,\\
  \label{eq:mixed_2}
		( \pmb{\mathsf{S}} , \nabla \boldsymbol{z} )_{\Omega}
		&=
		\langle \boldsymbol{f},  \boldsymbol{z} \rangle_{V}\,.
\end{align}
\end{subequations}
The unique solvability of the mixed formulation  \cref{eq:mixed} is a consequence of the unique solvability of the primal formulation \cref{eq:primal}. In fact, both formulations are equivalent: If $(\pmb{\mathsf{S}},\boldsymbol{u})^\top \in \Sigma\times V$ is a solution of \cref{eq:mixed}, then \cref{eq:mixed_1} implies that $\pmb{\mathsf{\mathcal{D}}}(\pmb{\mathsf{S}})=\nabla \boldsymbol{u}$, which is equivalent to $\pmb{\mathsf{S}}=\pmb{\mathsf{\mathcal{S}}}(\nabla\boldsymbol{u})$, and, thus, together with \cref{eq:mixed_2} shows that $\boldsymbol{u} \in V$ is a solution of \cref{eq:primal}. If, in turn, $\boldsymbol{u} \in V$ is~a~solution~of~\cref{eq:primal}, then, setting $\pmb{\mathsf{S}}\coloneqq \pmb{\mathsf{\mathcal{S}}}(\nabla\boldsymbol{u})\in \Sigma$, we have that  $\pmb{\mathsf{\mathcal{D}}}(\pmb{\mathsf{S}})=\nabla \boldsymbol{u}$ and, thus, \cref{eq:mixed_1}. In addition, $\pmb{\mathsf{S}}\in \Sigma$ then satisfies  \cref{eq:mixed_2}, so that $(\pmb{\mathsf{S}},\boldsymbol{u})^\top \in\Sigma\times V$ is a solution of \cref{eq:mixed}. It is  possible to prove the weak solvability of \cref{eq:mixed} directly using pseudo-monotone operator theory. For this, however,~one~first~needs to ascertain that $\pmb{\mathsf{\mathcal{D}}}\colon \mathbb{R}^{n\times d}\to \mathbb{R}^{n\times d}$ has similar~--but~\mbox{dual--~growth}~conditions~to~$\pmb{\mathsf{\mathcal{S}}}\colon \mathbb{R}^{n\times d}\to \mathbb{R}^{n\times d}$.\enlargethispage{1mm}

\begin{Lemma}\label{lem:D_structure}
	Let $\pmb{\mathsf{\mathcal{S}}}\colon \mathbb{R}^{n\times d}\to \mathbb{R}^{n\times d}$ have $(p,\delta)$-structure in the sense of \cref{assum:extra_stress}. Then, its inverse $\pmb{\mathsf{\mathcal{D}}}\coloneqq\pmb{\mathsf{\mathcal{S}}}^{-1}\colon \mathbb{R}^{n\times d}\to \mathbb{R}^{n\times d}$ has $(p',\delta^{p-1})$-structure in the sense of \cref{assum:extra_stress}.
\end{Lemma}

\begin{proof}
	Due to \cref{eq:hammera}, for every $\pmb{\mathsf{Q}},\pmb{\mathsf{P}}\in \mathbb{R}^{n\times d}$, it holds that
	\begin{align}
		\begin{aligned}
			\big({\pmb{\mathsf{\mathcal{D}}}}(\pmb{\mathsf{Q}}) - {\pmb{\mathsf{\mathcal{D}}}}(\pmb{\mathsf{P}})\big) : (\pmb{\mathsf{Q}}-\pmb{\mathsf{P}})&=	\big({\pmb{\mathsf{\mathcal{S}}}}({\pmb{\mathsf{\mathcal{D}}}}(\pmb{\mathsf{Q}})) - {\pmb{\mathsf{\mathcal{S}}}}({\pmb{\mathsf{\mathcal{D}}}}(\pmb{\mathsf{P}}))\big) : ({\pmb{\mathsf{\mathcal{D}}}}(\pmb{\mathsf{Q}}) - {\pmb{\mathsf{\mathcal{D}}}}(\pmb{\mathsf{P}})
			)%\\ &\ge C_0 \,\varphi_{\abs{{\fD}(\BQ)}}(\abs{{\fD}(\BQ) -{\fD}(\BP)})
			\\ &\sim  \,(\varphi^*)_{\left| \pmb{\mathsf{\mathcal{S}}}(\pmb{\mathsf{\mathcal{D}}}(\pmb{\mathsf{Q}})) \right|}(\left| \pmb{\mathsf{\mathcal{S}}}(\pmb{\mathsf{\mathcal{D}}}(\pmb{\mathsf{Q}})) -
				\pmb{\mathsf{\mathcal{S}}}(\pmb{\mathsf{\mathcal{D}}}(\pmb{\mathsf{P}})) \right|)
			\\ &=  \,(\varphi^*)_{\vert \pmb{\mathsf{Q}}\vert}(\left| \pmb{\mathsf{Q}} -
				\pmb{\mathsf{P}} \right|) \,.
		\end{aligned}
	\end{align}
	Due to $((\varphi_{\left| \pmb{\mathsf{\mathcal{D}}}(\pmb{\mathsf{Q}})  \right|})^*)'\circ\varphi'_{\left| \pmb{\mathsf{\mathcal{D}}}(\pmb{\mathsf{Q}})  \right|}=\textup{id}_{\mathbb{R}}$, $((\varphi_{\left| \pmb{\mathsf{\mathcal{D}}}(\pmb{\mathsf{Q}})  \right|})^*)'\sim (\varphi^*)'_{\left| \pmb{\mathsf{\mathcal{S}}}(\pmb{\mathsf{\mathcal{D}}}(\pmb{\mathsf{Q}}))  \right|}$, and \cref{eq:hammere}, for every $\pmb{\mathsf{Q}},\pmb{\mathsf{P}}\in \mathbb{R}^{n\times d}$, it holds that
	\begin{align}
		\begin{aligned}
			\left| \pmb{\mathsf{\mathcal{D}}}(\pmb{\mathsf{Q}}) - \pmb{\mathsf{\mathcal{D}}}(\pmb{\mathsf{P}}) \right| &=  (((\varphi_{\left| \pmb{\mathsf{\mathcal{D}}}(\pmb{\mathsf{Q}})  \right|})^*)'\circ\varphi'_{\left| \pmb{\mathsf{\mathcal{D}}}(\pmb{\mathsf{Q}})  \right|})(	\left| \pmb{\mathsf{\mathcal{D}}}(\pmb{\mathsf{Q}}) - \pmb{\mathsf{\mathcal{D}}}(\pmb{\mathsf{P}}) \right| )
			%\\&\sim  (\varphi^*)'_{\abs{\fS(\fD(\BQ)) }}(\varphi'_{\abs{\fD(\BQ) }}(	\abs{\fD(\BQ) - \fD(\BP)}) )
			\\&\lesssim  (\varphi^*)'_{\left| \pmb{\mathsf{\mathcal{S}}}(\pmb{\mathsf{\mathcal{D}}}(\pmb{\mathsf{Q}}))  \right|}(	\vert \pmb{\mathsf{\mathcal{S}}}(\pmb{\mathsf{\mathcal{D}}}(\pmb{\mathsf{Q}})) -\pmb{\mathsf{\mathcal{S}}}(\pmb{\mathsf{\mathcal{D}}}(\pmb{\mathsf{P}}))\vert  )
			\\&=  (\varphi^*)'_{\vert \pmb{\mathsf{Q}}\vert}(	\vert \pmb{\mathsf{Q}} -\pmb{\mathsf{P}}\vert  )\,.
		\end{aligned}
	\end{align}
	In other words, $\pmb{\mathsf{\mathcal{D}}}\colon \mathbb{R}^{n\times d}\to \mathbb{R}^{n\times d}$ has $(p',\delta^{p-1})$-structure in the sense of \cref{assum:extra_stress}.
\end{proof}

Since, appealing to \cref{lem:D_structure}, $\pmb{\mathsf{\mathcal{D}}}$ has $(p',\delta^{p-1})$-structure in the sense of \cref{assum:extra_stress},
similar to \cref{lem:hammer}, we have the following connections between $\pmb{\mathsf{\mathcal{D}}},\pmb{\mathsf{\mathcal{F}}},\pmb{\mathsf{\mathcal{F}}}^*\colon \ensuremath{\mathbb{R}}^{n \times d}\to \ensuremath{\mathbb{R}}^{n \times d}$ and $\varphi_a,(\varphi^*)_a,((\varphi^*)_a)^*\colon \mathbb{R}^{\ge 0}\to \mathbb{R}^{\ge 0}$, $a\ge 0$.

\begin{Proposition}\label{lem:hammer_inverse}
	Let $\pmb{\mathsf{\mathcal{S}}}$ satisfy \cref{assum:extra_stress}, let $\varphi$ be defined in \cref{eq:def_phi}, and let $\pmb{\mathsf{\mathcal{F}}},\pmb{\mathsf{\mathcal{F}}}^*$ be defined in \cref{eq:def_F}. Then, uniformly with respect to 
	$\pmb{\mathsf{Q}}, \pmb{\mathsf{P}} \in \ensuremath{\mathbb{R}}^{n \times d}$, we have that
	\begin{align}\label{eq:hammer_inversea}
		\begin{aligned}
			\big(\pmb{\mathsf{\mathcal{D}}}(\pmb{\mathsf{Q}}) - \pmb{\mathsf{\mathcal{D}}}(\pmb{\mathsf{P}})\big)
			:(\pmb{\mathsf{Q}}-\pmb{\mathsf{P}} ) &\sim  \vert \pmb{\mathsf{\mathcal{F}}}^*(\pmb{\mathsf{Q}}) - \pmb{\mathsf{\mathcal{F}}}^*(\pmb{\mathsf{P}})\vert ^2
			\\
			&\sim((\varphi^*)_{\vert \pmb{\mathsf{Q}}\vert})^*(\vert\pmb{\mathsf{\mathcal{D}}}(\pmb{\mathsf{Q}} ) - \pmb{\mathsf{\mathcal{D}}}(\pmb{\mathsf{P}} )\vert )
			\\&\sim \varphi _{|\pmb{\mathsf{\mathcal{D}}} ( \pmb{\mathsf{Q}})|} (|\pmb{\mathsf{\mathcal{D}}} (\pmb{\mathsf{Q}}) - \pmb{\mathsf{\mathcal{D}}} ( \pmb{\mathsf{P}}) |)
			\,,
		\end{aligned}
	\end{align}
	%\vspace{-6.5mm}
	\begin{align}
		\label{eq:hammer_inversee}
		\hspace*{2mm}\vert\pmb{\mathsf{\mathcal{D}}}(\pmb{\mathsf{Q}}) - \pmb{\mathsf{\mathcal{D}}}(\pmb{\mathsf{P}})\vert  &\sim   \smash{(\varphi^*)'_{\vert \pmb{\mathsf{Q}}\vert}(\vert\pmb{\mathsf{Q}}-\pmb{\mathsf{P}}\vert )}\,,
	\end{align}
	\begin{align}
		\mspace{-25mu}
		\label{eq:F-F*_inverse3}
		\mspace{-13mu}\smash{\vert\pmb{\mathsf{\mathcal{F}}}(\pmb{\mathsf{\mathcal{D}}}(\pmb{\mathsf{P}}))-\pmb{\mathsf{\mathcal{F}}}(\pmb{\mathsf{\mathcal{D}}}(\pmb{\mathsf{Q}}))\vert^2}
		&\sim  \smash{\vert\pmb{\mathsf{\mathcal{F}}}^*(\pmb{\mathsf{P}})-\pmb{\mathsf{\mathcal{F}}}^*(\pmb{\mathsf{Q}})\vert ^2}\,.
	\end{align}
	The constants in \crefrange{eq:hammer_inversea}{eq:F-F*_inverse3}
	depend only on the characteristics of ${\pmb{\mathsf{\mathcal{S}}}}$.
\end{Proposition}

\subsection{Discrete mixed formulation} For given $k\in \mathbb{N}$, we abbreviate
\begin{align*}
	\Sigma_h \coloneqq\Sigma_h^{k-1} \coloneqq {\mathbb{P}}^{k-1}(\mathcal{T}_h)^{ n \times d}\,.
\end{align*}
Then, for given $\boldsymbol{f}\in V^*$, the \textit{discrete mixed formulation} of \cref{eq:PDE} seeks for    $(\pmb{\mathsf{S}}_h,\boldsymbol{u}_h)\in
\Sigma_h\times V_h $  such that for every $(\pmb{\mathsf{T}}_h,\boldsymbol{z}_h)\in
\Sigma_h\times V_h $, it holds that
\begin{align}\label{eq:discrete_mixed_DG}
	\begin{aligned}
		(\pmb{\mathsf{\mathcal{D}}}(\pmb{\mathsf{S}}_h)-\tilde{\nabla}_h \boldsymbol{u}_h,\pmb{\mathsf{T}}_h)_{\Omega} &=0\,,\\
		( \pmb{\mathsf{S}}_h , \pmb{\mathsf{\mathcal{G}}}_h \boldsymbol{z}_h )_{\Omega}
		+ \alpha\, \langle \pmb{\mathsf{\mathcal{S}}}_{\boldsymbol{\beta}_h(\boldsymbol{u}_h)}( h_\Gamma^{-1} \llbracket{ \boldsymbol{u}_h \otimes\boldsymbol{n}}\rrbracket ) , \llbracket{ \boldsymbol{z}_h \otimes\boldsymbol{n}}\rrbracket\rangle_{\Gamma_h}
		&=
		\langle \boldsymbol{f}, \pmb{\mathsf{\mathcal{E}}}_h  \boldsymbol{z}_h \rangle_{V}\,.
	\end{aligned}
\end{align}

\begin{Proposition} \label{prop:mixed_existence}The following statements apply:
\begin{itemize}
	\item[(i)] 	If $\tilde{\nabla}_h=  \pmb{\mathsf{\mathcal{G}}}_h$, then the mixed formulation \cref{eq:discrete_mixed_DG} admits a solution.
	\item[(ii)] If $\tilde{\nabla}_h=  \nabla_h$, then for sufficiently large $\alpha>0$, the mixed formulation \cref{eq:discrete_mixed_DG} admits~a~solution.\enlargethispage{1mm}
\end{itemize}
\end{Proposition}

\begin{proof}
We equip $\Sigma_h\times V_h$~with~the~norm
\begin{align*}
	\|(\pmb{\mathsf{T}}_h,\boldsymbol{z}_h)\|_{\Sigma_h\times V_h}\coloneqq \|\pmb{\mathsf{T}}_h\|_{p',\Omega} + \|\boldsymbol{z}_h\|_{p,h}\,.
\end{align*}
For every $\varepsilon>0$, we consider the regularized operator $\pmb{\mathsf{\mathcal{T}}}_{\hspace{-1mm}h}^{\varepsilon}\colon \Sigma_h\times V_h\to (\Sigma_h\times V_h)^*$, 
defined, for every $(\pmb{\mathsf{S}}_h,\boldsymbol{u}_h)\in
\Sigma_h\times V_h $ and $(\pmb{\mathsf{T}}_h,\boldsymbol{v}_h)\in
\Sigma_h\times V_h $, via
\begin{align*}
    \langle \pmb{\mathsf{\mathcal{T}}}_{\hspace{-1mm}h}^{\varepsilon}(\pmb{\mathsf{S}}_h,\boldsymbol{u}_h),(\pmb{\mathsf{T}}_h,\boldsymbol{v}_h)\rangle_{\Sigma_h\times V_h}\coloneqq{}& (\pmb{\mathsf{\mathcal{D}}}(\pmb{\mathsf{S}}_h)-\tilde{\nabla}_h\boldsymbol{u}_h,\pmb{\mathsf{T}}_h)_{\Omega}
	+(\pmb{\mathsf{S}}_h , \pmb{\mathsf{\mathcal{G}}}_h \boldsymbol{v}_h )_{\Omega}
	\\&+\varepsilon\,(\pmb{\mathsf{\mathcal{S}}}( \nabla_h \boldsymbol{u}_h ) ,\nabla_h\boldsymbol{v}_h)_{\Omega}
	\\&+ \alpha\, \langle \pmb{\mathsf{\mathcal{S}}}_{\boldsymbol{\beta}_h(\boldsymbol{u}_h)}( h_\Gamma^{-1} \llbracket{\boldsymbol{u}_h \otimes\boldsymbol{n}}\rrbracket ) , \llbracket{\boldsymbol{v}_h \otimes\boldsymbol{n}}\rrbracket\rangle_{\Gamma_h}\,.
\end{align*}
          Since $V_h$ is finite dimensional, consists of broken
          polynomials, \cite[Lem.~3.18]{bdr-7-5}, the properties of
          $\pmb{\mathsf{\mathcal{D}}}$, $\pmb{\mathsf{\mathcal{S}}}$ and $\pmb{\mathsf{\mathcal{S}}}_a$, $a\ge 0$, imply that the operator
          $\pmb{\mathsf{\mathcal{T}}}_{\hspace{-1mm}h}^{\varepsilon}\colon \Sigma_h\times V_h\to (\Sigma_h\times V_h)^*$, for every fixed
          $h, \varepsilon>0$, is well-defined, continuous, and \mbox{monotone}.
          To prove the
          boundedness of $\pmb{\mathsf{\mathcal{T}}}_{\hspace{-1mm}h}^{\varepsilon}\colon \Sigma_h\times V_h\to (\Sigma_h\times V_h)^*$,
          we use the properties of $\pmb{\mathsf{\mathcal{D}}}, \pmb{\mathsf{\mathcal{S}}}, \pmb{\mathsf{\mathcal{S}}}_{\boldsymbol{\beta}_h(\boldsymbol{v}_h)}$
          (cf.\ \cref{lem:hammer,rem:sa}),
          Young's inequality, \cref{eq:psi'}, the stability of
          $\boldsymbol{\mathcal{R}}_h$ (cf.~\cite[(A.25)]{dkrt-ldg}), a~shift change
          (cf.\ \cref{lem:shift-change}),
          $\sum_{F \in \Gamma_h}\!h_F\mathscr{H}^{d-1}(F)\!\sim\! \vert \Omega\vert$, 
          \cite[Lemma 12.1]{EG.2021},  Jensen's inequality,   $\varphi(t)+t^p\sim
          t^p+ \delta^p$, and  $\varphi^*(t)+t^{p'}\sim
          t^{p'}+ \delta^{p'}$ for all $t\ge 0$ (cf.~\cite{bdr-7-5}), to
          find that 
          \begin{align*}
            &\left| \langle
              \pmb{\mathsf{\mathcal{T}}}_{\hspace{-1mm}h}^{\varepsilon}(\pmb{\mathsf{S}}_h,\boldsymbol{u}_h),(\pmb{\mathsf{T}}_h,\boldsymbol{v}_h)\rangle_{\Sigma_h\times
              V_h} \right|\\
            &\lesssim \rho_{\varphi^*,\Omega}(\pmb{\mathsf{S}}_h) +
              \rho_{\varphi^*,\Omega}(\pmb{\mathsf{T}}_h) +
              \rho_{\varphi,\Omega}(\tilde{\nabla}_h \boldsymbol{u}_h) +
              \rho_{\varphi,\Omega}({\nabla}_h \boldsymbol{u}_h) +
              \rho_{\varphi,\Omega}(\pmb{\mathsf{\mathcal{G}}}_h \boldsymbol{v}_h) \\
            &\quad +
              \rho_{\varphi,\Omega}(\tilde{\nabla}_h \boldsymbol{v}_h) + \alpha \,
              m_{\varphi_{\boldsymbol{\beta}_h(\boldsymbol{v}_h)},h}(\boldsymbol{u}_h) +\alpha \,
              m_{\varphi_{\boldsymbol{\beta}_h(\boldsymbol{v}_h)},h}(\boldsymbol{v}_h) \\
            &\lesssim \rho_{\varphi^*,\Omega}(\pmb{\mathsf{S}}_h) +
              \rho_{\varphi^*,\Omega}(\pmb{\mathsf{T}}_h) +
               \rho_{\varphi,\Omega}({\nabla}_h \boldsymbol{u}_h) +
            \rho_{\varphi,\Omega}(\nabla _h \boldsymbol{v}_h) +
              m_{\varphi,h}(\boldsymbol{u}_h) +\varphi (\boldsymbol{\beta}_h(\boldsymbol{v}_h)) \\
            &\lesssim \|\pmb{\mathsf{S}}_h\|_{p'}^{p'} +\|\pmb{\mathsf{T}}_h\|_{p'}^{p'}
              +\|\boldsymbol{u}_h\|^p_{p,h} +\| \boldsymbol{v}_h\|^p_{p,h} + \delta^p + \delta^{p'}\,.
          \end{align*}
          Thus,
          $\pmb{\mathsf{\mathcal{T}}}_{\hspace{-1mm}h}^{\varepsilon}\colon \Sigma_h\times V_h\to
          (\Sigma_h\times V_h)^*$ is pseudo-monotone.
If we can prove its coercivity, then, owing to the pseudo-monotone operator theory (cf.\ \cite[Thm.\ 27.A]{zei-IIB}), it follows that $\pmb{\mathsf{\mathcal{T}}}_{\hspace{-1mm}h}^{\varepsilon}\colon \Sigma_h\times V_h\to (\Sigma_h\times V_h)^* $  is surjective.
To this end, we distinguish the cases $\tilde{\nabla}_h=  \pmb{\mathsf{\mathcal{G}}}_h$ and $\tilde{\nabla}_h=   \nabla_h$:

%
%In order words, we  need to distinguish the cases $\tilde{\nabla}_h=  \dgr$ and $\tilde{\nabla}_h=   \nabla_h$:

\textit{Case $\tilde{\nabla}_h=  \pmb{\mathsf{\mathcal{G}}}_h$.} 
For every $(\pmb{\mathsf{S}}_h,\boldsymbol{u}_h)\in \Sigma_h\times V_h $,
using that $\pmb{\mathsf{\mathcal{D}}}(\pmb{\mathsf{Q}}):\pmb{\mathsf{Q}}\sim \varphi^*(\vert \pmb{\mathsf{Q}}\vert
)$ (cf.~\cref{eq:hammer_inversea}) and ${\pmb{\mathsf{\mathcal{S}}}_{\boldsymbol{\beta}_h(\boldsymbol{u}_h)}(\pmb{\mathsf{Q}}):\pmb{\mathsf{Q}} \sim
	\varphi_{\boldsymbol{\beta}_h(\boldsymbol{u}_h)}(\vert \pmb{\mathsf{Q}}\vert )}$  for all
$\pmb{\mathsf{Q}}\in \mathbb{R}^{n\times d}$ (cf.\ \cref{lem:hammer,rem:sa})
as well as that  $
\varphi_{\boldsymbol{\beta}_h(\boldsymbol{u}_h)}\ge   \varphi$,
$\delta^{p'} +\varphi^*(t)\sim  t^{p'}+\delta^{p'}$,
and $\delta^p +\varphi(t)\sim  t^p+\delta^p$ for all ${t\ge  0}$, and
$\sum_{F \in \Gamma_h}h_F\,\mathscr{H}^{d-1}(F) \lesssim \vert
\Omega\vert$, 
we find that
\begin{align*}
    &\langle \pmb{\mathsf{\mathcal{T}}}_{\hspace{-1mm}h}^{\varepsilon}(\pmb{\mathsf{S}}_h,\boldsymbol{u}_h),(\pmb{\mathsf{S}}_h,\boldsymbol{u}_h)\rangle_{\Sigma_h\times V_h}
    \\
    &=(\pmb{\mathsf{\mathcal{D}}}(\pmb{\mathsf{S}}_h),\pmb{\mathsf{S}}_h)_{\Omega} +\varepsilon\,(\pmb{\mathsf{\mathcal{S}}}( \nabla_h \boldsymbol{u}_h ) ,\nabla_h\boldsymbol{u}_h)_{\Omega}
    + \alpha \, \langle \pmb{\mathsf{\mathcal{S}}}_{\boldsymbol{\beta}_h(\boldsymbol{u}_h)}( h_\Gamma^{-1}
    \llbracket{ \boldsymbol{u}_h \otimes\boldsymbol{n}}\rrbracket ) , \llbracket{ \boldsymbol{u}_h
        \otimes\boldsymbol{n}}\rrbracket\rangle_{\Gamma_h}
    \\
 %   \begin{aligned}
        & \gtrsim 
        \rho_{\varphi^*,\Omega}(\pmb{\mathsf{S}}_h)+\varepsilon\,\rho_{\varphi,\Omega}( \nabla_h \boldsymbol{u}_h )+
        \alpha \,
        m_{\varphi_{\boldsymbol{\beta}_h(\boldsymbol{u}_h)},h}(\boldsymbol{u}_h) 
        \\
        & \gtrsim  \|\pmb{\mathsf{S}}_h\|_{p',\Omega}^{p'}+\ \min\{\varepsilon,\alpha\}
        \,\|\boldsymbol{u}_h\|_{p,h}^p-\delta^p\,\vert
        \Omega\vert\,(\varepsilon+\alpha) -\delta^{p'}\,\vert
        \Omega\vert\,.
    \end{align*}
%\end{multline*}
%From \cref{eq:discrete_mixed_DG}, we obtain $\dgr \bu_h=
%\Pi^{k-1}_h\fD(\BS_h)$.  Using this, the Orlicz stability of $\Pi^{k-1}_h$~(cf.~\mbox{\cite[(A.11)]{dkrt-ldg}}), that~$\fD$~has $(p',\delta^{p-1})$-structure (cf.\ \cref{lem:D_structure}), and the equivalence \cref{eq:psi'}, we obtain
%\begin{align*}
%	\rho_{\varphi,\Omega}(\dgr \bu_h)
%	=\rho_{\varphi,\Omega}(\Pi^{k-1}_h\fD(\BS_h)) \lesssim
%	\rho_{\varphi,\Omega}((\varphi^*)'(\BS_h))\lesssim  \rho_{\varphi^*,\Omega}(\BS_h)\,.
%\end{align*}
%This together with \cref{eq:coer} and $
%\varphi_{\bbeta_h(\bu_h)}\hspace*{-0.1em}\ge\hspace*{-0.1em}  \varphi$,
%$\delta^{p'} +\varphi^*(t)\hspace*{-0.1em}\sim\hspace*{-0.1em} t^{p'}+\delta^{p'}$,
%and ${\delta^p +\varphi(t)\hspace*{-0.1em}\sim\hspace*{-0.1em} t^p+\delta^p}$~for~all~${t\hspace*{-0.1em}\ge\hspace*{-0.1em} 0}$, and
%$h\,\mathscr{H}^{d-1}(\Gamma_h)\leq c\, \vert
%\Omega\vert$,  yields  that
%\begin{align*}
%	\langle
%	\fTh(\BS_h,\bu_h),(\BS_h,\bu_h)^\top\rangle_{\Sigma_h\times
%		V_h}
%	&\gtrsim
%	\rho_{\varphi^*,\Omega}(\BS_h)+ \rho_{\varphi,\Omega}(\dgr \bu_h)+\alpha \,
%	m_{\varphi_{\bbeta_h(\bu_h)},h}(\bu_h)
%	\\
%	& \gtrsim  \|\BS_h\|_{p',\Omega}^{p'}+\ \min\{1,\alpha\}
%	\,\|\bu_h\|_{p,h}^p-\delta^p\,\vert
%	\Omega\vert\,(1+\alpha) -\delta^{p'}\,\vert
%	\Omega\vert\,.
%\end{align*}
Putting everything together,  $\pmb{\mathsf{\mathcal{T}}}_{\hspace{-1mm}h}^{\varepsilon}\colon \Sigma_h\times V_h\to (\Sigma_h\times V_h)^*$ is well-defined, continuous, and coercive and, thus, surjective (cf.\ \cite[Thm.\ 27.A]{zei-IIB}).

\textit{Case $\tilde{\nabla}_h=  \nabla_h$.} 
For every $(\pmb{\mathsf{S}}_h,\boldsymbol{u}_h)\in \Sigma_h\times V_h $, proceeding as before, also using the $\kappa$-Young inequality \cref{ineq:young} with $\psi=\vert \cdot\vert^{p'}$ and that $\|\boldsymbol{\mathcal{R}}_h  \boldsymbol{u}_h\|_{p,\Omega}\lesssim  \|h^{1/p'}_\Gamma\llbracket{\boldsymbol{u}_h\otimes \boldsymbol{n}}\rrbracket\|_{p,\Gamma_h}$ (cf.\ \cite[(A.25)]{dkrt-ldg}), 
it holds that
\begin{align}\label{prop:mixed_existence.1}
	\begin{aligned}
	\langle \pmb{\mathsf{\mathcal{T}}}_{\hspace{-1mm}h}^{\varepsilon}(\pmb{\mathsf{S}}_h,\boldsymbol{u}_h),(\pmb{\mathsf{S}}_h,\boldsymbol{u}_h)\rangle_{\Sigma_h\times V_h} &=(\pmb{\mathsf{\mathcal{D}}}(\pmb{\mathsf{S}}_h)-\boldsymbol{\mathcal{R}}_h  \boldsymbol{u}_h,\pmb{\mathsf{S}}_h)_{\Omega}+\varepsilon\,(\pmb{\mathsf{\mathcal{S}}}( \nabla_h \boldsymbol{u}_h ) ,\nabla_h\boldsymbol{u}_h)_{\Omega}\\&\quad+ \alpha\, \langle \pmb{\mathsf{\mathcal{S}}}_{\boldsymbol{\beta}_h(\boldsymbol{u}_h)}( h_\Gamma^{-1} \llbracket{ \boldsymbol{u}_h \otimes\boldsymbol{n}}\rrbracket ) , \llbracket{ \boldsymbol{u}_h \otimes\boldsymbol{n}}\rrbracket\rangle_{\Gamma_h}
	\\&	\ge (c-\kappa)\, \|\pmb{\mathsf{S}}_h\|_{p',\Omega}^{p'}+\min\{\varepsilon,\alpha-c_\kappa\}  \,\|\boldsymbol{u}_h\|_{p,h}^p\\&\quad -c\,\delta^p\,\vert \Omega\vert\,(\varepsilon+\alpha) -c\,\delta^{p'}\,\vert
	\Omega\vert\,.
	\end{aligned}
\end{align}
Therefore, choosing first $\kappa>0$ sufficiently small and, subsequently, $\alpha>0$ sufficiently large~in~\cref{prop:mixed_existence.1}, for every $(\pmb{\mathsf{S}}_h,\boldsymbol{u}_h)\in \Sigma_h\times V_h $, we conclude that
\begin{align*}
	&\langle \pmb{\mathsf{\mathcal{T}}}_{\hspace{-1mm}h}^{\varepsilon}(\pmb{\mathsf{S}}_h,\boldsymbol{u}_h),(\pmb{\mathsf{S}}_h,\boldsymbol{u}_h)\rangle_{\Sigma_h\times V_h} \\
    &\gtrsim  \|\pmb{\mathsf{S}}_h\|_{p',\Omega}^{p'}+\ \min\{\varepsilon,\alpha\}
	\,\|\boldsymbol{u}_h\|_{p,h}^p-\delta^p\,\vert
	\Omega\vert\,(\varepsilon+\alpha) -\delta^{p'}\,\vert
	\Omega\vert\,.
\end{align*}
Putting everything together,  $\pmb{\mathsf{\mathcal{T}}}_{\hspace{-1mm}h}^{\varepsilon}\colon \Sigma_h\times V_h\to (\Sigma_h\times V_h)^*$ is well-defined, continuous, and coercive and, thus, surjective (cf.\ \cite[Thm.\ 27.A]{zei-IIB}).

Therefore, for every $\varepsilon,h>0$ and $\tilde{\nabla}_h\in \{\pmb{\mathsf{\mathcal{G}}}_h,\nabla_h\}$, there exists $(\pmb{\mathsf{S}}_h^{\varepsilon},\boldsymbol{u}_h^{\varepsilon})\in \Sigma_h\times V_h$ such that 
$\pmb{\mathsf{\mathcal{T}}}_{\hspace{-1mm}h}^{\varepsilon}(\pmb{\mathsf{S}}_h^{\varepsilon},\boldsymbol{u}_h^{\varepsilon})=(\boldsymbol{0},\pmb{\mathsf{\mathcal{E}}}_h^*\boldsymbol{f})$ in $(\Sigma_h\times V_h)^*\cong \Sigma_h^*\times V_h^*$, i.e., 
for every $(\pmb{\mathsf{T}}_h,\boldsymbol{z}_h)\in \Sigma_h\times V_h$, it holds that
\begin{subequations}
  \label{eq:reg_equation}
  \begin{align}
    %\SwapAboveDisplaySkip
    \label{eq:reg_equation_1}
    (\pmb{\mathsf{\mathcal{D}}}(\pmb{\mathsf{S}}_h^{\varepsilon}),\pmb{\mathsf{T}}_h)_{\Omega}
    -(\pmb{\mathsf{T}}_h ,\tilde{\nabla}_h \boldsymbol{u}_h^{\varepsilon})_{\Omega}&=0\,,\\
    \label{eq:reg_equation_2}
    \begin{multlined}
      ( \pmb{\mathsf{S}}_h^{\varepsilon}, \pmb{\mathsf{\mathcal{G}}}_h \boldsymbol{z}_h )_{\Omega}
      +\varepsilon\,(\pmb{\mathsf{\mathcal{S}}}(\nabla_h \boldsymbol{u}_h^{\varepsilon}),\nabla_h \boldsymbol{z}_h)_\Omega \\
      +\alpha\, \langle \pmb{\mathsf{\mathcal{S}}}_{\boldsymbol{\beta}_h(\boldsymbol{u}_h^{\varepsilon})}( h_\Gamma^{-1} \llbracket{ \boldsymbol{u}_h^{\varepsilon} \otimes\boldsymbol{n}}\rrbracket ) , \llbracket{ \boldsymbol{z}_h \otimes\boldsymbol{n}}\rrbracket\rangle_{\Gamma_h}
    \end{multlined}
    &
    \begin{aligned}
      &\\
      &= \langle \boldsymbol{f}, \pmb{\mathsf{\mathcal{E}}}_h  \boldsymbol{z}_h \rangle_{V}\,.
    \end{aligned}
  \end{align}
\end{subequations}
Then, choosing $\boldsymbol{z}_h=\boldsymbol{u}_h^{\varepsilon}\in V_h$,
$\pmb{\mathsf{T}}_h=\pmb{\mathsf{S}}_h^\varepsilon \in \Sigma_h$ in \cref{eq:reg_equation}, and
using \cref{rem:grad_tilde}, \cref{eq:equivalence_without_E} we find that
\begin{align}\label{prop:mixed_existence.2}
	\begin{aligned}
			\langle \boldsymbol{f}, \pmb{\mathsf{\mathcal{E}}}_h  \boldsymbol{u}_h^{\varepsilon} \rangle_{V}&=(\pmb{\mathsf{\mathcal{D}}}(\pmb{\mathsf{S}}_h^{\varepsilon}),\pmb{\mathsf{S}}_h^{\varepsilon})_{\Omega} +\varepsilon\,(\pmb{\mathsf{\mathcal{S}}}( \nabla_h \boldsymbol{u}_h^{\varepsilon} ) ,\nabla_h\boldsymbol{u}_h^{\varepsilon})_{\Omega}\\&\quad+ \alpha
			\,	\langle \pmb{\mathsf{\mathcal{S}}}_{\boldsymbol{\beta}_h(\boldsymbol{u}_h^{\varepsilon})}( h_\Gamma^{-1}
			\llbracket{ \boldsymbol{u}_h^{\varepsilon}\otimes\boldsymbol{n}}\rrbracket ) , \llbracket{ \boldsymbol{u}_h^{\varepsilon}
				\otimes\boldsymbol{n}}\rrbracket\rangle_{\Gamma_h} 
			\\&\gtrsim 
			\rho_{\varphi^*,\Omega}(\pmb{\mathsf{S}}_h^{\varepsilon})+\varepsilon\,\rho_{\varphi,\Omega}( \nabla_h \boldsymbol{u}_h^{\varepsilon} )+
			\alpha \,
			m_{\varphi_{\boldsymbol{\beta}_h(\boldsymbol{u}_h^{\varepsilon})},h}(\boldsymbol{u}_h^{\varepsilon}) \,.
	\end{aligned}
\end{align}
From \cref{eq:reg_equation_1}, we obtain $\tilde{\nabla}_h\boldsymbol{u}_h^{\varepsilon}=
\Pi^{k-1}_h\pmb{\mathsf{\mathcal{D}}}(\pmb{\mathsf{S}}_h^{\varepsilon})$.  Using this, the Orlicz stability of $\Pi^{k-1}_h$~(cf.~\mbox{\cite[(A.11)]{dkrt-ldg}}), that~$\pmb{\mathsf{\mathcal{D}}}$~has $(p',\delta^{p-1})$-structure (cf.\ \cref{lem:D_structure}), and the equivalence \cref{eq:psi'}, we obtain
\begin{align}
  \SwapAboveDisplaySkip
  \label{prop:mixed_existence.3}
	\begin{aligned}
	\rho_{\varphi,\Omega}(\tilde{\nabla}_h \boldsymbol{u}_h^{\varepsilon})
	&=\rho_{\varphi,\Omega}(\Pi^{k-1}_h\pmb{\mathsf{\mathcal{D}}}(\pmb{\mathsf{S}}_h^{\varepsilon})) \\&\lesssim
	\rho_{\varphi,\Omega}((\varphi^*)'(\pmb{\mathsf{S}}_h^{\varepsilon}))\\&\lesssim  \rho_{\varphi^*,\Omega}(\pmb{\mathsf{S}}_h^{\varepsilon})\,.
	\end{aligned}
\end{align}
Using \cref{prop:mixed_existence.3} in \cref{prop:mixed_existence.2}  and $
\varphi_{\boldsymbol{\beta}_h(\boldsymbol{u}_h)}\ge  \varphi$,
$\delta^{p'} +\varphi^*(t)\sim t^{p'}+\delta^{p'}$,
and ${\delta^p +\varphi(t)\sim  t^p+\delta^p}$ for all ${t\ge 0}$, 
$\sum_{F \in \Gamma_h}h_F\,\mathscr{H}^{d-1}(F) \lesssim \vert 
\Omega\vert$, and \cref{eq:eqiv0}  if $\tilde \nabla _h=\pmb{\mathsf{\mathcal{G}}}_h$, yields  that
\begin{align}\label{prop:mixed_existence.4}
	\begin{aligned}
		\langle \boldsymbol{f}, \pmb{\mathsf{\mathcal{E}}}_h  \boldsymbol{u}_h^{\varepsilon}\rangle_V
	&\gtrsim
	\rho_{\varphi^*,\Omega}(\pmb{\mathsf{S}}_h^{\varepsilon})+ (1+\varepsilon)\,\rho_{\varphi,\Omega}(\tilde{\nabla}_h \boldsymbol{u}_h^{\varepsilon})+\alpha \,
	m_{\varphi_{\boldsymbol{\beta}_h(\boldsymbol{u}_h^{\varepsilon})},h}(\boldsymbol{u}_h^{\varepsilon})
	\\
	& \gtrsim  \|\pmb{\mathsf{S}}_h^{\varepsilon}\|_{p',\Omega}^{p'}+\ \min\{1,\alpha\}
	\,\|\boldsymbol{u}_h^{\varepsilon}\|_{p,h}^p-\delta^p\,\vert
	\Omega\vert\,(1+\alpha) -\delta^{p'}\,\vert
	\Omega\vert\,.
\end{aligned}
\end{align}
Crucially, the constant in the last inequality is independent of $\varepsilon$. On the other hand, using the $\kappa$-Young inequality \cref{ineq:young} with $\psi=\vert \cdot\vert^p$ and \cref{prop:n-function_E} with $\psi=\vert \cdot\vert^p$, we find that
\begin{align}
    \SwapAboveDisplaySkip
    \label{prop:mixed_existence.5}
	\begin{aligned}
	\vert  	\langle \boldsymbol{f}, \pmb{\mathsf{\mathcal{E}}}_h  \boldsymbol{u}_h^{\varepsilon}\rangle_V\vert &\leq c_\kappa\,\|\boldsymbol{f}\|_{V^*}^{p'}+\kappa\,\| \pmb{\mathsf{\mathcal{E}}}_h  \boldsymbol{u}_h^{\varepsilon}\|_V^p
	\\&\lesssim c_\kappa\,\|\boldsymbol{f}\|_{V^*}^{p'}+\kappa\,\|  \boldsymbol{u}_h^{\varepsilon}\|_{p,h}^p\,.
		\end{aligned}
\end{align}
Therefore, combining \cref{prop:mixed_existence.4,prop:mixed_existence.5} as well as choosing $\kappa>0$ sufficiently small, we~conclude~that
\begin{align*}
    \SwapAboveDisplaySkip
	\|\pmb{\mathsf{S}}_h^{\varepsilon}\|_{p',\Omega}^{p'}+\ \min\{1,\alpha\}
	\,\|\boldsymbol{u}_h^{\varepsilon}\|_{p,h}^p\lesssim 1\,.
\end{align*}
Therefore, due to the finite dimensionality of $\Sigma_h$ and  $V_h$, the Bolzano--Weierstra\ss{} compactness theorem yields the existence of a not re-labeled subsequence as well as limits $\pmb{\mathsf{S}}_h\in \Sigma_h$ and $\boldsymbol{u}_h\in V_h$ such that
\begin{align}\label{prop:mixed_existence.6}
	\begin{aligned}
		\pmb{\mathsf{S}}_h^{\varepsilon}&\to \pmb{\mathsf{S}}_h&&\quad \text{strongly in }\Sigma_h&&\quad (\varepsilon\to 0)\,,\\
		\boldsymbol{u}_h^{\varepsilon}&\to \boldsymbol{u}_h&&\quad \text{strongly in }V_h&&\quad (\varepsilon\to 0)\,.
	\end{aligned}
\end{align}
Using \cref{prop:mixed_existence.6}, by passing for $\varepsilon\to 0$ in  \cref{eq:reg_equation}, we conclude that
\begin{align*}
	\begin{aligned}
		(\pmb{\mathsf{\mathcal{D}}}(\pmb{\mathsf{S}}_h),\pmb{\mathsf{T}}_h)_{\Omega}
		-(\pmb{\mathsf{T}}_h ,\tilde{\nabla}_h \boldsymbol{u}_h )_{\Omega}&=0\,,\\
		( \pmb{\mathsf{S}}_h , \pmb{\mathsf{\mathcal{G}}}_h \boldsymbol{z}_h )_{\Omega}+ \alpha\, \langle \pmb{\mathsf{\mathcal{S}}}_{\boldsymbol{\beta}_h(\boldsymbol{u}_h)}( h_\Gamma^{-1} \llbracket{ \boldsymbol{u}_h \otimes\boldsymbol{n}}\rrbracket ) , \llbracket{ \boldsymbol{z}_h \otimes\boldsymbol{n}}\rrbracket\rangle_{\Gamma_h}
		&=
		\langle \boldsymbol{f}, \pmb{\mathsf{\mathcal{E}}}_h  \boldsymbol{z}_h \rangle_{V}\,,
	\end{aligned}
\end{align*}
i.e., $(\pmb{\mathsf{S}}_h,\boldsymbol{u}_h)\in \Sigma_h\times V_h $ is a solution of the discrete mixed formulation \cref{eq:discrete_mixed_DG}.
\end{proof}

Let us next prove a best-approximation result similar to
\cref{thm:error_primal}, but now for the mixed
formulation \cref{eq:discrete_mixed_DG}. To this end, we need
to restrict ourselves either to the case of element-wise
affine functions or the case $p\leq 2$ together with the
assumption that $(\delta+\vert \overline{\nabla \boldsymbol{u}}\vert )^{2-p}$, where $\overline{\nabla \boldsymbol{u}}\coloneqq \nabla \boldsymbol{u}$ in $\Omega$ and $\overline{\nabla \boldsymbol{u}}\coloneqq \boldsymbol{0}$ in $\mathbb{R}^d\setminus \Omega$, belongs to the Muckenhoupt class $A_2(\ensuremath{\mathbb{R}}^d)$.~More~precisely,~given~$p\in [1,\infty)$, a weight  $\sigma\colon \mathbb{R}^d\to (0,+\infty)$, i.e., $\sigma \in L^1_{\textup{loc}}(\mathbb{R}^d)$ and $0<\sigma(x)<+\infty$ for a.e.\ $x\in \mathbb{R}^d$, is said to satisfy the \textit{$A_p$-condition}, if
\begin{align*}
[\sigma]_{A_p(\mathbb{R}^d)}\coloneqq 	\sup_{B\subseteq \mathbb{R}^d;B\text{ ball}}{\langle \sigma\rangle_B(\langle \sigma^{1-p'}\rangle_B)^{p-1}}<\infty\,.
\end{align*}
We denote by $A_p(\mathbb{R}^d)$ the \textit{class of all weights
satisfying the $A_p$-condition} and use \textit{weighted Lebesgue
spaces} $L^p(\Omega;\sigma)$ equipped with the norm
$\|v\|_{p,\sigma,\Omega}\coloneqq  (\int_\Omega
\left| v \right|^p\, \sigma \, \mathrm{d}x)^{\smash{\frac1p}}$.

\begin{Remark}[Comments on Muckenhoupt weights]%; \protect\cite[Expl. 1.2.5]{Turesson.2000}]
{\rm
	\begin{itemize}
		\item[(i)] If $\sigma \in L^1_{\textup{loc}}(\mathbb{R}^d)$ is a weight and there exist $c,C>0$ such that $c\leq \sigma\leq C$ a.e.\ in $\mathbb{R}^d$, then $\sigma\in A_p(\mathbb{R}^d)$ for all $p\in [1,\infty)$.
		\item[(ii)] If $\sigma\coloneqq \vert \cdot\vert^{\eta}\in L^1_{\textup{loc}}(\mathbb{R}^d)$, then $\sigma\in A_p(\mathbb{R}^d)$, $p\in (1,\infty)$, if $-d<\eta <d(p-1)$.
		% \item[(iii)] If $\sigma \in L^1_{\textup{loc}}(\mathbb{R}^d)$ is a weight, then $\sigma^{\eta}\in A_2(\mathbb{R}^d)$ for some $\eta>0$ if and only if ${\log \sigma\in \textrm{BMO}(\mathbb{R}^d)}$.
	\end{itemize}
}
\end{Remark}

\begin{Theorem}\label{thm:error_estimate_mixed} Assume that $k=1$ or that $p\leq 2$ and $(\delta+\vert \overline{\nabla \boldsymbol{u}}\vert )^{2-p}\in A_2(\mathbb{R}^d)$, where $\overline{\nabla \boldsymbol{u}}\coloneqq \nabla \boldsymbol{u}$ in $\Omega$ and $\overline{\nabla \boldsymbol{u}}\coloneqq \boldsymbol{0}$ in $\mathbb{R}^d\setminus \Omega$. Then, for $\alpha>0$ sufficiently large, we have that
\begin{multline*}
	\|\pmb{\mathsf{\mathcal{F}}}(\tilde{\nabla}_h\boldsymbol{u}_h)-\pmb{\mathsf{\mathcal{F}}}(\nabla \boldsymbol{u})\|_{2,\Omega}^2
	+
	\|\pmb{\mathsf{\mathcal{F}}}^*(\pmb{\mathsf{S}}_h) - \pmb{\mathsf{\mathcal{F}}}^*(\pmb{\mathsf{S}})\|_{2,\Omega}^2+
	m_{\varphi_{\boldsymbol{\beta}_h(\boldsymbol{u}_h)},h}(\boldsymbol{u}_h)
	\\
	\lesssim
    \inf_{\substack{(\pmb{\mathsf{T}}_h,\boldsymbol{v}_h) \\ \in \Sigma_h\times V_h}}{\left(
		\|\pmb{\mathsf{\mathcal{F}}}(\nabla \boldsymbol{u})-\pmb{\mathsf{\mathcal{F}}}(\tilde{\nabla}_h \boldsymbol{v}_h)\|_{2,\Omega}^2
		+\|\pmb{\mathsf{\mathcal{F}}}^*(\pmb{\mathsf{S}}) - \pmb{\mathsf{\mathcal{F}}}^*(\pmb{\mathsf{T}}_h)\|_{2,\Omega}^2+
		m_{\varphi_{\boldsymbol{\beta}_h(\boldsymbol{u}_h)},h}(\boldsymbol{v}_h)\right)}\,.
\end{multline*}
where the constant depends only on $k$, $\omega_0$, $[\sigma]_{A_2(\mathbb{R}^d)}$, and the characteristics of~$\pmb{\mathsf{\mathcal{S}}}$ and~$\pmb{\mathsf{\mathcal{D}}}$.
\end{Theorem}

Key ingredient in the proof of \cref{thm:error_estimate_mixed} is the following convex conjugation type inequality for the natural distance, which applies, if, e.g., $k=1$ or if $p\leq 2$ and $(\delta+\vert \overline{\nabla \boldsymbol{u}}\vert )^{2-p}\in A_2(\mathbb{R}^d)$. In particular, note that if one is capable of establishing these convex conjugation formula for more general assumptions, then \cref{thm:error_estimate_mixed} immediately applies for these assumptions.

\begin{Lemma}\label{lem:convex_conjugation} The following statements apply:
\begin{itemize}
	\item[(i)] If $k=1$, then   for every $\boldsymbol{v}_h,\boldsymbol{w}_h\in V_h $, it holds that
	\begin{align*}
		&\|\pmb{\mathsf{\mathcal{F}}}(\tilde{\nabla}_h  \boldsymbol{v}_h)-\pmb{\mathsf{\mathcal{F}}}(\tilde{\nabla}_h   \boldsymbol{w}_h)\|_{2,\Omega}^2 \\
        &\lesssim \sup_{\pmb{\mathsf{T}}_h\in \Sigma_h}{\big[(\tilde{\nabla}_h  \boldsymbol{v}_h-\tilde{\nabla}_h   \boldsymbol{w}_h,\pmb{\mathsf{T}}_h)_\Omega-\tfrac{1}{c} 
			\rho_{(\varphi_{\vert \tilde{\nabla}_h   \boldsymbol{v}_h\vert})^*,\Omega}(\pmb{\mathsf{T}}_h)
			\big]}\,,
	\end{align*}
	where the constants depend only on $k$, $\omega_0$, $[\sigma]_{A_2(\mathbb{R}^d)}$, and the characteristics~of~$\pmb{\mathsf{\mathcal{S}}}$~and~$\pmb{\mathsf{\mathcal{D}}}$.
	
	\item[(ii)] If $p\leq 2$, $\pmb{\mathsf{\mathcal{F}}}(\nabla  \boldsymbol{u}) \in
          L^2(\Omega)^{n\times d}$ and $(\delta+\vert\overline{\nabla \boldsymbol{u}}\vert )^{2-p}\in A_2(\mathbb{R}^d)$, then 
	for every $\boldsymbol{v}_h,\boldsymbol{w}_h\in V_h $, it holds that
	\begin{align*}
		&\|\pmb{\mathsf{\mathcal{F}}}(\tilde{\nabla}_h  \boldsymbol{v}_h)-\pmb{\mathsf{\mathcal{F}}}(\tilde{\nabla}_h  \boldsymbol{w}_h)\|_{2,\Omega}^2 \\
        &\lesssim \sup_{\pmb{\mathsf{T}}_h\in \Sigma_h}{\big[(\tilde{\nabla}_h  \boldsymbol{v}_h-\tilde{\nabla}_h \boldsymbol{w}_h,\pmb{\mathsf{T}}_h)_\Omega-\tfrac{1}{c} 
			\rho_{(\varphi_{\smash{\vert\tilde{\nabla}_h  \boldsymbol{v}_h\vert}})^*,\Omega}(\pmb{\mathsf{T}}_h)
			\big]}\\
		&\quad +\|\pmb{\mathsf{\mathcal{F}}}(\tilde{\nabla}_h \boldsymbol{v}_h)-\pmb{\mathsf{\mathcal{F}}}(\nabla  \boldsymbol{u})\|_{2,\Omega}^2\,,
	\end{align*}
	where the constants depend only on $k$, $\omega_0$, $[\sigma]_{A_2(\mathbb{R}^d)}$, and the characteristics~of~$\pmb{\mathsf{\mathcal{S}}}$~and~$\pmb{\mathsf{\mathcal{D}}}$.
\end{itemize}

\end{Lemma}

\cref{lem:convex_conjugation}(ii), in turn, is essentially based on the following (local) stability result 
for the $L^2$-projection in terms of weighted Lebesgue norms. 
\begin{Lemma}\label{lem:weigthed-stability}
If $\sigma \hspace*{-0.1em}\in \hspace*{-0.1em}A_p(\mathbb{R}^d)$, $p\hspace*{-0.1em}\in\hspace*{-0.1em} (1,\infty)$, then for every $\pmb{\mathsf{T}}\hspace*{-0.1em}\in\hspace*{-0.1em} L^p(\Omega;\sigma)^{n\times d}$ and $K\hspace*{-0.1em}\in\hspace*{-0.1em} \mathcal{T}_h$,~we~have~that
\begin{align*}
    \SwapAboveDisplaySkip
	\|\Pi_h^{k-1} \pmb{\mathsf{T}}\|_{p,\sigma,K}\lesssim \|\pmb{\mathsf{T}}\|_{p,\sigma,K}\,,
\end{align*}
where the constant depends only on $k$, $\omega_0$, $p$, and $[\sigma]_{A_p(\mathbb{R}^d)}$.
\end{Lemma}

\begin{proof} 
 Using H\"older's inequality, a local  norm equivalence  (cf.\
 \cite[Lem.\ 12.1]{EG.2021}), the $L^1$-stability~of $\Pi_h^{k-1}$  (cf.~\cite{dkrt-ldg}),   $\vert
B_{h_K}(x_K)\vert\sim \vert K\vert$,
where~$x_K$~is~the~barycenter~of~$K$, and $\sigma\in
A_p(\mathbb{R}^d)$, we observe that
\begin{align*}
	\|\Pi_h^{k-1} \pmb{\mathsf{T}}\|_{p,\sigma,K}&\le\| \Pi_h^{k-1}\pmb{\mathsf{T}}\|_{\infty,K}\|\sigma\|_{1,K}^{1/p}\\&\lesssim\vert K\vert^{-1}\|\Pi_h^{k-1}\pmb{\mathsf{T}}\|_{1,K} \|\sigma\|_{1,K}^{1/p}
	\\&\lesssim	\vert K\vert^{-1}\|\pmb{\mathsf{T}}\|_{1,K} \|\sigma\|_{1,K}^{1/p}
	\\&=	\vert K\vert^{-1}\|\pmb{\mathsf{T}}\sigma^{1/p}\sigma^{-1/p}\|_{1,K} \|\sigma\|_{1,K}^{1/p}
	\\&\lesssim \|\pmb{\mathsf{T}}\|_{p,\sigma,K}\|\vert K\vert^{-1}\sigma^{1-p'}\|_{1,K}^{1/p'} \|\vert K\vert^{-1}\sigma\|_{1,K}^{1/p}
	\\&\lesssim	\|\pmb{\mathsf{T}}\|_{p,\sigma,K}\|\vert B_{h_K}(x_K)\vert^{-1}\sigma^{1-p'}\|_{1,B_{h_K}(x_K)}^{1/p'} \|\vert B_{R_K}(x_K)\vert^{-1}\sigma\|_{1,B_{h_K}(x_K)}^{1/p}
	\\&\lesssim \|\pmb{\mathsf{T}}\|_{p,\sigma,K}\,,
\end{align*}where the constants depend only on $k$, $\omega_0$, $p$, and $[\sigma]_{A_p(\mathbb{R}^d)}$.
\end{proof}

\begin{proof}[Proof (of \cref{lem:convex_conjugation})]

\textit{ad (i).}
For $\tilde \nabla \boldsymbol{v}_h \in \Sigma_h$, the~local Orlicz-stability of $\Pi_h^{k-1}$
(cf.~\cite[(A.11)]{dkrt-ldg}), for every $\pmb{\mathsf{T}}\in
L^{\smash{(\varphi_{\smash{\vert \tilde{\nabla}_h  \boldsymbol{v}_h\vert}})^*}
}(\Omega)^{n\times d}$, implies that
\begin{align*}
	\rho_{(\varphi_{\smash{\vert \tilde{\nabla}_h   \boldsymbol{v}_h\vert}})^*,\Omega}(\Pi_h^{k-1}\pmb{\mathsf{T}})\leq \rho_{(\varphi_{\smash{\vert \tilde{\nabla}_h   \boldsymbol{v}_h\vert}})^*,\Omega}(\pmb{\mathsf{T}})\,.        
\end{align*}
Using this, \cref{lem:hammer}, that $\varphi_{\smash{\vert \tilde{\nabla}_h
	\boldsymbol{v}_h\vert}}=(\varphi_{\smash{\vert\tilde{\nabla}_h
	\boldsymbol{v}_h\vert}})^{**}=((\varphi_{\smash{\vert \tilde{\nabla}_h
	\boldsymbol{v}_h\vert}})^{*})^{*}$, the convex conjugation formula for
integral functionals (cf.\ \cite[Prop.~1.2]{ET.1999}), and  ${\tilde
\nabla \boldsymbol{v}_h, \tilde \nabla \boldsymbol{w}_h \in \Sigma_h}$,
we find that
\begin{align*}
	&\|\pmb{\mathsf{\mathcal{F}}}(\tilde{\nabla}_h  \boldsymbol{v}_h)-\pmb{\mathsf{\mathcal{F}}}(\tilde{\nabla}_h \boldsymbol{w}_h)\|_{2,\Omega}^2
	\lesssim\rho_{\varphi_{\smash{\vert\tilde{\nabla}_h  \boldsymbol{v}_h\vert}},\Omega}(\tilde{\nabla}_h   \boldsymbol{v}_h-\tilde{\nabla}_h   \boldsymbol{w}_h
	)
	\\&=\rho_{(\varphi_{\vert\tilde{\nabla}_h \boldsymbol{v}_h\vert})^{**},\Omega}(\tilde{\nabla}_h   \boldsymbol{v}_h-\tilde{\nabla}_h   \boldsymbol{w}_h)
	\\&=  \sup_{\pmb{\mathsf{T}}\in L^{\smash{(\varphi_{\smash{\vert \tilde{\nabla}_h  \boldsymbol{v}_h\vert}})^*} }(\Omega)^{n\times d}}{\big[(\tilde{\nabla}_h   \boldsymbol{v}_h-\tilde{\nabla}_h   \boldsymbol{w}_h,\pmb{\mathsf{T}})_\Omega-
		\rho_{(\varphi_{\vert \tilde{\nabla}_h\boldsymbol{v}_h\vert})^*,\Omega}(\pmb{\mathsf{T}})}
	\\&=  \sup_{\pmb{\mathsf{T}}\in L^{\smash{(\varphi_{\smash{\vert \tilde{\nabla}_h  \boldsymbol{v}_h\vert}})^*}}(\Omega)^{n\times d}}{\big[(\tilde{\nabla}_h   \boldsymbol{v}_h-\tilde{\nabla}_h   \boldsymbol{w}_h,\Pi_h^{k-1}\pmb{\mathsf{T}})_\Omega-
		\rho_{(\varphi_{\vert \tilde{\nabla}_h\boldsymbol{v}_h\vert})^*,\Omega}(\pmb{\mathsf{T}})}
	\\&\lesssim  \sup_{\pmb{\mathsf{T}}\in L^{\smash{(\varphi_{\smash{\vert \tilde{\nabla}_h  \boldsymbol{v}_h\vert}})^*}}(\Omega)^{n\times d}}{\big[(\tilde{\nabla}_h   \boldsymbol{v}_h-\tilde{\nabla}_h   \boldsymbol{w}_h,\Pi_h^{k-1}\pmb{\mathsf{T}})_\Omega-\tfrac{1}{c}
		\rho_{(\varphi_{\vert\tilde{\nabla}_h \boldsymbol{v}_h\vert})^*,\Omega}(\Pi_h^{k-1}\pmb{\mathsf{T}})}
	\\&\lesssim  \sup_{\pmb{\mathsf{T}}_h\in \Sigma_h}{\big[(\tilde{\nabla}_h   \boldsymbol{v}_h-\tilde{\nabla}_h \boldsymbol{w}_h,\pmb{\mathsf{T}}_h)_\Omega-\tfrac{1}{c}
		\rho_{(\varphi_{\vert \tilde{\nabla}_h   \boldsymbol{v}_h\vert})^*,\Omega}(\pmb{\mathsf{T}}_h)}\,,
\end{align*}
which is the claimed convex conjugation type inequality for the natural distance in the case $k=1$.

\textit{ad (ii).}
Abbreviate $\sigma\coloneqq (\delta+\vert \overline{\nabla \boldsymbol{u}}\vert )^{2-p}$. Due $p\leq 2$, for a.e.\ $x\in \Omega$ and every $t\ge 0$, it holds
that
\begin{align}
  \SwapAboveDisplaySkip
  \label{lem:convex_conjugation.1}
  \begin{aligned}
    &\bigl((\delta + \left| \nabla\boldsymbol{u}(x) \right|)^{p-1}+ t\bigr)^{p'-2} \, t^2 \\
    &\quad \leq 2^{p'-2}\,( \sigma(x) \,t^2 +\, t^{p'}) \leq 2^{p'-2}\,
    \bigl((\delta + \left| \nabla\boldsymbol{u}(x) \right|)^{p-1} + t\bigr)^{p'-2} \,
    t^2\,.
  \end{aligned}
\end{align}
Therefore, using \cref{lem:convex_conjugation.1},
the stability of $\Pi_h^{k-1}$ in  $L^{p'}(\Omega)^{n \times d}$
(cf.~\cite[Cor.~A.8]{kr-phi-ldg}), and in $L^2(\Omega,\sigma)^{n\times
d}$ (cf.~\cref{lem:weigthed-stability}) together with
$(\delta+\vert \overline{\nabla \boldsymbol{u}}\vert )^{2-p}\in
A_2(\mathbb{R}^d)$,  for every $\pmb{\mathsf{T}}\in
L^{p'}(\Omega)^{n\times d}\cap
L^{2}(\Omega;\sigma)^{n\times d}\sim  L^{(\varphi_{\vert\nabla\boldsymbol{u}\vert})^*}(\Omega)^{n\times d}$,~it~holds~that
\begin{align}\label{lem:convex_conjugation.2}
	\begin{aligned}
		\rho_{(\varphi_{\vert\nabla\boldsymbol{u}\vert})^*}(\Pi_h^{k-1} \pmb{\mathsf{T}})
		&\lesssim \|\Pi_h^{k-1}\pmb{\mathsf{T}}\|_{2,\omega}^2+\|\Pi_h^{k-1}\pmb{\mathsf{T}}\|_{p'}^{p'}
		\\&\lesssim \|\pmb{\mathsf{T}}\|_{2,\omega}^2+\|\pmb{\mathsf{T}}\|_{p'}^{p'}
		\\&\lesssim\rho_{(\varphi_{\vert\nabla\boldsymbol{u}\vert})^*}(\pmb{\mathsf{T}})\,,
	\end{aligned}
\end{align}
which, using a~shift change in \cref{lem:shift-change}, implies that
\begin{align}\label{lem:convex_conjugation.3}
	\begin{aligned}
		\rho_{(\varphi_{\vert \tilde{\nabla}_h \boldsymbol{v}_h\vert})^*,\Omega}(\Pi_h^{k-1}\pmb{\mathsf{T}})&\lesssim \rho_{(\varphi_{\vert \nabla\boldsymbol{u}\vert})^*,\Omega}(\Pi_h^{k-1}\pmb{\mathsf{T}})+\|\pmb{\mathsf{\mathcal{F}}}(\tilde{\nabla}_h   \boldsymbol{v}_h)-\pmb{\mathsf{\mathcal{F}}}(\nabla  \boldsymbol{u})\|_{2,\Omega}^2 
		\\&	\lesssim \rho_{(\varphi_{\vert \nabla  \boldsymbol{u}\vert})^*,\Omega}(\pmb{\mathsf{T}})+\|\pmb{\mathsf{\mathcal{F}}}(\tilde{\nabla}_h \boldsymbol{v}_h)-\pmb{\mathsf{\mathcal{F}}}(\nabla  \boldsymbol{u})\|_{2,\Omega}^2 
		\\&	\lesssim \rho_{\smash{(\varphi_{\smash{\vert \tilde{\nabla}_h  \boldsymbol{v}_h\vert}})^*},\Omega}( \pmb{\mathsf{T}})+\|\pmb{\mathsf{\mathcal{F}}}(\tilde{\nabla}_h   \boldsymbol{v}_h)-\pmb{\mathsf{\mathcal{F}}}(\nabla  \boldsymbol{u})\|_{2,\Omega}^2 \,. 
	\end{aligned}
\end{align}
Using \cref{lem:convex_conjugation.3} and proceeding
as in (i), %, that $\varphi_{\vert \tilde{\nabla}_h   \bv_h\vert}=(\varphi_{\vert \tilde{\nabla}_h   \bv_h\vert})^{**}=((\varphi_{\vert \tilde{\nabla}_h   \bv_h\vert})^{*})^{*}$, and the convex conjugation formula for integral functionals (cf.\ \cite[Prop. 1.2]{ET.1999}), 
we find that
\begin{align*}
	&\|\pmb{\mathsf{\mathcal{F}}}(\tilde{\nabla}_h  \boldsymbol{v}_h)-\pmb{\mathsf{\mathcal{F}}}(\tilde{\nabla}_h   \boldsymbol{w}_h)\|_{2,\Omega}^2
	\lesssim\rho_{\varphi_{\vert\tilde{\nabla}_h   \boldsymbol{v}_h\vert},\Omega}(\tilde{\nabla}_h  \boldsymbol{v}_h-\tilde{\nabla}_h  \boldsymbol{w}_h
	)\\&=\rho_{(\varphi_{\vert \tilde{\nabla}_h   \boldsymbol{v}_h\vert})^{**},\Omega}(\tilde{\nabla}_h \boldsymbol{v}_h-\tilde{\nabla}_h   \boldsymbol{w}_h)
	\\&=  \sup_{\pmb{\mathsf{T}}\in L^{(\varphi_{\vert \tilde{\nabla}_h    \boldsymbol{v}_h\vert})^*}(\Omega)^{n\times d}}{\big[(\tilde{\nabla}_h  \boldsymbol{v}_h-\tilde{\nabla}_h  \boldsymbol{w}_h,\pmb{\mathsf{T}})_\Omega-
		\rho_{(\varphi_{\vert \tilde{\nabla}_h    \boldsymbol{v}_h\vert})^*,\Omega}(\pmb{\mathsf{T}})\big]}
	\\&=  \sup_{\pmb{\mathsf{T}}\in L^{(\varphi_{\vert \tilde{\nabla}_h    \boldsymbol{v}_h\vert})^*}(\Omega)^{n\times d}}{\big[(\tilde{\nabla}_h  \boldsymbol{v}_h-\tilde{\nabla}_h  \boldsymbol{w}_h,\Pi_h^{k-1}\pmb{\mathsf{T}})_\Omega-
		\rho_{(\varphi_{\vert \tilde{\nabla}_h    \boldsymbol{v}_h\vert})^*,\Omega}(\pmb{\mathsf{T}})\big]}
	\\&\lesssim  \sup_{\pmb{\mathsf{T}}\in L^{(\varphi_{\vert \tilde{\nabla}_h    \boldsymbol{v}_h\vert})^*}(\Omega)^{n\times d}}{\big[(\tilde{\nabla}_h   \boldsymbol{v}_h-\tilde{\nabla}_h  \boldsymbol{w}_h,\Pi_h^{k-1}\pmb{\mathsf{T}})_\Omega-\tfrac{1}{c}
		\rho_{(\varphi_{\vert \tilde{\nabla}_h   \boldsymbol{v}_h\vert})^*,\Omega}(\Pi_h^k\pmb{\mathsf{T}})\big]}
	\\
	&\quad +\|\pmb{\mathsf{\mathcal{F}}}(\tilde{\nabla}_h  \boldsymbol{v}_h)-\pmb{\mathsf{\mathcal{F}}}(\nabla  \boldsymbol{u})\|_{2,\Omega}^2
	\\&\lesssim  \sup_{\pmb{\mathsf{T}}_h\in \Sigma_h}{\big[(\tilde{\nabla}_h   \boldsymbol{v}_h-\tilde{\nabla}_h   \boldsymbol{w}_h,\pmb{\mathsf{T}}_h)_\Omega-\tfrac{1}{c}
		\rho_{(\varphi_{\vert \tilde{\nabla}_h \boldsymbol{v}_h\vert})^*,\Omega}(\pmb{\mathsf{T}}_h)\big]}
	\\
	&\quad +\|\pmb{\mathsf{\mathcal{F}}}(\tilde{\nabla}_h  \boldsymbol{v}_h)-\pmb{\mathsf{\mathcal{F}}}(\nabla  \boldsymbol{u})\|_{2,\Omega}^2\,,
\end{align*}
which is the claimed convex conjugation type inequality for the natural distance in the case $p\leq 2$ and $\sigma\in A_2(\mathbb{R}^d)$.
\end{proof}

\begin{proof}[Proof (of \cref{thm:error_estimate_mixed})]
By the continuous mixed formulation \cref{eq:mixed},  the discrete mixed formulation \cref{eq:discrete_mixed_DG} and the crucial identity \cref{eq:equivalence_without_E}, for every $(\pmb{\mathsf{T}}_h,\boldsymbol{z}_h)\in \Sigma_h\times V_h $, we have that
\begin{align}
  \label{eq:error_estimate_mixed.1}
	(\pmb{\mathsf{\mathcal{D}}}(\pmb{\mathsf{S}}_h) - \pmb{\mathsf{\mathcal{D}}}(\pmb{\mathsf{S}}),\pmb{\mathsf{T}}_h)_{\Omega}&= (\tilde{\nabla}_h \boldsymbol{u}_h-\nabla \boldsymbol{u}, \pmb{\mathsf{T}}_h)_{\Omega}\,,\\
	(\pmb{\mathsf{S}}_h-\pmb{\mathsf{S}},{\nabla}_h  \pmb{\mathsf{\mathcal{E}}}_h \boldsymbol{z}_h)_{\Omega}
	&=-\alpha\,\langle\pmb{\mathsf{\mathcal{S}}}_{\boldsymbol{\beta}_h(\boldsymbol{u}_h)}(h_{\Gamma}^{-1}\llbracket{\boldsymbol{u}_h\otimes
		\boldsymbol{n}}\rrbracket),\llbracket{\boldsymbol{z}_h\otimes
		\boldsymbol{n}}\rrbracket\rangle_{\Gamma_h}\,.\label{eq:error_estimate_mixed.1.2}
\end{align}
Therefore, taking an arbitrary $\boldsymbol{v}_h \in V_h$, using
\cref{lem:convex_conjugation},
\cref{eq:error_estimate_mixed.1}, the
$\varepsilon$-Young inequality~\cref{ineq:young} with
$\psi=\varphi_{\vert\tilde{\nabla}_h\boldsymbol{v}_h\vert}$,
\cref{eq:hammera}, and a~shift change in
\cref{lem:shift-change}, choosing $\varepsilon>0$
sufficiently small, and using 
\cref{eq:hammer_inversea} with $\nabla \boldsymbol{u} =\pmb{\mathsf{\mathcal{D}}}(\pmb{\mathsf{S}})$, we find that
\begin{align*}\label{eq:error_estimate_mixed.2}
    &\|\pmb{\mathsf{\mathcal{F}}}(\tilde{\nabla}_h \boldsymbol{u}_h)-\pmb{\mathsf{\mathcal{F}}}(\tilde{\nabla}_h  \boldsymbol{v}_h)\|_{2,\Omega}^2 \\
		&\lesssim
		\sup_{\pmb{\mathsf{T}}_h\in \Sigma_h}{\big[(\tilde{\nabla}_h \boldsymbol{u}_h-\tilde{\nabla}_h \boldsymbol{v}_h,\pmb{\mathsf{T}}_h)_\Omega-\tfrac{1}{c}
			\rho_{(\varphi_{\vert\tilde{\nabla}_h
					\boldsymbol{v}_h\vert})^*,\Omega}(\pmb{\mathsf{T}}_h)\big]}
		+	\|\pmb{\mathsf{\mathcal{F}}}(\nabla \boldsymbol{u})-\pmb{\mathsf{\mathcal{F}}}(\tilde{\nabla}_h \boldsymbol{v}_h)\|_{2,\Omega}^2\\
		&=
		\sup_{\pmb{\mathsf{T}}_h \in \Sigma_h}\big[( \pmb{\mathsf{\mathcal{D}}}(\pmb{\mathsf{S}}_h) - \pmb{\mathsf{\mathcal{D}}}(\pmb{\mathsf{S}}), \pmb{\mathsf{T}}_h)_{\Omega} + ( \nabla\boldsymbol{u} -\tilde{\nabla}_h \boldsymbol{v}_h, \pmb{\mathsf{T}}_h)_{\Omega}-\tfrac{1}{c}
		\rho_{(\varphi_{\smash{\vert\tilde{\nabla}_h  \boldsymbol{v}_h\vert}})^*,\Omega}(\pmb{\mathsf{T}}_h)\big]\\&\quad+	\|\pmb{\mathsf{\mathcal{F}}}(\nabla \boldsymbol{u})-\pmb{\mathsf{\mathcal{F}}}(\tilde{\nabla}_h \boldsymbol{v}_h)\|_{2,\Omega}^2\\
		&\leq 
		c_{\varepsilon}\, \rho_{\varphi_{\vert\tilde{\nabla}_h \boldsymbol{v}_h\vert},\Omega}(\pmb{\mathsf{\mathcal{D}}}(\pmb{\mathsf{S}}_h) - \pmb{\mathsf{\mathcal{D}}}(\pmb{\mathsf{S}}))
		+
		(c_\varepsilon+1)\,\|\pmb{\mathsf{\mathcal{F}}}(\nabla \boldsymbol{u})-\pmb{\mathsf{\mathcal{F}}}(\tilde{\nabla}_h \boldsymbol{v}_h)\|_{2,\Omega}^2\\&\quad
		+  \sup_{\pmb{\mathsf{T}}_h \in \Sigma_h}{\big[2\varepsilon\,\rho_{(\varphi_{\smash{\vert\tilde{\nabla}_h  \boldsymbol{v}_h\vert}})^*,\Omega}(\pmb{\mathsf{T}}_h)-\tfrac{1}{c}\rho_{(\varphi_{\vert\tilde{\nabla}_h \boldsymbol{v}_h\vert})^*,\Omega}(\pmb{\mathsf{T}}_h)\big]}
		%	\\&\quad+	\|\BF(\nabla \bu)-\BF(\tilde{\nabla}_h \bv_h)\|_{2,\Omega}^2
		\\&\lesssim  \rho_{\varphi_{\vert \tilde{\nabla}_h \boldsymbol{v}_h\vert},\Omega}(\pmb{\mathsf{\mathcal{D}}}(\pmb{\mathsf{S}}_h) - \pmb{\mathsf{\mathcal{D}}}(\pmb{\mathsf{S}}))
		+
		\|\pmb{\mathsf{\mathcal{F}}}(\nabla \boldsymbol{u})-\pmb{\mathsf{\mathcal{F}}}(\tilde{\nabla}_h  \boldsymbol{v}_h)\|_{2,\Omega}^2
		\\&\lesssim  \rho_{\varphi_{\vert \nabla  \boldsymbol{u}\vert},\Omega}(\pmb{\mathsf{\mathcal{D}}}(\pmb{\mathsf{S}}_h) - \pmb{\mathsf{\mathcal{D}}}(\pmb{\mathsf{S}}))
		+
		\|\pmb{\mathsf{\mathcal{F}}}(\nabla \boldsymbol{u})-\pmb{\mathsf{\mathcal{F}}}(\tilde{\nabla}_h  \boldsymbol{v}_h)\|_{2,\Omega}^2
		%	\\&\quad +  \sup_{\BT_h \in \Sigma_h}{\big[(2\varepsilon-\tfrac{1}{c})\,\rho_{(\varphi_{\vert\tilde{\nabla}_h \bv_h\vert})^*,\Omega}(\BT_h)\big]}
		\\&\lesssim  \|\pmb{\mathsf{\mathcal{F}}}^*(\pmb{\mathsf{S}}_h) - \pmb{\mathsf{\mathcal{F}}}^*(\pmb{\mathsf{S}})\|_{2,\Omega}^2
		+
		\|\pmb{\mathsf{\mathcal{F}}}(\nabla \boldsymbol{u})-\pmb{\mathsf{\mathcal{F}}}(\tilde{\nabla}_h \boldsymbol{v}_h)\|_{2,\Omega}^2
		\,.
	\end{align*}
As a direct consequence of this, we obtain
\begin{align}\label{eq:error_estimate_mixed.3}
	\|\pmb{\mathsf{\mathcal{F}}}(\tilde{\nabla}_h \boldsymbol{u}_h)\!-\!\pmb{\mathsf{\mathcal{F}}}(\nabla  \boldsymbol{u})\|_{2,\Omega}^2\lesssim \|\pmb{\mathsf{\mathcal{F}}}^*(\pmb{\mathsf{S}}_h) \!-\! \pmb{\mathsf{\mathcal{F}}}^*(\pmb{\mathsf{S}})\|_{2,\Omega}^2
	+
	\|\pmb{\mathsf{\mathcal{F}}}(\nabla \boldsymbol{u})-\pmb{\mathsf{\mathcal{F}}}(\tilde{\nabla}_h   \boldsymbol{v}_h)\|_{2,\Omega}^2\,.
\end{align}
On the other hand, for every $(\pmb{\mathsf{T}}_h,\boldsymbol{v}_h)\in
\Sigma_h\times V_h$, using \cref{eq:hammer_inversea},
\cref{eq:error_estimate_mixed.1},  and the
$\varepsilon$-Young inequality \cref{ineq:young} with $\psi =(\varphi^*)_{\vert \pmb{\mathsf{S}}\vert}\sim (\varphi_{\vert \nabla\boldsymbol{u}\vert})^*$ and $\psi =\varphi_{\vert \nabla\boldsymbol{u}\vert}$, we obtain
\begin{align}\label{eq:error_estimate_mixed.4}
%	\begin{aligned}
		\|\pmb{\mathsf{\mathcal{F}}}^*(\pmb{\mathsf{S}}_h) - \pmb{\mathsf{\mathcal{F}}}^*(\pmb{\mathsf{S}})\|_{2,\Omega}^2&\lesssim (\pmb{\mathsf{D}}(\pmb{\mathsf{S}}_h)-\pmb{\mathsf{D}}(\pmb{\mathsf{S}}),\pmb{\mathsf{S}}_h-\pmb{\mathsf{S}})_{\Omega}\notag
		\\&= (\pmb{\mathsf{D}}(\pmb{\mathsf{S}}_h)-\pmb{\mathsf{D}}(\pmb{\mathsf{S}}),\pmb{\mathsf{S}}_h-\pmb{\mathsf{T}}_h)_{\Omega}+(\pmb{\mathsf{D}}(\pmb{\mathsf{S}}_h)-\pmb{\mathsf{D}}(\pmb{\mathsf{S}}),\pmb{\mathsf{T}}_h-\pmb{\mathsf{S}})_{\Omega}\notag
		\\&= (\tilde{\nabla}_h  \boldsymbol{u}_h -\nabla \boldsymbol{u},\pmb{\mathsf{S}}_h-\pmb{\mathsf{T}}_h)_{\Omega}+(\pmb{\mathsf{D}}(\pmb{\mathsf{S}}_h)-\pmb{\mathsf{D}}(\pmb{\mathsf{S}}),\pmb{\mathsf{T}}_h-\pmb{\mathsf{S}})_{\Omega}\notag
		\\&= (\tilde{\nabla}_h  \boldsymbol{u}_h -\nabla
  \boldsymbol{u},\pmb{\mathsf{S}}-\pmb{\mathsf{T}}_h)_{\Omega}+(\tilde{\nabla}_h  \boldsymbol{u}_h -\nabla
  \boldsymbol{u},\pmb{\mathsf{S}}_h-\pmb{\mathsf{S}})_{\Omega}\notag\\&\quad
  +(\pmb{\mathsf{D}}(\pmb{\mathsf{S}}_h)-\pmb{\mathsf{D}}(\pmb{\mathsf{S}}),\pmb{\mathsf{T}}_h-\pmb{\mathsf{S}})_{\Omega}\notag %+(\BS_h-\BS,\dgr \bu_h-\dgr  \bv_h)_{\Omega}\notag
		\\&= (\tilde{\nabla}_h \boldsymbol{u}_h -\nabla \boldsymbol{u},\pmb{\mathsf{S}}-\pmb{\mathsf{T}}_h)_{\Omega}+(\tilde{\nabla}_h  \boldsymbol{v}_h -\nabla \boldsymbol{u},\pmb{\mathsf{S}}_h-\pmb{\mathsf{S}})_{\Omega}\\&\quad+(\tilde{\nabla}_h  \boldsymbol{u}_h-\tilde{\nabla}_h   \boldsymbol{v}_h, \pmb{\mathsf{S}}_h-\pmb{\mathsf{S}})_{\Omega}+(\pmb{\mathsf{D}}(\pmb{\mathsf{S}}_h)-\pmb{\mathsf{D}}(\pmb{\mathsf{S}}),\pmb{\mathsf{T}}_h-\pmb{\mathsf{S}})_{\Omega}\notag
		\\&\leq 
		\varepsilon\, \|\pmb{\mathsf{\mathcal{F}}}(\tilde{\nabla}_h  \boldsymbol{u}_h ) - \pmb{\mathsf{\mathcal{F}}}(\nabla \boldsymbol{u})\|_{2,\Omega}^2+c_\varepsilon\, \|\pmb{\mathsf{\mathcal{F}}}^*(\pmb{\mathsf{T}}_h) - \pmb{\mathsf{\mathcal{F}}}^*(\pmb{\mathsf{S}})\|_{2,\Omega}^2 \notag
		\\&\quad +\varepsilon\,  \|\pmb{\mathsf{\mathcal{F}}}^*(\pmb{\mathsf{S}}_h) - \pmb{\mathsf{\mathcal{F}}}^*(\pmb{\mathsf{S}})\|_{2,\Omega}^2+c_\varepsilon\, \|\pmb{\mathsf{\mathcal{F}}}(\tilde{\nabla}_h   \boldsymbol{v}_h ) - \pmb{\mathsf{\mathcal{F}}}(\nabla \boldsymbol{u})\|_{2,\Omega}^2 \notag
		\\&\quad+(\tilde{\nabla}_h \boldsymbol{u}_h-\tilde{\nabla}_h  \boldsymbol{v}_h, \pmb{\mathsf{S}}_h-\pmb{\mathsf{S}})_{\Omega}\notag
		\,.
%	\end{aligned}
\end{align}
Moreover, abbreviating $\boldsymbol{z}_h\coloneqq \boldsymbol{u}_h-\boldsymbol{v}_h\in
V_h$, using \cref{eq:error_estimate_mixed.1.2} and
$\tilde{\nabla}_h  \pmb{\mathsf{\mathcal{E}}}_h \boldsymbol{z}_h=\nabla\pmb{\mathsf{\mathcal{E}}}_h \boldsymbol{z}_h$
(cf.~\cref{rem:grad_tilde}), we have that
\begin{align}
	\label{eq:error_estimate_mixed.5}
	\begin{aligned}
   & (\pmb{\mathsf{S}}_h-\pmb{\mathsf{S}},\tilde{\nabla}_h \boldsymbol{z}_h)_{\Omega} \\
		&=(\pmb{\mathsf{S}}_h-\pmb{\mathsf{S}},\tilde{\nabla}_h  (\boldsymbol{z}_h-\pmb{\mathsf{\mathcal{E}}}_h \boldsymbol{z}_h))_{\Omega}+(\pmb{\mathsf{S}}_h-\pmb{\mathsf{S}},\tilde{\nabla}_h  \pmb{\mathsf{\mathcal{E}}}_h \boldsymbol{z}_h)_{\Omega}
		\\&=(\pmb{\mathsf{S}}_h-\pmb{\mathsf{S}},\tilde{\nabla}_h  (\boldsymbol{z}_h-\pmb{\mathsf{\mathcal{E}}}_h \boldsymbol{z}_h))_{\Omega}-\alpha\,\langle\pmb{\mathsf{\mathcal{S}}}_{\boldsymbol{\beta}_h(\boldsymbol{u}_h)}(h_{\Gamma}^{-1}\llbracket{\boldsymbol{u}_h\otimes \boldsymbol{n}}\rrbracket),\llbracket{\boldsymbol{z}_h\otimes \boldsymbol{n}}\rrbracket\rangle_{\Gamma_h}
		\,.
	\end{aligned}
\end{align}
Using the $\varepsilon$-Young inequality
\cref{ineq:young} with $\psi =\varphi_{\vert \tilde{\nabla}_h\boldsymbol{u}_h\vert}$, that
$\varphi_{\vert \tilde{\nabla}_h\boldsymbol{u}_h\vert}\le
\varphi_{\boldsymbol{\beta}_h(\boldsymbol{u}_h)}$, a~shift change in \cref{lem:shift-change},
\cref{prop:n-function_E},  $\boldsymbol{z}_h =\boldsymbol{u}_h-\boldsymbol{v}_h\in V_h$, and \cref{eq:hammer_inversea} with $\nabla\boldsymbol{u} =\pmb{\mathsf{\mathcal{D}}}(\pmb{\mathsf{S}})$, we obtain
\begin{gather}
    \SwapAboveDisplaySkip
	\label{eq:error_estimate_mixed.6}
	\begin{aligned}
		&(\pmb{\mathsf{S}}_h-\pmb{\mathsf{S}}, \tilde{\nabla}_h
		(\boldsymbol{z}_h-\pmb{\mathsf{\mathcal{E}}}_h \boldsymbol{z}_h))_{\Omega}\\&\leq
		\varepsilon\,\rho_{(\varphi_{\left| \tilde{\nabla}_h\boldsymbol{u}_h \right|})^*,\Omega}(\pmb{\mathsf{S}}_h-\pmb{\mathsf{S}})+c_\varepsilon\,\rho_{\varphi_{\left| \tilde{\nabla}_h\boldsymbol{u}_h \right|},\Omega}(\tilde{\nabla}_h(\boldsymbol{z}_h-\pmb{\mathsf{\mathcal{E}}}_h\boldsymbol{z}_h))
		\\
		&\leq \varepsilon\,\rho_{(\varphi_{\left| \tilde{\nabla}_h\boldsymbol{u}_h \right|})^*,\Omega}(\pmb{\mathsf{S}}_h-\pmb{\mathsf{S}})+c_\varepsilon\,\rho_{\varphi_{\boldsymbol{\beta}_h(\boldsymbol{u}_h)},\Omega}(\tilde{\nabla}_h(\boldsymbol{z}_h-\pmb{\mathsf{\mathcal{E}}}_h \boldsymbol{z}_h))
		\\& \lesssim \varepsilon\,\rho_{(\varphi_{\left| \nabla\boldsymbol{u} \right|})^*,\Omega}(\pmb{\mathsf{S}}_h-\pmb{\mathsf{S}})+\varepsilon\,\|\pmb{\mathsf{\mathcal{F}}}(\tilde{\nabla}_h\boldsymbol{u}_h ) - \pmb{\mathsf{\mathcal{F}}}(\nabla \boldsymbol{u})\|_{2,\Omega}^2\\&\quad+c_\varepsilon\,m_{\varphi_{\boldsymbol{\beta}_h(\boldsymbol{u}_h)},h}(\boldsymbol{u}_h)+c_\varepsilon\,m_{\varphi_{\boldsymbol{\beta}_h(\boldsymbol{u}_h)},h}(\boldsymbol{v}_h)
		\\& \lesssim \varepsilon\,\|\pmb{\mathsf{\mathcal{F}}}^*(\pmb{\mathsf{S}}_h)-\pmb{\mathsf{\mathcal{F}}}^*(\pmb{\mathsf{S}})\|_{2,\Omega}^2+\varepsilon\,\|\pmb{\mathsf{\mathcal{F}}}(\tilde{\nabla}_h \boldsymbol{u}_h ) - \pmb{\mathsf{\mathcal{F}}}(\nabla \boldsymbol{u})\|_{2,\Omega}^2\\&\quad+c_\varepsilon\,m_{\varphi_{\boldsymbol{\beta}_h(\boldsymbol{u}_h)},h}(\boldsymbol{u}_h)+c_\varepsilon\,m_{\varphi_{\boldsymbol{\beta}_h(\boldsymbol{u}_h)},h}(\boldsymbol{v}_h)
		\,,
	\end{aligned}
    \shortintertext{and}
	\label{eq:error_estimate_mixed.7}
	\begin{multlined}
		(\pmb{\mathsf{\mathcal{S}}}_{\boldsymbol{\beta}_h(\boldsymbol{u}_h)}(h_{\Gamma}^{-1}\llbracket{\boldsymbol{u}_h\otimes \boldsymbol{n}}\rrbracket),\llbracket{\boldsymbol{z}_h\otimes \boldsymbol{n}}\rrbracket)_{\Omega}\\\ge (1\!-\!\varepsilon)\,m_{\varphi_{\boldsymbol{\beta}_h(\boldsymbol{u}_h)},h}(\boldsymbol{u}_h)\!-\!c_\varepsilon\,m_{\varphi_{\boldsymbol{\beta}_h(\boldsymbol{u}_h)},h}(\boldsymbol{v}_h)\,.
	\end{multlined}
\end{gather}
Combining \crefrange{eq:error_estimate_mixed.3}{eq:error_estimate_mixed.7}, choosing $\varepsilon>0$ sufficiently small, we arrive at
\begin{align*}
	&\|\pmb{\mathsf{\mathcal{F}}}(\tilde{\nabla}_h\boldsymbol{u}_h)-\pmb{\mathsf{\mathcal{F}}}(\nabla \boldsymbol{u})\|_{2,\Omega}^2
	+
	\|\pmb{\mathsf{\mathcal{F}}}^*(\pmb{\mathsf{S}}_h) - \pmb{\mathsf{\mathcal{F}}}^*(\pmb{\mathsf{S}})\|_{2,\Omega}^2+
	(\alpha-c)\,m_{\varphi_{\boldsymbol{\beta}_h(\boldsymbol{u}_h)},h}(\boldsymbol{u}_h)
    \\
      &\lesssim \inf_{(\pmb{\mathsf{T}}_h,\boldsymbol{v}_h) \in \Sigma_h\times V_h} \Bigl(
      \|\pmb{\mathsf{\mathcal{F}}}^*(\pmb{\mathsf{S}}) - \pmb{\mathsf{\mathcal{F}}}^*(\pmb{\mathsf{T}}_h)\|_{2,\Omega}^2+
      \|\pmb{\mathsf{\mathcal{F}}}(\nabla \boldsymbol{u})-\pmb{\mathsf{\mathcal{F}}}(\tilde{\nabla}_h  \boldsymbol{v}_h)\|_{2,\Omega}^2 \\
      &\quad + (\alpha+c)\,m_{\varphi_{\boldsymbol{\beta}_h(\boldsymbol{u}_h)},h}(\boldsymbol{v}_h)
      \Bigr)\,.
\end{align*}
Eventually, for $\alpha>0$ sufficiently large, we obtain the desired best-approximation result.
\end{proof}

As a first immediate consequence of the best-approximation result in \cref{thm:error_estimate_mixed}, we obtain the convergence of the method under minimal regularity assumptions, i.e., merely $\boldsymbol{u} \in V$, $\pmb{\mathsf{S}}\in \Sigma$, and $\boldsymbol{f}\in V^*$.

\begin{Corollary}[Convergence]\label{cor:mixed_convergence}
For $\alpha >0$ sufficiently large, it holds that
\begin{align*}
	\|\pmb{\mathsf{\mathcal{F}}}(\tilde{\nabla}_h \boldsymbol{u}_h) - \pmb{\mathsf{\mathcal{F}}}(\nabla \boldsymbol{u})\|^2_{2,\Omega}
	+ 	\|\pmb{\mathsf{\mathcal{F}}}^*(\pmb{\mathsf{S}}_h) - \pmb{\mathsf{\mathcal{F}}}^*(\pmb{\mathsf{S}})\|_{2,\Omega}^2+m_{\varphi_{\boldsymbol{\beta}_h(\boldsymbol{u}_h)},h}(\boldsymbol{u}_h)\to 0\quad (h\to 0)\,.
\end{align*}
\end{Corollary}

\begin{proof}
Using the stability and approximation properties of
$\Pi_h^{k-1}$, and the density of smooth functions, we obtain
$\Pi_h^{k-1}\pmb{\mathsf{S}}\to \pmb{\mathsf{S}}$ in $L^{p'}(\Omega)^{n\times d}$ $(h\to
0)$, which implies that
\begin{align}\label{cor:mixed_convergence.1}
	\|\pmb{\mathsf{\mathcal{F}}}^*(\pmb{\mathsf{S}}) -
	\pmb{\mathsf{\mathcal{F}}}^*(\Pi_h^{k-1}\pmb{\mathsf{S}})\|^2_{2,\Omega} \lesssim \|\pmb{\mathsf{S}} - \Pi_h^{k-1} \pmb{\mathsf{S}}\|_{p'}^{\min\{p',2\}}\to 0\quad(h\to 0)
\end{align} 
	with a constant depending possibly on $\delta$,
                $\|\pmb{\mathsf{S}}\|_{p'}$. Therefore, choosing $(\pmb{\mathsf{T}}_h,\boldsymbol{v}_h)=(\Pi_h^{k-1}\pmb{\mathsf{S}},\Pi_h^{SZ}\boldsymbol{u})\in \Sigma_h\times V_h$ in \cref{thm:error_estimate_mixed}, using \cref{cor:mixed_convergence.1,cor:primal_convergence.1,cor:primal_convergence.2},
we conclude that
\begin{align*}
    \SwapAboveDisplaySkip
	&\|\pmb{\mathsf{\mathcal{F}}}(\tilde{\nabla}_h \boldsymbol{u}_h) - \pmb{\mathsf{\mathcal{F}}}(\nabla \boldsymbol{u})\|^2_{2,\Omega}
	+ 	\|\pmb{\mathsf{\mathcal{F}}}^*(\pmb{\mathsf{S}}_h) - \pmb{\mathsf{\mathcal{F}}}^*(\pmb{\mathsf{S}})\|_{2,\Omega}^2+m_{\varphi_{\boldsymbol{\beta}_h(\boldsymbol{u}_h)},h}(\boldsymbol{u}_h)\\&\lesssim \|\pmb{\mathsf{\mathcal{F}}}(\nabla \boldsymbol{u}) - \pmb{\mathsf{\mathcal{F}}}(\nabla\Pi_h^{SZ} \boldsymbol{u})\|^2_{2,\Omega}
	+\|\pmb{\mathsf{\mathcal{F}}}^*(\pmb{\mathsf{S}}) -
	\pmb{\mathsf{\mathcal{F}}}^*(\Pi_h^{k-1}\pmb{\mathsf{S}})\|^2_{2,\Omega}\to 0\quad(h\to 0)\,,
\end{align*}
which is the claimed convergence under minimal regularity assumptions.
\end{proof}
\begin{Corollary}[Fractional convergence rates]\label{cor:mixed_rate}
		Assume that the family of triangulations $\{\mathcal{T}_h\}_{h}$ is quasi-uniform, and that
 $\pmb{\mathsf{\mathcal{F}}}(\nabla \boldsymbol{u})\in N^{\beta,2}(\Omega)^{n\times d}$ for some $\beta\in (0,1]$.
 Then, for $\alpha >0$ sufficiently large, it holds that
\begin{align*}
	&\|\pmb{\mathsf{\mathcal{F}}}(\tilde{\nabla}_h \boldsymbol{u}_h) - \pmb{\mathsf{\mathcal{F}}}(\nabla \boldsymbol{u})\|^2_{2,\Omega}	+ 	\|\pmb{\mathsf{\mathcal{F}}}^*(\pmb{\mathsf{S}}_h) - \pmb{\mathsf{\mathcal{F}}}^*(\pmb{\mathsf{S}})\|_{2,\Omega}^2
	+ m_{\varphi_{\boldsymbol{\beta}_h(\boldsymbol{u}_h)},h}(\boldsymbol{u}_h) \\
    &\lesssim h^{2\beta}\, [\pmb{\mathsf{\mathcal{F}}}(\nabla \boldsymbol{u})]_{N^{\beta,2}(\Omega)}^2\,.
\end{align*}
\end{Corollary}

\begin{proof}
First, we note that $\pmb{\mathsf{\mathcal{F}}}(\nabla \boldsymbol{u})\in N^{\beta,2}(\Omega)^{n\times d}$ for $\beta\in (0,1]$ is equivalent to $\pmb{\mathsf{\mathcal{F}}}^*(\pmb{\mathsf{S}})\in N^{\beta,2}(\Omega)^{n\times d}$ for  $\beta\in (0,1]$ and that $ [\pmb{\mathsf{\mathcal{F}}}(\nabla \boldsymbol{u})]_{N^{\beta,2}(\Omega)}\sim  [\pmb{\mathsf{\mathcal{F}}}^*(\pmb{\mathsf{S}})]_{N^{\beta,2}(\Omega)}$.
This is an immediate consequence of the fact that, due to \cref{eq:F-F*3}, for every $h\in \mathbb{R}^d\setminus \{\boldsymbol{0}\}$ and $x\in \Omega\cap (\Omega-h)$, we have that
\begin{align*}
	\vert \pmb{\mathsf{\mathcal{F}}}(\nabla \boldsymbol{u}(x+h))- \pmb{\mathsf{\mathcal{F}}}(\nabla \boldsymbol{u}(x))\vert^2\sim \vert \pmb{\mathsf{\mathcal{F}}}^*(\pmb{\mathsf{S}}(x+h))- \pmb{\mathsf{\mathcal{F}}}^*(\pmb{\mathsf{S}}(x))\vert^2\,.
\end{align*}
Using that $\Pi_h^0\pmb{\mathsf{S}}=\Pi_h^{k-1}\Pi_h^0\pmb{\mathsf{S}}$,
\cref{eq:hammerf}, the Orlicz-stability of
$\Pi_h^{k-1}$
(cf.~\cite[Cor.~A.8]{kr-phi-ldg}),
\cref{eq:hammerf}, \cite[Lem.~4.4]{dkrt-ldg},
and \cite[(4.6), (4.7)]{breit-lars-etal}, we obtain
\begin{align}\label{cor:mixed_rate.1}
		\|  \pmb{\mathsf{\mathcal{F}}}^*(\pmb{\mathsf{S}})- \pmb{\mathsf{\mathcal{F}}}^*(\Pi_h^{k-1}\pmb{\mathsf{S}})\|_{2,\Omega}^2&\lesssim \|  \pmb{\mathsf{\mathcal{F}}}^*(\pmb{\mathsf{S}})-  \pmb{\mathsf{\mathcal{F}}}^*(\Pi_h^0\pmb{\mathsf{S}})\|_{2,\Omega}^2+ \|  \pmb{\mathsf{\mathcal{F}}}^*(\Pi_h^0\pmb{\mathsf{S}})-  \pmb{\mathsf{\mathcal{F}}}^*(\Pi_h^{k-1}\pmb{\mathsf{S}})\|_{2,\Omega}^2\notag
		\\&	\lesssim \|  \pmb{\mathsf{\mathcal{F}}}^*(\pmb{\mathsf{S}})-  \pmb{\mathsf{\mathcal{F}}}^*(\Pi_h^0\pmb{\mathsf{S}})\|_{2,\Omega}^2+ \rho_{\varphi^*_{\vert\Pi_h^0 \pmb{\mathsf{S}}\vert}}(\Pi_h^{k-1}(\Pi_h^0\pmb{\mathsf{S}}-\pmb{\mathsf{S}}))\notag
		\\&	\lesssim \|  \pmb{\mathsf{\mathcal{F}}}^*(\pmb{\mathsf{S}})-  \pmb{\mathsf{\mathcal{F}}}^*(\Pi_h^0\pmb{\mathsf{S}})\|_{2,\Omega}^2+ \rho_{\varphi^*_{\vert\Pi_h^0 \pmb{\mathsf{S}}\vert}}(\Pi_h^0\pmb{\mathsf{S}}-\pmb{\mathsf{S}})\notag
		\\&	\lesssim \|  \pmb{\mathsf{\mathcal{F}}}^*(\pmb{\mathsf{S}})-  \pmb{\mathsf{\mathcal{F}}}^*(\Pi_h^0\pmb{\mathsf{S}})\|_{2,\Omega}^2
		\\&	\lesssim \|  \pmb{\mathsf{\mathcal{F}}}^*(\pmb{\mathsf{S}})-  \Pi_h^0\pmb{\mathsf{\mathcal{F}}}^*(\pmb{\mathsf{S}})\|_{2,\Omega}^2\notag
		\\&	\lesssim h^{2\beta}\,[\pmb{\mathsf{\mathcal{F}}}^*(\pmb{\mathsf{S}})]_{N^{\beta,2}(\Omega)}^2\notag
		\\&	\sim h^{2\beta}\,[\pmb{\mathsf{\mathcal{F}}}(\nabla\boldsymbol{u})]_{N^{\beta,2}(\Omega)}^2\,.\notag
\end{align}
Therefore, choosing $(\pmb{\mathsf{T}}_h,\boldsymbol{v}_h)=(\Pi_h^{k-1}\pmb{\mathsf{S}},\Pi_h^{SZ}\boldsymbol{u})\in \Sigma_h\times V_h$ in \cref{thm:error_estimate_mixed}, using \cref{cor:mixed_rate.1,cor:primal_convergence.1,cor:primal_rate.1},
we conclude 		  that
\begin{align*}
	&\|\pmb{\mathsf{\mathcal{F}}}(\tilde{\nabla}_h \boldsymbol{u}_h) - \pmb{\mathsf{\mathcal{F}}}(\nabla \boldsymbol{u})\|^2_{2,\Omega}
	+ 	\|\pmb{\mathsf{\mathcal{F}}}^*(\pmb{\mathsf{S}}_h) - \pmb{\mathsf{\mathcal{F}}}^*(\pmb{\mathsf{S}})\|_{2,\Omega}^2+m_{\varphi_{\boldsymbol{\beta}_h(\boldsymbol{u}_h)},h}(\boldsymbol{u}_h)\\&\lesssim \|\pmb{\mathsf{\mathcal{F}}}(\nabla \boldsymbol{u}) - \pmb{\mathsf{\mathcal{F}}}(\nabla\Pi_h^{SZ} \boldsymbol{u})\|^2_{2,\Omega}
	+\|\pmb{\mathsf{\mathcal{F}}}^*(\pmb{\mathsf{S}}) -
	\pmb{\mathsf{\mathcal{F}}}^*(\Pi_h^{k-1}\pmb{\mathsf{S}})\|^2_{2,\Omega}\\&\lesssim h^{2\beta}\,[\pmb{\mathsf{\mathcal{F}}}(\nabla\boldsymbol{u})]_{N^{\beta,2}(\Omega)}^2 \,,
\end{align*}
which is the claimed fractional a priori error estimate.
\end{proof}
        \begin{Remark}
          In view of \cite[Corollary 5.8]{dr-interpol} the assertion
          of \cref{cor:mixed_rate} for $\beta=1$ is also valid if $\pmb{\mathsf{\mathcal{F}}}(\nabla \boldsymbol{u})\in
          W^{1,2}(\Omega)^{n\times d}$ without the additional
          assumption that the triangulation $\{\mathcal T_h\}_h$ is quasi-uniform. 
        \end{Remark}
\begin{Corollary}[Ansatz class competition]\label{cor:mixed_competition} Let $k=1$ and $\boldsymbol{u}_h^c\in V_{h,c}\coloneqq V_h\cap W^{1,p}_0(\Omega)$ the continuous Lagrange solution of \cref{eq:PDE}, cf.\ \cref{eq:continuous_lagrange}.
Then, for $\alpha>0$ sufficiently large,  it holds that
\begin{align*}
	&\|\pmb{\mathsf{\mathcal{F}}}(\tilde{\nabla}_h \boldsymbol{u}_h) - \pmb{\mathsf{\mathcal{F}}}(\nabla \boldsymbol{u})\|^2_{2,\Omega}+\|\pmb{\mathsf{\mathcal{F}}}^*(\pmb{\mathsf{S}}_h) - \pmb{\mathsf{\mathcal{F}}}^*(\pmb{\mathsf{S}})\|_{2,\Omega}^2+ m_{\varphi_{\boldsymbol{\beta}_h(\boldsymbol{u}_h)},h}(\boldsymbol{u}_h) \\
    &\sim 	\|\pmb{\mathsf{\mathcal{F}}}(\nabla \boldsymbol{u}_h^c) - \pmb{\mathsf{\mathcal{F}}}(\nabla \boldsymbol{u})\|^2_{2,\Omega}\,,
\end{align*}
i.e., the approximation capabilities of the discrete mixed formulation \cref{eq:discrete_mixed_DG} and the continuous Lagrange approximation \cref{eq:continuous_lagrange} of \cref{eq:PDE} are comparable.
\end{Corollary}

\begin{proof}
\textit{ad $\lesssim$.} Using \cref{thm:error_estimate_mixed} with $(\pmb{\mathsf{T}}_h,\boldsymbol{v}_h)=(\pmb{\mathsf{S}}(\nabla \boldsymbol{u}_h^c),\boldsymbol{u}_h^c)\in  \Sigma_h\times V_{h,c}\subseteq  \Sigma_h\times V_h$,  that $m_{\varphi_{\boldsymbol{\beta}_h(\boldsymbol{u}_h)},h}(\boldsymbol{u}_h^c)=0$, and \cref{eq:hammera}, we find that
\begin{multline*}
	\|\pmb{\mathsf{\mathcal{F}}}(\tilde{\nabla}_h \boldsymbol{u}_h) - \pmb{\mathsf{\mathcal{F}}}(\nabla \boldsymbol{u})\|^2_{2,\Omega}+\|\pmb{\mathsf{\mathcal{F}}}^*(\pmb{\mathsf{S}}_h) - \pmb{\mathsf{\mathcal{F}}}^*(\pmb{\mathsf{S}})\|_{2,\Omega}^2+ m_{\varphi_{\boldsymbol{\beta}_h(\boldsymbol{u}_h)},h}(\boldsymbol{u}_h)\\
    \begin{aligned}
      &\lesssim 	\|\pmb{\mathsf{\mathcal{F}}}(\nabla \boldsymbol{u}_h^c) - \pmb{\mathsf{\mathcal{F}}}(\nabla \boldsymbol{u})\|^2_{2,\Omega}+\|\pmb{\mathsf{\mathcal{F}}}^*(\pmb{\mathsf{S}}(\nabla \boldsymbol{u}_h^c)) - \pmb{\mathsf{\mathcal{F}}}^*(\pmb{\mathsf{S}})\|_{2,\Omega}^2\\
      &\lesssim 	\|\pmb{\mathsf{\mathcal{F}}}(\nabla \boldsymbol{u}_h^c) - \pmb{\mathsf{\mathcal{F}}}(\nabla \boldsymbol{u})\|^2_{2,\Omega}\,.
    \end{aligned}
\end{multline*}

\textit{ad $\gtrsim$.} This is proved in \cref{eq:primal_competition.2}.

Putting everything together, we arrive at the claimed equivalence. 
	\end{proof}
	
	\section{Numerical experiments}\label{sec:experiments}
	In this section, we show numerical results that confirm our
          theoretical findings, in particular,
          \cref{cor:primal_rate,cor:mixed_rate}. In our
          implementation, the max-shift in the jump penalisation is
          handled through a fixed point iteration, i.e., when solving
          the discrete primal
          formulation~\cref{eq:discrete_primal_DG}~starting from a
          solution guess $\smash{\boldsymbol{u}}_h^{(k-1)}\in V_h$, we define the
          residual $\mathfrak{F}(\boldsymbol{u}_h^{(k-1)};\,\cdot)\in (V_h)^*$,
          for every $\boldsymbol{v}_h,\boldsymbol{z}_h\in V_h$, via 
	    \begin{align*}
			\langle	\mathfrak{F}(\boldsymbol{u}_h^{k-1};\boldsymbol{v}_h) , \boldsymbol{z}_h \rangle_{V_h}
        &\coloneqq 
       (\pmb{\mathsf{\mathcal{S}}}(\tilde{\nabla}_h \boldsymbol{v}_h) ,\nabla \pmb{\mathsf{\mathcal{E}}}_h \boldsymbol{z}_h)_\Omega
        \\&\quad+
        \alpha \langle \pmb{\mathsf{\mathcal{S}}}_{\beta_h(\boldsymbol{u}^{k-1}_h)}(h_\Gamma^{-1} \llbracket{ \boldsymbol{v}_h\otimes \bm{n}}\rrbracket ), \llbracket{\boldsymbol{z}_h\otimes \bm{n}}\rrbracket\rangle_{\Gamma_h} 
        -
        \langle \bm{f} , \pmb{\mathsf{\mathcal{E}}}_h\boldsymbol{v}_h \rangle_{V}\,,
    \end{align*}
    and, then, apply Newton's method to find the next guess $\boldsymbol{u}_h^{(k)}\in
    V_h$. At a given Newton step, the linear systems are solved using the
    sparse direct solver \texttt{MUMPS} \cite{MUMPS:1}. All the examples were
    implemented using \texttt{Firedrake} \cite{Firedrake} and \texttt{PETSc} \cite{petsc-user-ref}.
    The complete code for reproducing the experiments can be found at
    \cite{ble-gaz-2023} with exact version of its dependencies being
    recorded at \cite{zenodo/Firedrake-20231101.0}, additionally using Gmsh
    version~4.8.4 \cite{Gmsh2009}.

    Only polynomial degree $k=1$ is considered. Note that, in this case, the
    smoothing operator only needs to be applied on the forcing term,
          since    $\pmb{\mathsf{\mathcal{S}}}(\tilde{\nabla}_h \boldsymbol{u}_h) \in \mathbb{P}^0(\mathcal{T}_h)^{n\times
    d}$. In addition, in all the examples, we restrict to the case $d=2$ and
    $n=1$.

		\subsection{Primal formulation}
		For the experiments based on the discrete primal
          formulation \cref{eq:discrete_primal_DG}, we employ the
          non-linear term $\pmb{\mathsf{\mathcal{S}}}\colon \mathbb{R}^{1\times 2}\to
          \mathbb{R}^{1\times 2}$, for every $\pmb{\mathsf{D}}\in
          \mathbb{R}^{1\times 2}$, defined via
		\begin{align}
          \SwapAboveDisplaySkip
          \label{eq:nonlinearity_primal}
          \pmb{\mathsf{\mathcal{S}}}(\pmb{\mathsf{D}}) \coloneqq (\delta + \vert\pmb{\mathsf{D}}\vert)^{p-2} \,\pmb{\mathsf{D}},
		\end{align}
        which has $(p,\delta)$-structure according to \cref{assum:extra_stress}.
        We choose $\delta \coloneqq 0.01$ and various values of $p\in [1.5,4.5]$.
        The shifted constitutive relation~$\pmb{\mathsf{\mathcal{S}}}_a$ in the
        jump penalty term in~\cref{eq:discrete_primal_DG} is then
        \cref{eq:flux} and $\alpha=10$ is chosen.
        The computational domain is defined via $\Omega\coloneqq (-1,1)^2$ and, for $\beta>1-\frac{2}{p}$, the exact solution $\boldsymbol{u}\in V$ with a point singularity~at~the~origin, for every $x\coloneqq (x_1,x_2)^\top\in \Omega$, is defined via 
		\begin{equation*}
\boldsymbol{u}(x) = (1-x_1^2)(1-x_2^2)\vert x\vert ^\beta \,.
		\end{equation*}
		Note that $\beta> 1-\frac{2}{p}$ guarantees that at least $\boldsymbol{u} \in V$.
        \begin{figure}
          \centering
          \includegraphics[trim=15 15 15 15,clip,scale=0.47]{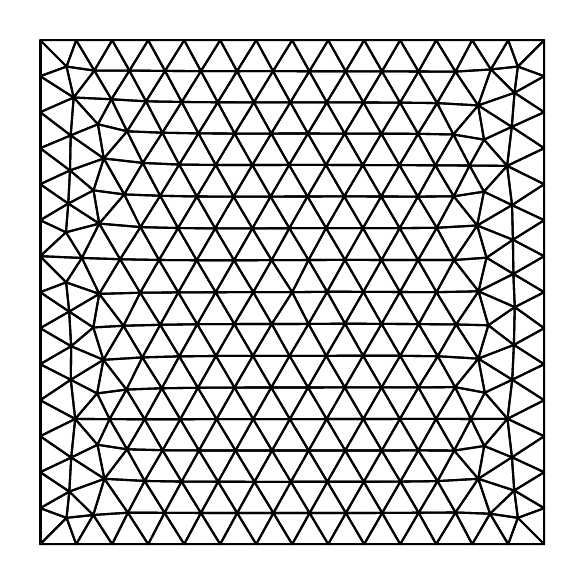}
          \caption{Initial mesh $\mathcal{T}_{h_0}$}
          \label{fig:square2.msh}
        \end{figure}
        We start with an initial unstructured mesh $\mathcal{T}_{h_0}$ with 517 elements, 257 vertices, of which one is at the origin, and $h_0\approx0.1668$; see \cref{fig:square2.msh}. We consider six additional levels of uniform refinement, i.e., $\mathcal{T}_{h_l}$, $l\in \{1,\ldots,6\}$, where $h_l=\frac{h_{l-1}}{2}$ for all  $l\in \{1,\ldots,6\}$. %In the implementation we set $\tilde{\nabla}_h \coloneqq  \nabla_h$ in the discrete formulation \cref{eq:discrete_primal_DG}; i.e.\ we employ the Incomplete Interior Penalty Method (IIDG). 
		The error corresponding to the discrete solution $\boldsymbol{u}_{h_\ell}\in V_{h_\ell}$ of \cref{eq:discrete_primal_DG}, associated to a given refinement level $\ell \in \{0,\ldots, 6\}$, is defined via
		\begin{equation*}
e_\ell \coloneqq  \|\pmb{\mathsf{\mathcal{F}}}(\nabla_{h_\ell} \boldsymbol{u}_{h_\ell}) - \pmb{\mathsf{\mathcal{F}}}(\nabla \boldsymbol{u})\|_{2,\Omega}
+
\bigl(\alpha \, m_{\varphi_{\beta_h}(\boldsymbol{u}_{h_\ell}),h}(\boldsymbol{u}_{h_\ell})\bigr)^{\frac{1}{2}}.
		\end{equation*}
		The experimental rate of convergence is then set to
		\begin{equation}\label{eq:EOC}
\mathrm{EOC}_\ell \coloneqq  \frac{\log(e_\ell/e_{\ell-1})}{\log(h_\ell/h_{\ell -1})}.
		\end{equation}
		An important observation is that $\beta>0$ determines the regularity of the exact solution~and, thus, also the expected rate of convergence. More precisely, to obtain a rate of convergence $\rho  \in (0,1]$, one needs to choose $\beta > 1 - \frac{2(1-\rho)}{p}$. \Cref{tbl:DG_rate_1.0,tbl:DG_rate_0.5,tbl:DG_rate_0.2} show the results for the IIDG formulation for expected rates of convergence of $1$, $0.5$, and $0.2$, respectively; \cref{tbl:LDG_rate_1.0,tbl:LDG_rate_0.5,tbl:LDG_rate_0.2} show the same for the LDG formulation. It can be observed that the values are in agreement with the theoretical predictions.

\TblDgRateOne
\TblLdgRateOne
\TblDgRateHalf
\TblLdgRateHalf
\TblDgRateFifth
\TblLdgRateFifth

 As mentioned in \cref{rmk:CR}, our results do not cover the case of a~Crouzeix--Raviart discretisation without jump stabilisation terms. However, as seen in \cref{tbl:CR_rate_1.0,tbl:CR_rate_0.5,tbl:CR_rate_0.2}, the rates are roughly in agreement with the same rates as the DG discretisation, suggesting that there might be a~proof strategy that also covers this case.

\TblCrRateOne
\TblCrRateHalf
\TblCrRateFifth

 \subsection{Mixed formulation}
    For the experiments based on the discrete mixed formulation
    \cref{eq:discrete_mixed_DG}, for which now the following nonlinear term is
    employed:
 \begin{equation*}
     \pmb{\mathsf{\mathcal{D}}}(\pmb{\mathsf{S}}) \coloneqq  (\delta^{2(p-1)} + |\pmb{\mathsf{S}}|^2)^{\frac{p'-2}{2}} \,\pmb{\mathsf{S}},
 \end{equation*}
which has $(p',\delta^{p-1})$-structure in the sense of \cref{assum:extra_stress}.
 Note that this relation is the inverse of \cref{eq:nonlinearity_primal} when $\delta = 0$. In the experiments, we set $\delta = 0.01$. The jump penalty term in~\cref{eq:discrete_mixed_DG} is again defined using~\eqref{eq:flux} and~$\alpha=10$.
 The computational domain, once again, $\Omega = (-1,1)^2$ and we choose the exact flux $\pmb{\mathsf{S}}\in \Sigma$ as 
  \begin{align*}
    \SwapAboveDisplaySkip
     \pmb{\mathsf{S}}  \coloneqq  (\delta^2+|\nabla \tilde{\boldsymbol{u}}|^2)^{\frac{p-2}{2}} \nabla\tilde{\boldsymbol{u}},
	 \qquad
	 \tilde{\boldsymbol{u}}(x) = (1-x_1^2)(1-x_2^2)|x|^\beta,
 \end{align*}
 where $p>1$ is specified beforehand, and $\beta\in \ensuremath{\mathbb{R}}$. The gradient of the exact solution can, then, be computed as $\nabla \boldsymbol{u}\coloneqq  \pmb{\mathsf{\mathcal{D}}}(\pmb{\mathsf{S}} )$. In the implementation, we set $\tilde{\nabla}_h = \mathcal{G}_h$ in the discrete formulation~\cref{eq:discrete_mixed_DG}, which corresponds to an LDG method. For the mixed formulation, we define the error corresponding to a solution $(\pmb{\mathsf{S}}_{h_\ell},\boldsymbol{u}_{h_\ell})\in \Sigma_{h_\ell}\times V_{h_\ell}$ associated to a refinement level $\ell\in \{0,\ldots, 6\}$ as:
 		\begin{align*}
e_\ell \coloneqq 
\|\pmb{\mathsf{\mathcal{F}}}^*(\pmb{\mathsf{S}}_{h_\ell}) - \pmb{\mathsf{\mathcal{F}}}^*(\pmb{\mathsf{S}} )\|_{2,\Omega}
+
\|\pmb{\mathsf{\mathcal{F}}}(\nabla_{h_\ell} \boldsymbol{u}_{h_\ell}) - \pmb{\mathsf{\mathcal{F}}}(\nabla \boldsymbol{u})\|_{2,\Omega}
+
\bigl(\alpha \, m_{\varphi_{\beta_h}(\boldsymbol{u}_{h_\ell}),h}(\boldsymbol{u}_{h_\ell})\bigr)^{\frac{1}{2}}.
		\end{align*}
		The experimental order of convergence is then defined analogously to \cref{eq:EOC}. The results in this case can be found in \cref{tbl:mixed_DG_rate_1.0,tbl:mixed_DG_rate_0.5,tbl:mixed_DG_rate_0.2}.

\TblMixedDgRateOne
\TblMixedDgRateHalf
\TblMixedDgRateFifth

\begingroup
\raggedbottom
\interlinepenalty=10000
\bibliographystyle{siamplain}

\endgroup

\end{document}